\documentclass{amsart}
\usepackage{amssymb}
\usepackage{mathrsfs}
\usepackage{stmaryrd}
\usepackage{bbm}
\usepackage{oldgerm}
\usepackage[english]{babel}
\usepackage[T1]{fontenc}
\usepackage[utf8x]{inputenc}
\usepackage[all]{xy}
\usepackage[colorlinks]{hyperref}
\usepackage{upref}
\usepackage{xcolor}
\usepackage{scrextend}
\usepackage{tikz-cd}

\colorlet{wine-stain}{red!50!black}
\colorlet{light-blue}{cyan!60!black}
\hypersetup{
					colorlinks=true, 
					linkcolor=wine-stain, 
					citecolor=cyan, 
					urlcolor=green, 
					linktoc=page}

\newtheorem{thm}[subsection]{Theorem}
\newtheorem{prop}[subsection]{Proposition}
\newtheorem{cor}[subsection]{Corollary}
\newtheorem{lem}[subsection]{Lemma}

\theoremstyle{definition}

\newtheorem{defi}[subsection]{Definition}
\newtheorem{rem}[subsection]{Remark}
\newtheorem{rems}[subsection]{Remarks}

\numberwithin{equation}{subsection}
\def\bZ{\mathbb{Z}}

\def\d{\mathfrak{d}}
\def\m{\mathfrak{m}}

\def\p{\mathfrak{p}}
\def\q{\mathfrak{q}}

\DeclareMathOperator{\Spec}{Spec}
\DeclareMathOperator{\Spf}{Spf}
\DeclareMathOperator{\Sp}{Sp}
\DeclareMathOperator{\Hom}{Hom}
\DeclareMathOperator{\Gal}{Gal}

\DeclareMathOperator{\sw}{sw}
\DeclareMathOperator{\Ind}{Ind}
\DeclareMathOperator{\ord}{ord}

\def\Z{\mathbb{Z}}
\def\Q{\mathbb{Q}}

\def\O{\mathcal{O}}
\def\D{\mathcal{D}}
\def\F{\mathcal{F}}
\def\L{\mathbb{L}}
\def\K{\mathbb{K}}

\def\m{\mathfrak{m}}

\def\p{\mathfrak{p}}
\def\q{\mathfrak{q}}

\def\d{\mathfrak{d}}

\title{Variation of the Swan conductor of an $\mathbb{F}_{\ell}$-sheaf on a rigid disc}
\author{Amadou Bah}
\address{Université Paris-Saclay, IH\'ES, CNRS, Laboratoire Alexander Grothendieck, 91440, Bures-sur-Yvette, France}
\email{a.bah@ihes.fr}

\date{}

\begin{document}

\begin{abstract}
This article studies the variation of the Swan conductor of a lisse étale sheaf of $\mathbb{F}_{\ell}$-modules $\F$ on the rigid unit disc $D$ over a complete discrete valuation field $K$ with algebraically closed residue field of characteristic $p\neq \ell$. We associate to $\F$ a function $\sw_{\rm AS}(\F, \cdot): \Q_{\geq 0}\to \Q$, defined with the Abbes-Saito logarithmic ramification filtration, which measures, at each $t\in \Q_{\geq 0}$, the ramification of the restriction of $\F$ to the subdisc of radius $t$ along the special fiber of the normalized integral model. We prove that this function is continuous and piecewise linear, with finitely many slopes which are all integers. We compute the slope at $t\in \Q_{\geq 0}$ in terms of a characteristic cycle associated to $\F$, a (power of a) logarithmic differential form defined by ramification theory.
\end{abstract}

\maketitle

\tableofcontents

\section{Introduction.}\label{Intro}
\subsection{}\label{NotationIntro}
Let $\O_K$ be a henselian discrete valuation ring, $K$ its field of fractions, $\m_K$ its maximal ideal, $k$ its residue field of characteristic $p>0$, and $\pi$ a uniformizer of $\O_K$. Let also $\overline{K}$ be a separable closure of $K$, $\O_{\overline{K}}$ the integral closure of $\O_K$ in $\overline{K}$, $\overline{k}$ its residue field, $G_K$ the Galois group of $\overline{K}$ over $K$, and $v:\overline{K}^{\times}\to \Q$ the valuation of $\overline{K}$ normalized by $v(\pi)=1$.

\subsection{}\label{k parfait}
When $k$ is perfect, the classical ramification theory \cite[IV-VI]{Serre1} gives a filtration of $G_K$ by closed normal subgroups indexed by $\Q_{\geq 0}$ and studies the action of theses subgroups on representations of $G_K$, producing numerical measures such as the Artin and Swan conductors.

\subsection{} \label{ModernRamifTheory}
The geometric study of ramification theory was initiated by Grothendieck. The setting is that of a variety $X$ over a perfect base field and, for $\ell$ a prime different from the characteristic of the base field, a constructible $\ell$-adic sheaf $\F$ on $X$ which is lisse on a non-empty open subset $U$.
One wishes first to construct local invariants, such as Swan conductors, to measure the ramification of $\F$ at the points of $X\backslash U$ and, second, one wants to use these invariants to produce an index formula computing the Euler-Poincaré characteristic $\chi(X, \F)$ of $\F$. 

When $X$ is a smooth, projective and geometrically connected curve of genus $g$ and ${\rm rk}(\F)$ denotes the dimension of the stalk of $\F$ at a geometric generic point of $X$, the wish formulated above is achieved by the Grothendieck-Ogg-Shafarevich formula. The latter computes a global invariant, $\chi(X,\F) - (2-2g){\rm rk}(\F)$ in terms of finite local data at the points of $X\backslash U$, the Swan conductors of $\F$ at these points.

\subsection{} \label{DimensionSup}
The generalization of this index formula to higher dimensions was a driving force in much subsequent works in the field of ramification theory. In this direction, following suggestions by Deligne, Laumon treated in his thesis \cite{Laumon2} the case of connected, normal and projective surfaces over an algebraically closed fied.
Deligne and Laumon also proved a formula for the dimension of the space of vanishing cycles of a relative curve in terms of Swan conductors \cite[Theorem 5.1.1]{Laumon1}, which was later refined by Kato \cite[Theorem 6.7]{K1}, and deduced from it the lower semi-continuity of Swan conductors \cite[Theorem 2.1.1]{Laumon1}.
In the quest for higher dimensional invariants and a higher dimensional index formula, the analogy between the theory of $D$-modules and the theory of $\ell$-adic sheaves was a guiding principle and remains so today. To any holonomic $D$-module $\mathcal{M}$ on a complex analytic variety, one can associate a \textit{characteristic variety} ${\rm Char}(\mathcal{M})$, which is a closed conical subset of the cotangent bundle of the variety, and a \textit{characteristic cycle} ${\rm CC}(\mathcal{M})$, which is a linear combination of the connected components of ${\rm Char}(\mathcal{M})$. Then, the Euler-Poincaré characteristic of $\mathcal{M}$ was shown, by Dubson \cite{Dubson} and Kashiwara \cite{Kashiwara} to be the intersection of ${\rm CC}(\mathcal{M})$ with the zero-section of the cotangent bundle of the variety. 
Recently, exploiting the aforementioned analogy, Beilinson constructed the singular support ${\rm SS}(\F)$ of a constructible $\ell$-adic sheaf $\F$ \cite{Beilinson}.
Building on this crucial result, on earlier work of K. Kato and on his joint work with A. Abbes (see below), T. Saito \cite{Saito3} constructed the characteristic cycle ${\rm CC}(\F)$ in arbitrary dimension and established the generalization of the index formula.

\subsection{} \label{mperfect residue field}
In the higher dimensional setting, one has to deal with henselian discrete valuation rings with imperfect residue fields. Progress in the study of a ramification theory of $K$ that allows its residue field $k$ to be imperfect contributed to the aforementioned generalizations. In the eighties, K. Kato initiated such a study for rank one characters \cite{K2, K1}. In the 2000's, A. Abbes and T. Saito, through geometric methods, produced a compelling ramification theory that accommodates an imperfect residue field \cite{A.S.1, A.S.2, A.S.3}. More precisely, they defined a decreasing filtration $(G_{K, \log}^r)_{r\in \Q_{\geq 0}}$ of $G_K$ by closed normal subgroups, the logarithmic ramification filtration, which coincides with the classical ramification filtration when $k$ is perfect and is such that, for $r\in\Q_{\geq 0}$, if we put 
\begin{equation}
	\label{imperfect residue field 1}
G_{K, \log}^{r+}=\overline{\cup_{s>r}G_{K, \log}^s},
\end{equation}
\begin{equation}
	\label{imperfect residue field 2}
{\rm Gr}_{\log}^r G_K= G_{K, \log}^r/G_{K, \log}^{r+},
\end{equation}
then $G_{K, \log}^0=I_K$ is the inertia subgroup of $G_K$ and $G_{K, \log}^{0+}$ coincides with the wild inertia subgroup $P_K$, the unique $p$-Sylow subgroup of $I_K$. Moreover, this filtration behaves well under a tame extension of $K$\cite[3.15]{A.S.1}. As in the classical setting, the graded pieces ${\rm Gr}_{\log}^r G_K$ are abelian and killed by $p$ (\cite[1.24]{Saito1}, \cite[Theorem 2]{Saito2} and \cite[Theorem 4.3.1]{Saito4}).

\subsection{}\label{TheRefinedSwanConductor}
For $r\in \Q$, we let $\m_{\overline{K}}^r$ (resp. $\m_{\overline{K}}^{r+}$) be the set of elements $x$ of $\overline{K}$ satisfying $v(x)\geq r$ (resp. $v(x)>r$). Assuming that $k$ is of finite type over a perfect sub-field $k_0$, we let $\Omega_k^1(\log)$ be the $k$-vector space of logarithmic differential $1$-forms
\begin{equation}
	\label{TheRefinedSwanConductor 1}
\Omega_k^1(\log)=(\Omega_{k/k_0}^1 \oplus (k\otimes_{\Z}K^{\times}))/({\rm d}\overline{a}-\overline{a}\otimes a, a\in \O_K^{\times}).
\end{equation}
Generalizing a construction of Kato for characters of $G_K$ of degree one \cite[Theorem 0.1]{K3} and previous work with Abbes \cite[\S 9]{A.S.4}, T. Saito (\cite[1.24]{Saito1}, \cite[Theorem 2]{Saito2}, \cite[Theorem 6.13]{A.S.3}), shows that there is an injective homomorphism, the \textit{refined Swan conductor}
\begin{equation}
	\label{TheRefinedSwanConductor 2}
{\rm rsw}: \Hom({\rm Gr}^r_{\log} G_K, \mathbb{F}_p)\to \Hom_{\overline{k}}(\m^r_{\overline{K}}/\m^{r+}_{\overline{K}}, \Omega^1_k(\log)\otimes_k\overline{k}).
\end{equation}

\subsection{}\label{Decompositions}
Let $\Lambda$ be a finite field of characteristic $\ell\neq p$ and fix a nontrivial character $\psi: \mathbb{F}_p \to \Lambda^{\times}$. Let $L \subset \overline{K}$ be a finite Galois extension of $K$ of group $G$. Let $M$ be a finite dimensional $\Lambda$-vector space with a linear action of $G$. Then, the filtration $(G_{K, \log}^r)_{r\in \Q_{\geq 0}}$ induces a canonical \textit{slope decomposition} of $M$ into $P_K$-stable sub-modules (\ref{SlopeDecomposition})
\begin{equation}
	\label{Decompositions 1}
M= \bigoplus_{r\in \Q_{\geq 0}} M^{(r)}, 
\end{equation}
where $M^{(0)}=M^{P_K}$. The Swan conductor of $M$ is defined as
\begin{equation}
	\label{Decompositions 2}
\sw_G^{\rm AS}(M)=\sum_{r\in \Q_{\geq 0}} r\cdot \dim_{} M^{(r)}.
\end{equation}
It is readily seen that $\sw_G^{\rm AS}(M)=0$ if and only if $P_K$ acts trivially on $M$. Each non-vanishing piece $M^{(r)}$, for $r>0$, has in turn a \textit{central character decomposition} (\ref{CentralCharacterDecomposition})
\begin{equation}
	\label{Decompositions 3}
M^{(r)}=\bigoplus_{\chi} M_{\chi}^{(r)},
\end{equation}
indexed by a finite number of characters $\chi: {\rm Gr}_{\log}^r G_K \to \Lambda_{\chi}^{\times}$, where $\Lambda_{\chi}$ is a finite separable extension of $\Lambda$. As ${\rm Gr}_{\log}^r G_K$ is killed by $p$, the existence of $\psi$ ensures that $\chi$ factors as ${\rm Gr}_{\log}^r G_K \xrightarrow{\overline{\chi}}\mathbb{F}_p\xrightarrow{\psi} \Lambda^{\times}$. Then, H. Hu \cite{Hu1} defines the Abbes-Saito \textit{characteristic cycle} ${\rm CC}_{\psi}(M)$ of $M$ as
\begin{equation}
	\label{Decompositions 4}
{\rm CC}_{\psi}(M)=\bigotimes_{r\in \Q{>0}}\bigotimes_{\chi\in X(r)} ({\rm rsw}(\overline{\chi})(\pi^r))^{\otimes(\dim_{\Lambda} M^{(r)}_{\chi})} ~ \in (\Omega^1_k(\log)\otimes_k\overline{k})^{\otimes m},
\end{equation}
where $m=\dim_{\Lambda} M/M^{(0)}$ and $X(r)$ is the set of characters $\chi:{\rm Gr}_{\log}^r G_K \to \Lambda_{\chi}^{\times}$ in \eqref{Decompositions 3}. Note, in relation to $\pi^r$ appearing in \eqref{Decompositions 4}, that ${\rm CC}_{\psi}(M)$ is unambiguously defined (see below \eqref{CC 5}).
A finite separable extension of $K$ is said to be of \textit{type} (II) if its ramification index over $K$ is one and its residue field is a purely inseparable and monogenic extension of $k$.
Under the assumption that $p$ is not a uniformizer of $K$ and that $L$ is of type (II) over a sub-extension which is unramified over $K$, it is shown in \cite[10.5]{Hu1} that ${\rm CC}_{\psi}(M) \in (\Omega^1_k)^{\otimes m}$.
The characteristic cycle ${\rm CC}_{\psi}(M)$ is the codimension one incarnation of ${\rm CC}(\F)$ for a constructible $\ell$-adic étale sheaf $\F$ (\ref{DimensionSup}).

\subsection{} \label{DisqueFaisceau}
In this paper, we establish a new relationship between $\sw_G^{\rm AS}(M)$ and ${\rm CC}_{\psi}(M)$ in the following setting. Assume that $K$ is complete and $k$ is algebraically closed. Let $D$ be the rigid unit disc over $K$, $D^{(t)}$ its subdisc of radius $\lvert \pi\lvert^t$, for $t\in \Q_{\geq 0}$, and $\F$ a lisse sheaf of $\Lambda$-modules on $D$. Let $\overline{0}\to D$ be a geometric point above the origin $0$ of $D$ and $\pi_1^{\textrm{ ét}}(D, \overline{0})$ the algebraic fundamental group of $D$ with base point $\overline{0}$ \cite[\S 2]{deJong}.
By \cite[2.10]{deJong}, $\F$ corresponds to the data of a finite Galois étale connected cover $f: X\to D$ and a finite dimensional continuous $\Lambda$-representation $\rho_{\F}$ of $\pi_1^{\textrm{ ét}}(D, \overline{0})$ factoring through the finite quotient $G={\rm Aut}(X/D)$ of $\pi_1^{\textrm{ ét}}(D, \overline{0})$.
There exists a finite extension $K'/K$ such that $t\in v(K')$ (\ref{NotationIntro}), $D_{K'}^{(t)}$  and $X_{K'}^{(t)}=f^{-1}(D_{K'}^{(t)})$ have integral models $\mathfrak{D}_{K'}^{(t)}$ and $\mathfrak{X}_{K'}^{(t)}$ over $\Spf(\O_{K'})=\{s'\}$ with geometrically reduced special fibers $\mathfrak{D}_{s'}^{(t)}$ and $\mathfrak{X}_{s'}^{(t)}$ respectively (\ref{ModelsFormels}). Let $\overline{\p}^{(t)}$ be a geometric point of $\mathfrak{D}_{s'}^{(t)}$ above its generic point $\p^{(t)}$ and $\overline{\q}^{(t)}$ a codimension one geometric point of $\mathfrak{X}_{s'}^{(t)}$ above $\overline{\p}^{(t)}$. Then, these geometric points define an extension of henselian discrete valuation rings $\O_{\mathfrak{D}_{K'}^{(t)}, \overline{\p}^{(t)}} \subset \O_{\mathfrak{X}_{K'}^{(t)}, \overline{\q}^{(t)}}$, where $\O_{\mathfrak{D}_{K'}^{(t)}, \overline{\p}^{(t)}}$ (resp. $\O_{\mathfrak{X}_{K'}^{(t)}, \overline{\q}^{(t)}}$) is the formal étale local ring of $\mathfrak{D}_{K'}^{(t)}$ (resp. $\mathfrak{X}_{K'}^{(t)}$) at $\overline{\p}^{(t)}$ (resp. $\overline{\q}^{(t)}$) \eqref{ALFormelEtale 1}. The induced extension of fields of fractions is Galois of group the stabilizer $G_{\overline{\q}^{(t)}}$ of $\overline{\q}^{(t)}$ under the natural action of $G$ on the set of codimension one geometric points of $\mathfrak{X}_{s'}^{(t)}$ above $\overline{\p}^{(t)}$ (\ref{Action Transitive}). We complete this extension and get a representation $M_{\overline{\q}^{(t)}}=\rho_{\F}\lvert G_{\overline{\q}^{(t)}}$ of its Galois group for which we can compute the Swan conductor $\sw_{G_{\overline{\q}^{(t)}}}^{\rm AS}(M_{\overline{\q}^{(t)}})$ \eqref{Decompositions 2} and the characteristic cycle ${\rm CC}_{\psi}(M_{\overline{\q}^{(t)}})$ \eqref{Decompositions 4}, both independent of the choice of both $K'$ large enough and $\overline{\q}^{(t)}$ in the set of codimension one geometric points of $\mathfrak{X}_{s'}^{(t)}$ above $\overline{\p}^{(t)}$ (see \ref{FaisceauLisseSurD}). 

We note also that the residue field $\kappa(\overline{\p}^{(t)})$ of $\O_{\mathfrak{D}_{K'}^{(t)}, \overline{\p}^{(t)}}$ coincide with $\O_{\mathfrak{D}_{s'}^{(t)}, \overline{\p}^{(t)}}$. Let 
 $\ord_{\overline{\p}^{(t)}}: \kappa(\overline{\p}^{(t)})^{\times}\to \bZ$ be the normalized discrete valuation map defined by the origin of $\mathfrak{D}_{s'}^{(t)}$; we denote still by $\ord_{\overline{\p}^{(t)}}$ its unique multiplicative extension to $(\Omega^1_{\kappa(\overline{\p}^{(t)})})^{\otimes (\dim_{\Lambda}(M_{\overline{\q}^{(t)}}/M_{\overline{\q}^{(t)}}^{(0)}))}$ (where $M_{\overline{\q}^{(t)}}^{(0)}$ is the tame part of $M_{\overline{\q}^{(t)}}$) defined by the relation $\ord_{\overline{\p}^{(t)}} (b{\rm d}a)=\ord_{\overline{\p}^{(t)}}(b)$, for $a, b \in \kappa(\overline{\p}^{(t)})^{\times}$ such that $\ord_{\overline{\p}^{(t)}}(a)=1$.

The main result of this article is the following.

\begin{thm}
	\label{Theoreme principal}
The function
\begin{equation}
	\label{Theoreme principal 1}
\sw_{\rm AS}(\F, \cdot) : \Q_{\geq 0}\to \Q, \quad t\mapsto \sw_{G_{\overline{\q}^{(t)}}}^{\rm AS}(M_{\overline{\q}^{(t)}})
\end{equation}
is continuous and piecewise linear, with finitely many slopes which are all integers. Its right derivative is the locally constant function
\begin{equation}
	\label{Theoreme principal 2}
\varphi_s(\F, \cdot): \Q_{\geq 0}\to \Q, \quad t\mapsto -\ord_{\overline{\p}^{(t)}}({\rm CC}_{\psi}(M_{\overline{\q}^{(t)}})) - \dim_{\Lambda}(M_{\overline{\q}^{(t)}}/M_{\overline{\q}^{(t)}}^{(0)}).
\end{equation}
\end{thm}

We remark that, as $v(K)=\Z$, the induced value group of $\O_{\mathfrak{D}_{K'}^{(t)}, \overline{\p}^{(t)}}$ is canonically isomorphic to the subgroup $\frac{1}{e_{K'/K}}\Z$ of $\Q$, where $e_{K'/K}$ is the ramification index of $K'/K$. Thus, with this normalization, $\sw_{\rm AS}(\F, t)$ is not an integer in general.

We note also that, following \cite[\S 11]{Hu1}, by the theorem of Deligne and Kato on the dimension of the space of nearby cycles (\cite[6.7]{K1}, \cite[11.9]{Hu1}), $\varphi_s(\F, t)$ is equal to the total dimension of $\Psi_0(\F_{\lvert D^{(t)}})$, up to a correction term (notation of \textit{loc. cit}). Hence, one can vaguely interpret the second half of Theorem \ref{Theoreme principal} as saying that the derivative of the function $\sw_{\rm AS}(\F, \cdot)$ at $t$ is the dimension of the nearby cycles of $\F_{\lvert D^{(t)}}$ at the origin of $D^{(t)}$.

Theorem \ref{Theoreme principal} is proved in section \ref{Proof} as Theorem \ref{Main theorem}. We expect $\sw_{\rm AS}(\F, \cdot)$ to be a decreasing and convex function, and we also expect the theorem to hold even when $\F$ has horizontal ramification.
We will come back to these questions in a forthcoming work.

\subsection{} \label{LutkebohmertRiemann}
For a cover $f$ as in \ref{DisqueFaisceau}, a variational result for a discriminant associated to $f$, analogous to Theorem \ref{Theoreme principal}, was established by Lütkebohmert and played a key role in his proof of the $p$-adic Riemann existence theorem \cite{Lutke}.

\subsection{}\label{Lutkebohmert}
Let $X$ be a smooth $K$-rigid space and $f: X\to D$ a finite flat morphism which is étale over an admissible open subset of $D$ containing $0$. Lütkebohmert associates to $f$ a discriminant function $\partial_f^{\alpha}: \Q_{\geq 0}\to \mathbb{R}^+$ of the variable $t$ and shows that it is continuous and piecewise linear while also making its slopes explicit (\ref{DérivéeLutke}). By the Weierstrass preparation theorem, if a function $g$ on an annulus $A(\rho, \rho')=\{x\in \overline{K}~ \lvert ~ \rho \geq v(x)\geq \rho'\}$ ($\rho, \rho'\in \Q$) is invertible, it can be written in the form
\begin{equation}
	\label{Lutkebohmert 1}
	g(\xi)=c \xi^d (1+ h(\xi)), \quad {\rm with}\quad h(\xi)=\sum_{i\in\Z-\{0\}} h_i\xi^{i},
\end{equation}
where $c\in K^{\times}$, $d\in \Z$ and $h$ is a function on $A(\rho, \rho')$ such that $\lvert h\rvert_{{\sup}}< 1$; then, the integer $d$ is called the \textit{order of} $g$. When $X=A(r/d, r'/d)$, for rational numbers $r\geq r' \geq 0$, and $f: A(r/d, r'/d)\to A(r, r')\subset D$ is finite étale of order $d$, Lütkebohmert computes $\partial_f^{\alpha}$ explicitly and finds that its right derivative at $t\in [r', r[\cap\Q$ is 
\begin{equation}
	\label{Lutkebohmert 2}
\frac{d}{dt}\partial_f^{\alpha}(t+)=\sigma-d+1,
\end{equation}
where $\sigma$ is the order of the derivative of $f$ (\ref{FormuleLutke}). Lütkebohmert's general statement reduces to this case thanks to the semi-stable reduction theorem which gives a finite sequence of rational numbers $0=r_{n+1} < r_n < \cdots < r_1 < r_0=\infty$ such that, over each open annulus $A^{\circ}(r_{i-1}, r_i)=\{x\in \overline{K}~ \lvert ~ r > v(x) > r'\}$, $f$ is a sum of étale morphisms on annuli (hence of the form \eqref{Lutkebohmert 1}).

\subsection{}\label{Kazuya Kato} Theorem \ref{Theoreme principal} will ultimately be deduced from the aforementioned variational result of Lütkebohmert \ref{DérivéeLutke}. The bridge between the two is provided by Kato's ramification theory for a $\Z^2$-valuation ring. The latter is a valuation ring whose value group is isomorphic to $\Z^2$ endowed with the lexicographic order. The theory partly rests on an important theorem of Epp on elimination of wild ramification \cite{Epp}. Let $\overline{x}^{(t)}$ be a geometric point at the origin of $\mathfrak{D}_{s'}^{(t)}$. We put $A^{(t)}=\O_{\mathfrak{D}_{K'}^{(t)}, \overline{x}^{(t)}}$ \eqref{ALFormelEtale 1}; and with the generic point $\p^{(t)}$, we cook up a henselian $\Z^2$-valuation ring $V^h_t$, which is a henselization of a $\Z^2$-valuation ring $V_t$ such that $A^{(t)}\subset V_t\subset A^{(t)}_{\p^{(t)}}$ (\ref{(V) formel})). The restriction of $f$ to $X^{(t)}=f^{-1}(D^{(t)})$ has a "normalized" integral model $\widehat{f}^{(t)}:\mathfrak{X}_{K'}^{(t)}\to \mathfrak{D}_{K'}^{(t)}$ (\ref{Normalized integral model}, \ref{radmissible}) (possibly after enlarging $K'$) which induces monogenic extensions of $\Z^2$-valuations rings $V^h_t\to W^h_{j, t}$, where $W^h_{j, t}$ is the henselization of a $\Z^2$-valuation ring $W_{j, t}\supset A_j^{(t)}=\O_{\mathfrak{X}_{K'}^{(t)}, \overline{x}_j}$ and the $\overline{x}_j$ are the geometric points of the special fiber $\mathfrak{X}_{s'}^{(t)}$ of $\mathfrak{X}_{K'}^{(t)}$, above $\overline{x}^{(t)}$. Applied to these extensions, Kato's theory yields characters $\widetilde{a}_f^{\alpha}(t)$ and $\widetilde{\sw}_f^{\beta}(t)$ of $G={\rm Aut}(X/D)$ with values in $\Q$ and $\Z$ respectively.

\begin{thm}[{Corollary \ref{VariationCoeffsFinis}}]
	\label{Theorem à la Ramero}
We assume that $X$ has trivial canonical sheaf and let $\chi \in R_{\Lambda}(G)$. We denote by $\langle \cdot, \cdot \rangle_G$ the usual pairing of class functions on $G$. Then, the map
\begin{equation}
	\label{Theorem à la Ramero 1}
\widetilde{a}_f^{\alpha}(\chi, \cdot): \Q_{\geq 0} \to \Q,\quad t\mapsto	\langle \widetilde{a}_f^{\alpha}(t), \chi\rangle_{G}
\end{equation}
is continuous and piecewise linear, with finitely many slopes which are all integers. Its right derivative at $t\in \Q_{\geq 0}$ is
\begin{equation}
	\label{Theorem à la Ramero 2}
\frac{d}{dt} \widetilde{a}_f^{\alpha}(\chi, t+)=\langle \widetilde{\sw}_f^{\beta}(t), \chi\rangle_{G}.
\end{equation}
\end{thm}

In \cite[\S 3 and \S 4]{Ramero}, L. Ramero proved a similar result for $f: X\to D$ an \textit{étale} morphism between adic spaces in the sense of Huber and with a somewhat ad hoc ramification filtration due to R. Huber. Although our proofs are independent of his, we took inspiration from his work to arrive at this statement. We should also mention that F. Baldassarri \cite{Balda}, and later A. Pulita \cite{Pulita} along with J. Poineau \cite{PP}, as well as K. Kedlaya \cite{Kedlaya}, proved analogous continuity and piecewise linearity results for the radii of convergence of differential equations with irregular singularities on $p$-adic analytic curves.

\subsection{}\label{KatoLutkebohmert} By Brauer induction, we reduce Theorem \ref{Theorem à la Ramero} to the case of the character $\chi=r_G$ of the regular representation of $G$. Then, the link to Lütkebohmert's discriminant is through the following identities. For $t\in\Q_{\geq 0}$, we have
\begin{equation}
	\label{KatoLutkebohmert 1}
\langle \widetilde{a}_f^{\alpha}(t), r_G\rangle_{G}=\partial_f^{\alpha}(t).
\end{equation}
Whence, the function $t\mapsto\langle \widetilde{a}_f^{\alpha}(t), r_G\rangle_{G}$ is piecewise linear. Let $0=r_{n+1} < r_n < \cdots < r_1 < r_0=\infty$ be the partition of $\Q_{\geq 0}$ given by the semi-stable reduction theorem as in \ref{Lutkebohmert}. Then, assuming also that $X$ \textit{has trivial canonical sheaf}, for $t\in [r_i, r_{i-1}[$, we have
\begin{equation}
	\label{KatoLutkebohmert 2}
\langle \widetilde{\sw}_f^{\beta}(t), r_G\rangle_{G}=\frac{d}{dt}\partial_f^{\alpha}(t+)=\sigma_i -d + \delta_f(i),
\end{equation}
where $\sigma_i$ is the total order of the derivative of the restriction of $f$ over $A^{\circ}(r_{i-1}, r_r)$ (\ref{WeierstrassDiscOuvert}, \ref{total order}) and $\delta_f(i)$ the number of connected components of the inverse image of $A^{\circ}(r_{i-1}, r_i)$ by $f$.

While the identity \eqref{KatoLutkebohmert 1} is an incarnation of the classical equality of the valuation of the different with the value of the Artin character at $1$ \cite[IV, \S 2, Prop. 4]{Serre1}, the identity \eqref{KatoLutkebohmert 2} is new and more subtle. Let us explain how it is established. With the notation of \ref{Kazuya Kato}, the quotient ring $A_{j, 0}^{(t)}=A_j^{(t)}/\m_{K'}A_j^{(t)}$ is reduced. Let $\widetilde{A_{j, 0}^{(t)}}$ be its normalization in its total ring of fractions and put $\delta_j^{(t)}=\dim_k(\widetilde{A_{j, 0}^{(t)}}/A_{j, 0}^{(t)})$. Set $A_{K'}^{(t)}=A^{(t)}\otimes_{\O_{K'}} K'$ and $A_{j, K'}^{(t)}=A_j^{(t)}\otimes_{\O_{K'}} K'$. Let $T_j^{(t)}$ be the determinant $K'$-linear homomorphism induced by the bilinear trace map $A_{j, K'}^{(t)}\times A_{j, K'}^{(t)}\to A_{K'}^{(t)}$ and set $d_j^{(t)}$ to be the $K'$-dimension of the cokernel of $T_j^{(t)}$. Then, \eqref{KatoLutkebohmert 2} is deduced from the following key result, which is interpreted (and proved) as a nearby cycle formula.

\begin{prop}[{Proposition \ref{Vanishing Cycles Lutke}}]
	\label{Cycles Evanescents}
Assume that X has trivial canonical sheaf. Then, for each $i=1,\ldots, n$ and each $t\in \rbrack r_i, r_{i-1}\lbrack \cap \Q$, we have the following equality
\begin{equation}
	\label{Cycles Evanescents 1}
	\sum_j \left(d_j^{(t)}- 2\delta_j^{(t)}+ \lvert P_ j^{(t)}\lvert\right)= \sigma_i + \delta_f(i),
\end{equation}
where $ P_j^{(t)}$ is the set of height $1$ prime ideals of $A_j^{(t)}$ above $\m_K$ and $\lvert \cdot\lvert$ denotes cardinality.
\end{prop}

Proposition \ref{Cycles Evanescents} (with slightly modified notation) is proved at the end of section \ref{Discriminant} in the following way. We first construct a compactification $\mathfrak{Y}_{K'}^{(t)}$	of the integral model $\mathfrak{X}_{K'}^{(t)}$, which is then algebraized as $Y_{K'}^{(t)}$, via Grothendiek's algebraization theorem. We then approximate $\widehat{f}^{(t)}$ by an algebraic function $g^{(t)}$ on $Y_{K'}^{(t)}$ thanks to a rigid Runge theorem due to Raynaud. The identity \eqref{Cycles Evanescents 1} follows from the computation of the degree of the divisor defined by the differential $dg^{(t)}$ using the Riemann-Hurwitz formula and an expression for $2\delta_j^{(t)}- \lvert P_j^{(t)}\lvert + 1$ in the form of a nearby cycles formula for $Y_{K'}^{(t)}\to \Spec(\O_{K'})$ due to Kato.

\subsection{}\label{KatoHu}
The last bridge from this result to Theorem \ref{Theoreme principal} is provided by work of Hu \cite[Theorem 10.5]{Hu1} which identifies ${\rm CC}_{\psi}(M)$ to another characteristic cycle ${\rm KCC}_{\psi(1)}(\chi_M)$ constructed from Kato's Swan conductor with differential values \cite{K2}. The latter is not hard to compare with $\langle \widetilde{\sw}_f^{\beta}(t), \chi_M\rangle_{G}$. Such a comparison is carried out in Proposition \ref{Comparaison Kato AbbesSaito} and allows us to conclude.

\subsection{} \label{OrganisationDuTexte}
The text is organized as follows. In section \ref{Preliminaires}, we gather some constructions  about formal schemes, mainly the \emph{formal étale local ring} of a formal scheme at a geometric point and its relative behaviour. It also contains a recollection of normalized integral models and a \emph{maximum principle}-like result of Bosch and Lütkebohmert computing the degree of the divisor defined by the differential of a function in terms of the orders of its zeros. Section \ref{Valuations} is devoted to $\Z^2$-valuation rings and how they arise in algebraic and formal geometric settings. In section \ref{Discriminant}, we recall Lütkebohmert's work and relate it to Kato's by proving the nearby cycle formula of Proposition (\ref{Vanishing Cycles Lutke}). We also recall (and expand a little) the ramification theory of $\Z^2$-valuation (sections \ref{Kato 1} and \ref{Kato 2}) developed in \cite{K1}. The first variational result for the conductors given by this theory are stated and proved in section \ref{Var}. Kato's Swan conductor with differential values, the ramification theory of Abbes and Saito and the link between the two established by H. Hu are gathered in section \ref{Characteristic cycles}, as well as the comparison of the conductors coming from the two ramification theories. Then, we have everything at hand to deduce Theorem \ref{Theoreme principal} in section \ref{Proof}.

\subsection*{Acknowledgements} This is part of the author's PhD dissertation. It was funded by Université Paris-Saclay's \'Ecole Doctorale Mathématiques Hadamard and preprared at the Institut des Hautes \'Etudes Scientifiques whose hospitality he benefited from. The author is very much indebted to his advisor Ahmed Abbes for his constant support, careful reading, numerous suggestions, remarks and corrections. He thanks Q. Guignard, H. Hu and T. Saito for their interest in this work. Their comments helped improve its presentation.

\section{Preliminaries on formal schemes.}\label{Preliminaires}
\subsection{} \label{Notations Formelles}
Let $\O_K$ be a complete discrete valuation ring with field of fractions $K$, maximal ideal $\m_K$ and residue field $k$. We fix a uniformizer $\pi$ of $\O_K$ and assume, except in \ref{supnorm}-\ref{Normalized integral model}, that $k$ is algebraically closed. We put $\mathcal{S}=\Spf(\O_K)$ and denote by $s$ its unique point. In the whole section, all formal schemes are assumed to be locally noetherian.
 
\subsection{} \label{ConstructionKatoFormel} 
Let $\mathfrak{X}$ be an adic formal scheme over $\mathcal{S}$ \cite[\S 2.1]{EGR} and $x\in \mathfrak{X}$ a point.
Recall that the local ring $\O_{\mathfrak{X}, x}$ of $\mathfrak{X}$ at $x$ is defined as
\begin{equation}
	\label{ALFormel}
\O_{\mathfrak{X}, x}=\varinjlim_{x\in \mathcal{U}}\O_{\mathfrak{X}}(\mathcal{U}),\
\end{equation}
where $\mathcal{U}$ runs over affine formal open subschemes of $\mathfrak{X}$ containing $x$. It is indeed a local ring with residue field isomorphic to the residue field of the local ring $\O_{\mathfrak{X}_s, x}$ of the special fiber $\mathfrak{X}_s$ of $\mathfrak{X}$ \cite[10.1.6]{EGA.I}.
The formal scheme $\mathfrak{X}$ is said to be \textit{normal} at $x$ if $\O_{\mathfrak{X}, x}$ is normal. We say that $\mathfrak{X}$ is normal if it is normal at all its points.

\begin{lem}[{\cite[Lemma 1.2.1]{Conrad}}]
	\label{normal}
Assume that $\mathfrak{X}=\Spf(A)$ is affine. Then, $\mathfrak{X}$ is normal if and only if $A$ is normal.
\end{lem}

\begin{lem}	\label{Completion formelle}
Assume that $\mathfrak{X}=\Spf(A)$, $x$ corresponds to an open prime ideal $\p$ of $A$ and $\widehat{A_{\p}}$ denotes the $\pi$-adic completion of $A_{\p}$. Then, we have canonical flat local homomorphisms $A_{\p}\to \O_{\mathfrak{X}, x}\to \widehat{A_{\p}}$ which induce an isomorphism of $\pi$-adic completions $\widehat{\O_{\mathfrak{X}, x}}\xrightarrow{\sim} \widehat{A_{\p}}$. Hence, the canonical projection $\O_{\mathfrak{X}, x}\to \O_{\mathfrak{X}_s, x}$ yields an isomorphism
 \begin{equation}
 	\label{Completion formelle 1}
 \O_{\mathfrak{X}, x}/\pi\O_{\mathfrak{X}, x} \xrightarrow{\sim} \O_{\mathfrak{X}_s, x}.
 \end{equation}
\end{lem}

\begin{proof}
The construction of the local homomorphisms $A_{\p}\to \O_{\mathfrak{X}, x}\to\widehat{A_{\p}}$ is clear from
\begin{equation}
	\label{Completion formelle 2}
\O_{\mathfrak{X}, x}=\varinjlim_{f\in A-\p}\varprojlim_n A[\frac{1}{f}]/(\pi^{n+1}),
\end{equation}
with the composition being the canonical completion homomorphism $A_{\p}\to \widehat{A_{\p}}$. Since the latter is flat and $\O_{\mathfrak{X}, x}\to\widehat{A_{\p}}$ is faithfully flat \cite[Chap. 0, 7.6.18]{EGA.I}, $A_{\p}\to \O_{\mathfrak{X}, x}$ is also faithfully flat. From \eqref{Completion formelle 2}, we see that $\O_{\mathfrak{X}, x}/(\pi^{n+1})\xrightarrow{\sim} A_{\p}/(\pi^{n+1})$ since the filtered colimit functor is exact and thus commutes with quotients. Taking the projective limit gives the desired isomorphism.
\end{proof}

\subsection{} \label{ALFormelEtale}
Let $\overline{x}\to \mathfrak{X}_s$ be a geometric point with image $x$. The $k$-schemes $\mathfrak{X}_s$ and $\overline{x}$ have natural structures of $\mathcal{S}$-formal schemes which make $\mathfrak{X}_s$ into a formal closed subscheme of $\mathfrak{X}$ and make $\overline{x}\to \mathfrak{X}_s$ into an $\mathcal{S}$-morphism. We say that the composition $\overline{x}\to \mathfrak{X}_s \hookrightarrow \mathfrak{X}$ is a \textit{geometric point of} $\mathfrak{X}$. We denote by ${\rm Nbh}_{\overline{x}}^{\rm aff}(\mathfrak{X})$ the following category. The objects of ${\rm Nbh}_{\overline{x}}^{\rm aff}(\mathfrak{X})$ are all triples $(\mathcal{U}, \mathcal{U}\to \mathfrak{X}, \varphi_{\mathcal{U}}: \overline{x}\to \mathcal{U})$, simply denoted $\mathcal{U}$, where $\mathcal{U}$ is a formal affine scheme, $\mathcal{U}\to \mathfrak{X}$ is a formal étale morphism and $\varphi_{\mathcal{U}}$ is an $\mathfrak{X}$-morphism. A morphism of ${\rm Nbh}_{\overline{x}}^{\rm aff}(\mathfrak{X})$ is an $\mathfrak{X}$-morphism $f: \mathcal{U}\to\mathcal{V}$ such that $f\circ\varphi_{\mathcal{U}}=\varphi_{\mathcal{V}}$.
The étale local ring of $\mathfrak{X}$ at $\overline{x}$ is defined as
\begin{equation}
	\label{ALFormelEtale 1}
\O_{\mathfrak{X}, \overline{x}}=\varinjlim_{\mathcal{U}\in {\rm Nbh}_{\overline{x}}^{\rm aff}(\mathfrak{X})}\O_{\mathfrak{X}}(\mathcal{U}).
\end{equation}
The ring $\O_{\mathfrak{X}, \overline{x}}$ is local and henselian by Proposition \ref{ALFE} below.

\begin{lem}
	\label{Foncteurs cofinaux}
Let $X$ be a scheme, $X_0$ a closed subscheme of $X$ and $\overline{x}$ a geometric point of $X_0$. Denote by ${\rm Nbh}_{\overline{x}}(X)$ $($resp. ${\rm Nbh}_{\overline{x}}(X_0))$ the category of étale neighborhoods of $\overline{x}$ in $X$ $($resp. in $X_0)$ and let ${\rm Nbh}_{\overline{x}}^{\rm aff}(X)$ $($resp. ${\rm Nbh}_{\overline{x}}^{\rm aff}(X_0))$ be its full subcategory of affine objects. Let $\varphi: {\rm Nbh}_{\overline{x}}(X)\to {\rm Nbh}_{\overline{x}}(X_0)$ $($resp. $\varphi_{\rm aff}: {\rm Nbh}_{\overline{x}}^{\rm aff}(X)\to {\rm Nbh}_{\overline{x}}^{\rm aff}(X_0))$ be the functor $U\mapsto U_0=U\times_X X_0$. Then, $\varphi_{\rm aff}$ and $\varphi$ are cofinal.
\end{lem}
\begin{proof}
 We can assume that $X=\Spec(A)$ is affine and $X_0=\Spec(A_0)$ with $A_0=A/J$ for some ideal $J$. Let $\Spec(B_0)=U_0\to X_0$ be an object of ${\rm Nbh}_{\overline{x}}^{\rm aff}(X_0)$. We prove that there exists an affine étale neighborhood $V\to X$ of $\overline{x}$ with a map $V_0\to U_0$. By \cite[Chap. V, Théorème 1]{Raynaud-ALH}, we can assume that $B_0$ is a standard étale algebra over $A_0$, i.e. $B_0=(A_0[X]/\overline{P})[\frac{1}{\bar{Q}}]$ for some $P, Q \in A[X]$ such that $\overline{P}'=P' ~{\rm mod}~J$ is invertible in $B_0$. Set $B=(A[X]/P)[\frac{1}{QP'}]$. Then, $V=\Spec(B)$ is étale over $X$ and $V_0=V\times_X X_0=U_0$. Hence, $V$, with the composition $\overline{x}\to U_0\hookrightarrow V$ is the desired object of ${\rm Nbh}_{\overline{x}}^{\rm aff}(X)$. Now suppose that $U_0$ is an object of ${\rm Nbh}_{\overline{x}}^{\rm aff}(X_0)$, $V$ is an object of ${\rm Nbh}_{\overline{x}}^{\rm aff}(X)$ and $f_0, g_0 : \varphi_{\rm aff}(V)\rightrightarrows U_0$ are morphisms of ${\rm Nbh}_{\overline{x}}^{\rm aff}(X_0)$. Since ${\rm Nbh}_{\overline{x}}^{\rm aff}(X_0)$ is a filtered category, we can find an equalizer  of $f_0$ and $g_0$, i.e a morphism $h_0: W_0\to V_0$ in ${\rm Nbh}_{\overline{x}}^{\rm aff}(X_0)$ such that $f_0 h_0= g_0 h_0$. The first part of the proof, applied to $V$ and $V_0$ instead of $X$ and $X_0$, gives an object $W'$ of ${\rm Nbh}_{\overline{x}}^{\rm aff}(V)$ with a morphism $W'_0\to W_0$; hence, composing with $V\to X$,  $W'$ is an object of ${\rm Nbh}_{\overline{x}}^{\rm aff}(X)$ with a morphism $h': W'\to V $. The morphism $\varphi_{\rm aff}(h')$ is then an equalizer of $f_0, g_0$. This proves that $\varphi_{\rm aff}$ is cofinal \cite[I, 8.1.3]{SGA4}.
It implies that $\varphi$ is cofinal since the inclusion functors ${\rm Nbh}_{\overline{x}}^{\rm aff}(X) \to {\rm Nbh}_{\overline{x}}(X) $ and ${\rm Nbh}_{\overline{x}}^{\rm aff}(X_0) \to {\rm Nbh}_{\overline{x}}(X)$ are cofinal.
\end{proof}

\subsection{} \label{Equivalence sites étales}
We keep the notation of \ref{ALFormelEtale}.
The functor $\mathcal{U}\mapsto \mathcal{U}_s$ induces an equivalence of the étale sites $\mathfrak{X}_{\textrm{ét}}$ and $\mathfrak{X}_{s, \textrm{ét}}$ (\textit{cf.} \cite[18.1.2]{EGA.IV} and \cite[2.4.8]{EGR}). Therefore, the colimit in \eqref{ALFormelEtale 1} can be taken in the category ${\rm Nbh}_{\overline{x}}^{\rm aff}(\mathfrak{X}_s)$ as defined in \ref{Foncteurs cofinaux}
\begin{equation}
	\label{Changement sites}
\O_{\mathfrak{X}, \overline{x}}=\varinjlim_{\mathcal{U}_s \in {\rm Nbh}_{\overline{x}}^{\rm aff}(\mathfrak{X}_s)} \O_{\mathfrak{X}}(\mathcal{U}).
\end{equation}
We thus get a canonical surjective homomorphism
\begin{equation}
	\label{Surjection canonique}
\O_{\mathfrak{X}, \overline{x}}\to \O_{\mathfrak{X}_s, \overline{x}}.
\end{equation}
When $\mathfrak{X}=\Spf(A)$ is affine, denoting $X=\Spec(A)$, we get from Lemma \ref{Foncteurs cofinaux} above
\begin{equation}
	\label{Limit Inductive}
\O_{\mathfrak{X}, \overline{x}}\cong \varinjlim_{\Spec(B)\in {\rm Nbh}_{\overline{x}}^{\rm aff}(X)} \widehat{B}.
\end{equation}

\begin{prop}
	\label{ALFE}
We keep the notation of {\rm \ref{ALFormelEtale}}. We assume that $\mathfrak{X}=\Spf(A)$ is affine, $\p$ is the open prime ideal of $A$ corresponding to $x$ and denote by $A_{\p}^{\rm sh}$ the strict henselization of $A_{\p}$ with respect to the separably closed field $\kappa(\overline{x})$ defining $\overline{x}$. Then, $\O_{\mathfrak{X}, \overline{x}}$ is a noetherian henselian local ring whose residue field is $\kappa(\overline{x})$, and we have canonical local homomorphisms of local rings
\begin{equation}
	\label{ALFE 1}
A_{\p}^{\rm sh}\to \O_{\mathfrak{X}, \overline{x}}\to \widehat{A_{\p}^{\rm sh}}
\end{equation}
inducing an isomophism $\widehat{\O_{\mathfrak{X}, \overline{x}}}\xrightarrow{\sim}\widehat{A_{\p}^{\rm sh}}$ between $\pi$-adic completions. Hence, the canonical surjection {\rm \eqref{Surjection canonique}} induces an isomorphism 
\begin{equation}
	\label{ALFE 2}
\O_{\mathfrak{X}, \overline{x}}/\pi\O_{\mathfrak{X}, \overline{x}}\xrightarrow{\sim} \O_{\mathfrak{X}_s, \overline{x}}.
\end{equation}
\end{prop}

\begin{proof}
We use the description \eqref{Limit Inductive} of $\O_{\mathfrak{X}, \overline{x}}$ and run the argument given in \cite[Chap. 0, 7.6.17]{EGA.I} for the proof of the analogous statement for $\O_{\mathfrak{X}, x}$. For $B\in {\rm Nbh}_{\overline{x}}^{\rm aff}(X)$, we let $\p_B$ be the the image in $\Spec(B)$ of the geometric point $\overline{x}$ and set $\p_{\overline{x}}=\varinjlim_B \p_B$. We show that any element in $\O_{\mathfrak{X}, \overline{x}}-\p_{\overline{x}}$ is invertible. Indeed such an element is the image in $\O_{\mathfrak{X}, \overline{x}}$ of an element $z=(z_n)$ of some $\widehat{B}$, and it is enough to show that $z$ is invertible modulo $\pi$ \cite[Chap. 0, 7.1.12]{EGA.I}. Since $\widehat{B}/\pi\widehat{B}\cong B/\pi B=B_0$, it is enough to show that $ z_0$ is invertible in some $C_0$ such that $\Spec(C)\to\Spec(B)$ is a morphism in ${\rm Nbh}_{\overline{x}}^{\rm aff}(X)$. Since $z_0$ is not in $\p_B$, the latter is a prime ideal of the étale $A_0$-algebra $C_0=B_0[\frac{1}{z_0}]$. By Lemma \ref{Foncteurs cofinaux}, the latter algebra lifts to some $C$ with $\Spec(C)\in {\rm Nbh}_{\overline{x}}^{\rm aff}(X)$ and above $\Spec(B)$. Hence, the image of $z$ in $\widehat{C}$ is invertible. This proves that $\O_{\mathfrak{X}, \overline{x}}$ is a local ring with maximal ideal $\p_{\overline{x}}$. To prove that $\O_{\mathfrak{X}, \overline{x}}$ is henselian, we must show that for every étale homomorphism $\varphi:\O_{\mathfrak{X}, \overline{x}}\to A'$, every section of $\varphi\otimes\kappa(\overline{x})$ descends to a section of $\varphi$. Since $\O_{\mathfrak{X}, \overline{x}}$ is colimit over $\mathcal{U}$, by \cite[17.7.8]{EGA.IV}, $\varphi$ is the base change to $\O_{\mathfrak{X}, \overline{x}}$ of an étale homomorphism $\varphi_{\mathcal{U}}: A_{\mathcal{U}}\to A_1$ for some $\mathcal{U}=\Spf(A_{\mathcal{U}})$ in ${\rm Nbh}_{\overline{x}}^{\rm aff}(\mathfrak{X})$. Taking the $\pi$-adic completion, this homomorphism extends to an adic étale homomorphism $A_{\mathcal{U}}\to\widehat{A}_1$. The composition $A\to A_{\mathcal{U}}\to \widehat{A}_1$ is thus étale, hence $\mathcal{U}_1=\Spf(\widehat{A}_1)$ is in ${\rm Nbh}_{\overline{x}}^{\rm aff}(\mathfrak{X})$. Hence, we have the following commutative diagram
\begin{equation}
	\label{ALFE 3}
\xymatrix{
A_1 \ar[rr] \ar[dr] & & A' \\
A_{\mathcal{U}} \ar[u] \ar[r] & \widehat{A}_1\ar[r] & \O_{\mathfrak{X}, \overline{x}}\ar[u]
}
\end{equation}
Since the square in this diagram is co-Cartesian, the composition $A_1\to \widehat{A}_1\to\O_{\mathfrak{X}, \overline{x}}$ induces a section $A'\to \O_{\mathfrak{X}, \overline{x}}$ of $\varphi$.

Now, on the one hand, since, for any $f\in A-\p$, $D(f)$ is an étale neighborhood of $\overline{x}$, using \eqref{Limit Inductive}, we get a canonical local homomorphism
\begin{equation}
	\label{ALFE 4}
A_{\p}=\varinjlim_{\p\in D(f)} A_f \to \O_{\mathfrak{X}, \overline{x}}
\end{equation}
which extends to $A_{\p}^{\rm sh}\to \O_{\mathfrak{X}, \overline{x}}$ by the universal property of the strict henselization of $A_{\p}$ with respect to the residue extension $\kappa(\p)\to \kappa(\overline{x})$. On the other hand, every $\Spec(B)\in {\rm Nbh}_{\overline{x}}^{\rm aff}(X)$ induces an $X$-morphism $X_{(\overline{x})}\to \Spec(B)$, i.e an $A$-homomorphism $B\to A_{\p}^{\rm sh}$. Completing it $\pi$-adically and taking the limit in \eqref{Limit Inductive} give a local homomorphism $\O_{\mathfrak{X}, \overline{x}}\to \widehat{A_{\p}^{\rm sh}}$, whence the homomorphisms \eqref{ALFE 1}. The $\pi$-adic completion of the homomorphism $\O_{\mathfrak{X}, \overline{x}}\to \widehat{A_{\p}^{\rm sh}}$ is an isomorphism because, being exact, the filtered colimit functor commutes with quotients and $\widehat{B}/\pi^n \widehat{B}\xrightarrow{\sim} B/\pi^n B$. This implies the claim on the residue fields. As $A$ is noetherian, it also implies that $\widehat{\O_{\mathfrak{X}, \overline{x}}}$ is noetherian; hence, by the faithful flatness of $\O_{\mathfrak{X}, \overline{x}}\to \widehat{\O_{\mathfrak{X}, \overline{x}}}$, $\O_{\mathfrak{X}, \overline{x}}$ is also noetherian.
The last statement of the proposition follows since $\O_{\mathfrak{X}_s, \overline{x}}=(A/(\pi))_{\p}^{\rm sh}$.
\end{proof}

\begin{lem}
	\label{RelevementVoisEtaleSchema}
Let $X\to Y$ be a finite morphism of schemes. Let $\overline{y}$ be a geometric point of $Y$ and $\overline{x}_1, \ldots, \overline{x}_n$ be the geometric points of $X$ above $\overline{y}$. Then, there exist an étale neighborhood $Y'\to Y$ of $\overline{y}$ and, for each $i$, an étale neighborhood $X'_i\to X$ of $\overline{x}_i$ such that
\begin{equation}
	\label{RelevementVoisEtaleSchema 1}
Y' \times_Y X \xrightarrow{\sim} \coprod_{i=1}^n X'_i .
\end{equation}
\end{lem}

\begin{proof}
We can assume $X=\Spec(B)$ and $Y=\Spec(A)$, with $B$ a finite $A$-algebra. Then, $B\otimes_A \O_{Y, \overline{y}}$ is a finite algebra over the henselian local ring $\O_{Y, \overline{y}}$ and thus it is isomorphic to the product of its localizations at its maximal ideals. These maximal ideals are in bijection with the points of the fiber $X_{\overline{y}}=\{\overline{x}_1, \ldots, \overline{x}_n\}$ (harmless abuse of notation here) and we denote them by $\q_1, \ldots, \q_n$. By \cite[2.8.20]{Fu}, we have  $(B\otimes_A \O_{Y, \overline{y}})_{\q_i}\cong \O_{X, \overline{x}_i}$ for each $i$. Hence, the aforementioned decomposition of $B\otimes_A \O_{Y, \overline{y}}$ yields a decomposition into connected components
\begin{equation}
	\label{RelevementVoisEtaleSchema 2}
Y_{(\overline{y})}\times_Y X\cong \coprod_{i=1}^n X_{(\overline{x}_i)},
\end{equation}
where $Y_{(\overline{y})}$ and $X_{(\overline{x}_i)}$ denote strict localizations. Let $e_1,\ldots, e_n$ be the idempotent elements of
\begin{equation}
	\label{RelevementVoisEtaleSchema 3}
\Gamma(Y_{(\overline{y})}\times_Y X, \O_{Y_{(\overline{y})}\times_Y X})= \varinjlim_{Y'\in {\rm Nbh}_{\overline{y}}^{\rm aff}(Y)} \Gamma(Y'\times_Y X, \O_{Y'\times_Y X}),
\end{equation}
corresponding to the decomposition \eqref{RelevementVoisEtaleSchema 2}. Then, there exists $Y'\in {\rm Nbh}_{\overline{y}}^{\rm aff}(Y)$ such that $e_1, \ldots, e_n$ form a complete orthogonal set of idempotent elements of $\Gamma(Y'\times_Y X, \O_{Y'\times_Y X})$. This gives a decomposition of $Y'\times_Y X$ into
\begin{equation}
	\label{RelevementVoisEtaleSchema 4}
Y'\times_Y X\xrightarrow{\sim} \coprod_{i=1}^n X'_i
\end{equation}
such that the fiber product with $Y_{(\overline{y})}\to Y'$ gives back \eqref{RelevementVoisEtaleSchema 2}. Since $Y'\times_Y X$ is étale over $X$, $X'_i$ is an étale neighborhood of $\overline{x}_i$, for each $i$. This proves the lemma.
\end{proof}

\begin{lem}
	\label{RelevementVoisEtaleFormel}
Let $\mathfrak{X}\to \mathfrak{Y}$ be a finite adic morphism of locally noetherian $\mathcal{S}$-formal schemes. Let $\overline{y}$ be a geometric point of $\mathfrak{Y}$ and let $\overline{x}_1, \ldots, \overline{x}_n$ be the geometric points of $\mathfrak{X}$ above $\overline{y}$. Then, there exist an étale neighborhood $\mathfrak{Y}'\to \mathfrak{Y}$ of $\overline{y}$ and, for each $i$, an étale neighborhood $\mathfrak{X}'_i\to \mathfrak{X}$ of $\overline{x}_i$, such that
\begin{equation}
	\label{RelevementVoisEtaleFormel 1}
\mathfrak{Y}'\times_{\mathfrak{Y}}\mathfrak{X}\xrightarrow{\sim}\coprod_{i=1}^n \mathfrak{X}'_i.
\end{equation}
\end{lem}

\begin{proof}
From \ref{RelevementVoisEtaleSchema}, there exist an étale neighborhood $Y'\to \mathfrak{Y}_s$ of $\overline{y}$ and, for each $i$, an étale neighborhood $X'_i\to \mathfrak{X}_s$ of $\overline{x}_i$ such that 
\begin{equation}
	\label{RelevementVoisEtaleFormel 2}
Y'\times_{\mathfrak{Y}_s}\mathfrak{X}_s \xrightarrow{\sim}	\coprod_{i=1}^n X'_i .
\end{equation}
From the equivalence of étale sites recalled in \ref{Equivalence sites étales}, there exist formal étale neighborhoods $\mathfrak{Y}'\to \mathfrak{Y}$ of $\overline{y}$ and $\mathfrak{X}'_i\to \mathfrak{X}$ of $\overline{x}_i$ such that $Y'=\mathfrak{Y}'_s$ and $X'_i=\mathfrak{X}'_{i, s}$. Now \eqref{RelevementVoisEtaleFormel 2} becomes an $\mathfrak{X}_s$-isomorphism of the special fibers of formal étale $\mathfrak{X}$-schemes
\begin{equation}
	\label{RelevementVoisEtaleFormel 3}
(\mathfrak{Y}'\times_{\mathfrak{Y}}\mathfrak{X})_s\xrightarrow{\sim} (\coprod_{i=1}^n \mathfrak{X}'_i)_s,
\end{equation}
which then lifts to \eqref{RelevementVoisEtaleFormel 1} by the aforementioned equivalence of étales sites.
\end{proof}

\begin{lem}
	\label{Produit Locaux}
Let $\mathfrak{X}\to \mathfrak{Y}, \overline{y}$ and $\overline{x}_1,\ldots, \overline{x}_n$ be as in {\rm \ref{RelevementVoisEtaleFormel}}. Then,
\begin{equation}
	\label{Produit Locaux 1}
\varinjlim_{\mathfrak{Y}'\in {\rm Nbh}_{\overline{y}}^{\rm aff}(\mathfrak{Y})} \O_{\mathfrak{X}}(\mathfrak{Y}'\times_{\mathfrak{Y}}\mathfrak{X})\xrightarrow{\sim} \prod_{i=1}^n \O_{\mathfrak{X}, \overline{x}_i}.
\end{equation}
\end{lem}

\begin{proof}
We can assume that $\mathfrak{Y}=\Spf(A)$ and $\mathfrak{X}=\Spf(B)$ are affine. By \ref{RelevementVoisEtaleFormel}, we can also assume that $n=1$ and put $\overline{x}_1=\overline{x}$. By  \ref{Equivalence sites étales}, the canonical homomorphism \eqref{Produit Locaux 1} can also be written as 
\begin{equation}
	\label{Produit Locaux 2}
\varinjlim_{\mathfrak{Y}'_s\in {\rm Nbh}_{\overline{y}}^{\rm aff}(\mathfrak{Y}_s)} \O_{\mathfrak{X}}(\mathfrak{Y}'\times_{\mathfrak{Y}}\mathfrak{X})\to\O_{\mathfrak{X}, \overline{x}}=\varinjlim_{\mathfrak{X}'_s\in {\rm Nbh}_{\overline{x}}^{\rm aff}(\mathfrak{X}_s)} \O_{\mathfrak{X}}(\mathfrak{X}').
\end{equation}
To prove that this is an isomorphism, it is enough to show that the right-hand side colimit of \eqref{Produit Locaux 2} can be taken in the full subcategory of objects of the form $\mathfrak{Y}'_s\times_{\mathfrak{Y}_s}\mathfrak{X}_s$, where $\mathfrak{Y}'_s$ is an object of ${\rm Nbh}_{\overline{y}}^{\rm aff}(\mathfrak{Y}_s)$. Let $\varphi$ be the functor ${\rm Nbh}_{\overline{y}}^{\rm aff}(\mathfrak{Y}_s) \to {\rm Nbh}_{\overline{x}}^{\rm aff}(\mathfrak{X}_s), ~ \mathfrak{Y}'_s\mapsto \mathfrak{Y}'_s\times_{\mathfrak{Y}_s}\mathfrak{X}_s$. It is enough to show that $\varphi$ is cofinal. We let
$\mathfrak{X}''_s$ be an object of $ {\rm Nbh}_{\overline{x}}^{\rm aff}(\mathfrak{X}_s)$ and look for an object $\mathfrak{Y}'_s\in {\rm Nbh}_{\overline{y}}^{\rm aff}(\mathfrak{Y}_s)$ and a morphism $h:\mathfrak{Y}'_s\times_{\mathfrak{Y}_s}\mathfrak{X}_s\to \mathfrak{X}''_s$. We have already seen \ref{RelevementVoisEtaleSchema 2} that
\begin{equation}
	\label{Produit Locaux 3}
\varprojlim_{\mathfrak{Y}'_s\in {\rm Nbh}_{\overline{y}}^{\rm aff}(\mathfrak{Y}_s)}\mathfrak{Y}'_s\times_{\mathfrak{Y}_s}\mathfrak{X}_s =(\mathfrak{Y}_s)_{(\overline{y})}\times_{\mathfrak{Y}_s}\mathfrak{X}_s\xrightarrow{\sim} (\mathfrak{X}_s)_{(\overline{x})}=\varprojlim_{\mathfrak{X}'_s\in {\rm Nbh}_{\overline{x}}^{\rm aff}(\mathfrak{X}_s)}\mathfrak{X}'_s,
\end{equation} 
where the first identification stems from the fact that colimits commute with tensor product. Then, by \cite[8.8.2]{EGA.IV}, there exists an object $\mathfrak{Y}'_s$ of ${\rm Nbh}_{\overline{y}}^{\rm aff}(\mathfrak{Y}_s)$ and an isomorphism $\mathfrak{Y}'_s\times_{\mathfrak{Y}_s}\mathfrak{X}_s\xrightarrow{\sim} (\mathfrak{X}_s)_{(\overline{x})}$ inducing \eqref{Produit Locaux 3}. Post-composing the latter with the canonical projection $(\mathfrak{X}_s)_{(\overline{x})}\to \mathfrak{X}''_s$ yields the desired morphism $h$.
Now suppose that $\mathfrak{Y}''_s$ is an object of ${\rm Nbh}_{\overline{y}}^{\rm aff}(\mathfrak{Y}_s)$, and $f_0, g_0 : \varphi_{\rm aff}(\mathfrak{Y}''_s)\rightrightarrows \mathfrak{X}''_s$ are morphisms of ${\rm Nbh}_{\overline{x}}^{\rm aff}(\mathfrak{X}_s)$. Then, mutatis mutandis, the same argument given in the proof of \ref{Foncteurs cofinaux} yields an equalizer $\varphi(h')$ of $f_0$ and $g_0$, for some morphism $h': Z'\to \mathfrak{Y}''_s$ of ${\rm Nbh}_{\overline{y}}^{\rm aff}(\mathfrak{Y}_s)$. This proves that $\varphi$ is cofinal \cite[I, 8.1.3]{SGA4} and establishes the Lemma.
\end{proof}

\begin{lem}
	\label{Produit Locaux Affine}
We keep the assumptions of {\rm \ref{RelevementVoisEtaleFormel}} and further assume that $\mathfrak{Y}=\Spf(A)$ and $\mathfrak{X}=\Spf(B)$. Then,
\begin{equation}
	\label{Produit Locaux Affine 1}
\O_{\mathfrak{Y}, \overline{y}}\otimes_A B \xrightarrow{\sim} \prod_{i=1}^n \O_{\mathfrak{X}, \overline{x}_i}.
\end{equation}
\end{lem}

\begin{proof}
From \ref{Produit Locaux}, we see that
\begin{equation}
	\label{Produit Locaux Affine 2}
\varinjlim_{\Spf(A')\in {\rm Nbh}_{\overline{y}}^{\rm aff}(\mathfrak{Y})} A' \widehat{\otimes}_A B \xrightarrow{\sim} \prod_{i=1}^n \O_{\mathfrak{X}, \overline{x}_i}.
\end{equation}
Since $B$ is a finite algebra over $A$, $A'\otimes B$ is a finite algebra over the adic ring $\widehat{A'}$, hence $\pi$-adically complete \cite[1.8.25.4 and 1.8.29]{EGR}. It then follows from \cite[Chap. 0, 7.7.1]{EGA.I} that
\begin{equation}
	\label{Produit Locaux Affine 3}
A' \widehat{\otimes}_A B \xrightarrow{\sim} A'\otimes_A B.
\end{equation}
Since colimits commute with tensor products, \eqref{Produit Locaux Affine 2} and \eqref{Produit Locaux Affine 3} yield \eqref{Produit Locaux Affine 1}.
\end{proof}

\begin{lem}
	\label{FreeAdic}
Let $A\to B$ be a finite homomorphism of adic rings over $\O_K$. Assume that $A$ is flat over $\O_K$ and that $B\widehat{\otimes}_{\O_K} \O_{\widehat{\overline{K}}}$ is free of finite rank over $A\widehat{\otimes}_{\O_K} \O_{\widehat{\overline{K}}}$. Then, there exists a finite extension $K' \subseteq \overline{K}$ of $K$ such that $B\otimes_{\O_K} \O_{K'}$ is free of finite rank over $A\otimes_{\O_K} \O_{K'}$.
\end{lem}

\begin{proof}
Let $b_1, \ldots, b_m$ be a basis of $B\widehat{\otimes}_{\O_K} \O_{\widehat{\overline{K}}}$ over $A\widehat{\otimes}_{\O_K} \O_{\widehat{\overline{K}}}$. As colimits commute with quotient and tensor product, it is straightforwardly seen that the $\pi$-adic completion of $\varinjlim_{K'} B\widehat{\otimes}_{\O_K} \O_{K'}$ is canonically isomorphic to $B\widehat{\otimes}_{\O_K} \O_{\overline{K}}\cong B\widehat{\otimes}_{\O_K} \O_{\widehat{\overline{K}}}$. It follows that the canonical map
\begin{equation}
	\label{FreeAdic 1}
\varphi: \varinjlim_{K'} B\widehat{\otimes}_{\O_K} \O_{K'} \to B\widehat{\otimes}_{\O_K} \O_{\widehat{\overline{K}}}
\end{equation} 
has dense image. Hence, there exist a finite extension $K'$ of $K$ and elements $b'_i\in B\widehat{\otimes}_{\O_K} \O_{K'}$ such that $\lvert b_i-\varphi(b'_i)\lvert < \lvert \pi\lvert$. Then, $(\varphi(b'_1), \ldots, \varphi(b'_m))$ is also a basis of $B\widehat{\otimes}_{\O_K} \O_{\widehat{\overline{K}}}$ over $A\widehat{\otimes}_{\O_K} \O_{\widehat{\overline{K}}}$. Let
\begin{equation}
	\label{FreeAdic 2}
\oplus_{i=1}^m (A\widehat{\otimes}_{\O_K} \O_{\widehat{\overline{K}}}) \varphi(b'_i) \xrightarrow{\sim} B\widehat{\otimes}_{\O_K} \O_{\widehat{\overline{K}}}
\end{equation}
be the canonical isomorphism thus defined. We claim that the induced canonical homomorphism 
\begin{equation}
	\label{FreeAdic 3}
\oplus_{i=1}^m (A\widehat{\otimes}_{\O_K} \O_{K'}) b'_i \to B\widehat{\otimes}_{\O_K} \O_{K'}
\end{equation}
is an isomorphism. To prove injectivity, it is enough to show that the canonical homomorphism
\begin{equation}
	\label{FreeAdic 4}
A\widehat{\otimes}_{\O_K} \O_{K'}\to A\widehat{\otimes}_{\O_K} \O_{\widehat{\overline{K}}}
\end{equation}
is injective. It is enough to show that it is injective mod $\pi^n$ for any $n\geq 1$. Since $A$ is flat over $\O_K$, $A/\pi^nA$ is flat over $\O_K/(\pi^n)$. Moreover, $\O_{K'}/(\pi^n)\to \O_{\widehat{\overline{K}}}/(\pi^n)=\O_{\overline{K}}/(\pi^n)$ is injective since $\O_{\overline{K}}\cap K'=\O_{K'}$. Thus \eqref{FreeAdic 3} is injective mod $\pi^n$ for any $n\geq 1$, hence \eqref{FreeAdic 2} is injective. To prove surjectivity, it is enough to show that \eqref{FreeAdic 2} is surjectif mod $\pi$ \cite[1.8.5]{EGR}. Since the residue field of $\O_K$ is algebraically closed, \eqref{FreeAdic 1} and \eqref{FreeAdic 2} coincide mod $\pi$, hence \eqref{FreeAdic 2} is sujective mod $\pi$ too, hence surjective. Finally, as $\O_{K'}$ is a finite free $\O_K$-module, we deduce that $A\widehat{\otimes}_{\O_K} \O_{K'}=A\otimes_{\O_K} \O_{K'}$ and $B\widehat{\otimes}_{\O_K} \O_{K'}=B\otimes_{\O_K} \O_{K'}$. This proves the lemma.
\end{proof}

\subsection{} \label{spécialisation}
We take up again the notation of \ref{ALFE}, set $X=\Spec(A)$ and let $\overline{y}\to \mathfrak{X}_s$ be another geometric point. On the one hand, since the diagram
\begin{equation}
	\label{spécialisation 1}
\xymatrix{ \ar @{} [dr] | {\Box}
(\mathfrak{X}_s)_{(\overline{x})} \ar[r] \ar[d] & \mathfrak{X}_s\ar[d]\\
X_{(\overline{x})} \ar[r] & X
}
\end{equation}
is Cartesian, specialization maps $\overline{y}\rightsquigarrow \overline{x}$ on $X$ and its subscheme $\mathfrak{X}_s$ coincide:
\begin{equation}
	\label{spécialisation 2}
\Hom_{\mathfrak{X}_s}(\overline{y}, (\mathfrak{X}_s)_{(\overline{x})})\cong\Hom_{X}(\overline{y}, X_{(\overline{x})}).
\end{equation}
On the other hand, if $\mathcal{F}$ is a sheaf on $\mathfrak{X}_{s, \textrm{ét}}$, any specialization map $\overline{y}\to (\mathfrak{X}_s)_{(\overline{x})}$ induces a homomorphism
$\mathcal{F}_{\overline{x}}\to \mathcal{F}_{\overline{y}}$. Indeed, we have 
\begin{equation}
	\label{spécialisation 3}
\mathcal{F}_{\overline{x}}=\Gamma\left((\mathfrak{X}_s)_{(\overline{x})}, \mathcal{F}\lvert_{(\mathfrak{X}_s)_{(\overline{x})}}\right)
\end{equation}
and the map $\mathcal{F}_{\overline{x}}\to \mathcal{F}_{\overline{y}}$ is induced by the $\mathfrak{X}$-morphism $\overline{y}\to (\mathfrak{X}_s)_{(\overline{x})}$. In particular, when $\mathcal{F}=\O_{\mathfrak{X}}$, we get a morphism
\begin{equation}
	\label{spécialisation 4}
\Hom_{\mathfrak{X}_s}(\overline{y}, (\mathfrak{X}_s)_{(\overline{x})})\to \Hom_{\O_{\mathfrak{X}}}(\O_{\mathfrak{X}, \overline{x}}, \O_{\mathfrak{X}, \overline{y}}).
\end{equation}

\subsection{} \label{supnorm}
From here until \ref{Normalized integral model} included, the residue field $k$ is not assumed to be algebraically closed, $\overline{K}$ is an algebraic closure of $K$, $\O_{\overline{K}}$ is the integral closure of $\O_K$ in $\overline{K}$ and $\overline{k}$ is its residue field.
For an affinoid $K$-algebra $A$, let us recall that its supremum (semi-) norm $\lvert \cdot \lvert_{\sup}$ is defined as follows. If $f \in A$, then,
\begin{equation}
	\label{supnorm 1}
\lvert f \lvert_{\sup}= {\sup}_{x\in {\rm Sp(A)}} \lvert f(x)\lvert_x ,
\end{equation}
where $\Sp(A)$ is the set of maximal ideals of $A$ and $f(x)$ is the image of $f$ in the finite extension $A/x$ of $K$ and $\lvert \cdot \lvert_x$ is the unique extension of $\lvert \cdot \lvert_K$ to $A/x$. If $A$ is a standard Tate algebra over $K$, then the supremum norm coincides with the usual Gauss norm.

\subsection{}\label{Integral model}
Let $A_K$ be affinoid $K$-algebra. If $A$ is an $\O_K$-algebra which is topologically of finite type over $\O_K$ (hence $\pi$-adically complete) such that $A\otimes_{\O_K}K=A_K$, then the formal scheme $\mathfrak{X}=\Spf(A)$ is a model of the affinoid variety $\mathfrak{X}_{\eta}=\Sp(A_K)$, and $\mathfrak{X}_{\eta}$ coincide with the \textit{rigid fiber} (in the sense of Raynaud) of $\mathfrak{X}$. To construct such a model, let $\rho: K\{ T_1, \ldots, T_n\} \to A_K$ be a surjective homomorphism, set $A_{\rho}=\rho(\O_K\{T_1, \ldots, T_n\})$ and take $A=\{f\in A_{K} ~\lvert ~\lvert f \lvert_{\sup} \leq 1\}$ to be the unit ball.

\begin{lem}[{\cite[Lemma 4.1]{A.S.1}}]
	\label{A.S.1 Lemme 4.1}
We keep the notation of {\rm \ref{Integral model}} and assume that $A_K$ is reduced. Then,
\begin{enumerate}
\item[$(i)$] $A_{\rho}\subseteq A$ and $A$ is the integral closure of $A_{\rho}$ in $A_K$.
\item[$(ii)$] If $A_{\rho}\otimes_{\O_K} k$ is reduced, then $A=A_{\rho}$.
\end{enumerate}
\end{lem}

As $\O_K$ is a complete discrete valuation ring, the unit ball $A$ of the reduced $K$-algebra $A_K$ is topologically of finite type over $\O_K$ \cite[Theorem 1.2]{BLR}; hence, it defines a formal model $\Sp(A_K)$. Even then, $A\otimes_{\O_K} k$ may not be reduced and thus the formation of $A$ may not commute with base change. However, we have the following generalization of a finiteness result of Grauert and Remmert.

\begin{thm}[{\cite[Theorem 3.1]{BLR}}]
 \label{Reduced Fiber Theorem}
 Let $A_K$ be a geometrically reduced affinoid $K$-algebra. Then, there exists a finite $($separable$)$ extension $K'$ of $K$ such that the unit ball $A_{\O_{K'}}$ of $A_K\otimes_K K'$ is topologically of finite type over $\O_{K'}$ and has geometrically reduced special fiber $A_{\O_{K'}}\otimes_{\O_{K'}} k'$, where $k'$ is the residue field of $\O_{K'}$. Moreover, the formation of $A_{\O_{K'}}$ commutes with any finite extension of $K'$.
\end{thm}

\begin{defi} \label{Normalized integral model}
 Let $A_K$ be a geometrically reduced affinoid $K$-algebra. We think of the collection $(\Spf(A_{\O_{K'}}))_{K'}$  of $\O_{K'}$-formal schemes, where $K'$ and $A_{\O_{K'}}$ are as in \ref{Reduced Fiber Theorem}, as a unique model of $\Sp(A_K)$ defined over $\O_{\overline{K}}$ and call it \textit{the normalized integral model of} $\Sp(A_K)$ over $\O_{\overline{K}}$. We say that the normalized integral model is defined over $\O_{K'}$ if the unit ball $A'\subset A_{K'}$ has a geometrically reduced special fiber $A'\otimes_{\O_{K'}} \overline{k}$.
\end{defi}

\subsection{}\label{Norme et Ordre}
We resume the assumptions of \ref{Notations Formelles}.
Let $C\to \mathcal{S}$ be a formal relative curve, i.e. an adic morphism of formal schemes which is of finite type, separated and flat, with special fiber of equidimension $1$. We assume that $C/\mathcal{S}$ is proper and has smooth rigid fiber $C_{\eta}$. We also assume that $C$ contains an admissible open subset isomorphic to a disjoint union $\coprod_{j=1}^n \widehat{\Delta}_j$, where each $\widehat{\Delta}_j$ is isomorphic to the formal torus $\widehat{\mathbb{G}}_{m, \O_K}={\Spf}(\O_K\{T_j, T_j^{-1}\})$ with rigid fiber $\widehat{\Delta}_{j, \eta}=\Delta_j$, and also that $C_s - \coprod_j \widehat{\Delta}_{j, s}$ is a finite set. 
Let $f$ be a rigid analytic function on $\coprod_j \Delta_j$ and define the norm $\lvert f\lvert_j$ to be the sup-norm $\lvert f_{\lvert\Delta_j}\lvert_{\sup}$ on $\Delta_j$. The function $f$ corresponds to a rigid morphism $\coprod_j \Delta_j \to D$, where $D$ is the rigid unit disc over $K$. 
If $f\neq 0$ on each $\Delta_j$, we can write $\lvert f \lvert_j=\lvert c_j\lvert$, for some constant $c_j \in L^{\times}$, where $L$ is some finite extension of $K$. Then, $c_j^{-1} f$ is defined on $\widehat{\Delta}_j \widehat{\otimes}_{\O_K}\O_L$, corresponds to a formal morphism $\widehat{\Delta}_j \widehat{\otimes}_{\O_K}\O_L \to \widehat{\mathbb{A}}^1_{\O_L}$ and thus reduces to a morphism $\widetilde{c_j^{-1} f}: \widehat{\Delta}_{j, s} \to \mathbb{A}^1_k$.
As the $\widehat{\Delta}_{j, s}$ are open in $C_s$ and smooth over $k$, they are dense open subsets of distinct irreducible components $\widetilde{C}_{s, j}$ of the normalization $\widetilde{C}_s$ of $C_s$. Hence, $\widetilde{c_j^{-1} f}$ defines a rational function on $\widetilde{C}_{s, j}$; the divisor of $\widetilde{c_j^{-1}f}$ on $\widetilde{C}_{s, j}$ is independent of $c_j$. Therefore, for a point $\widetilde{y}\in \widetilde{C}_s$ which is in the irreducible component $\widetilde{C}_{s, j}$, we define the \textit{order of $f$ at} $\widetilde{y}$ by
\begin{equation}
	\label{Norme et Ordre 1}
{\rm ord}_{\widetilde{y}}(f)={\rm ord}_{\widetilde{y}}(\widetilde{c_j^{-1} f}).
\end{equation} 

\begin{lem}
	\label{BoschLutke}
We keep the notation of {\rm \ref{Norme et Ordre}}.  
Let $y$ be a point in $C_s$ and denote by $\widetilde{y}_1,\ldots, \widetilde{y}_m$ the points in $\widetilde{C}_s$ above $y$. Let also ${\rm sp}: C_{\eta}\to C_s$ be the specialization map and set $C_+(y)={\rm sp}^{-1}(y)$. Then, we have
\begin{equation}
	\label{BoschLutke 1}
\deg({\rm div}(f)_{\lvert C_+(y)})=\sum_j {\rm ord}_{\widetilde{y}_j}(f).
\end{equation}
\end{lem}

\begin{proof}
This is \cite[Proposition 3.1 and Remark]{BoschLutke} since $\lvert \cdot\lvert_j$ is the norm over the irreducible component $C_{s, j}$ of $C_s$ containing $\widehat{\Delta}_{j, s}$, as defined in \textit{loc. cit.}
\end{proof}

\section{Valuation rings.}\label{Valuations}
\subsection{} \label{remrank}
Let $V$ be a valuation ring with value group $\Gamma_V$ and field of fractions $K$. Then, the group $\Gamma=\Q\otimes_{\Z}\Gamma_V$ is a totally ordered $\Q$-vector space, whose dimension $r(\Gamma_V)$ is called the \textit{rational rank} of $\Gamma_V$. Let $L$ be an algebraic extension of $K$ and $W$ a valuation ring of $L$ extending $V$. By \cite[Chap. VI, § 8, n°1, Proposition 1 and Corollaire 1]{Bourbaki1}, the value groups $\Gamma_V\subseteq \Gamma_W$ have the same rank (or height). Since the quotient $\Gamma_W/\Gamma_V$ is torsion, these value groups have also the same rational rank, thus $\Q\otimes_{\Z}\Gamma_W\cong\Q\otimes_{\Z}\Gamma_V=\Gamma$.

\begin{lem}
	\label{W valuation}
Let $V$ be a henselian valuation ring with field of fractions $K$, $L$ a finite extension of $K$ and $W$ the integral closure of $V$ in $L$. Then, there is a unique valuation on $L$ extending the valuation of $K$, and its valuation ring is $W$.
\end{lem}

\begin{proof}
We know from \cite[Chap. VI, § 8, n°3, Théorème 1 and Remarque]{Bourbaki1} that $W=\cap_{i\in I}V_i$, where the $V_i$ are the valuation rings of all the (inequivalent) valuations of $L$ extending the valuation of $V$, that $W$ is a semi-local ring whose maximal ideals are the intersections $\m_i=W\cap \m(V_i)$, and that $W_{\m_i}=V_i$. Also the cardinal of $I$ is the number of maximal ideals of $W$ (\cite[Chap. VI, § 8, n°6, Proposition 6]{Bourbaki1} or simply by going-up). So it is enough to prove that $W$ is a local ring. Since $W$ is integral over $V$, it is a direct limit of its finite type, hence finite, $V$-subalgebras $W_i$. Since $V$ is assumed to be henselian, each domain $W_i$ is a product of local rings, hence local.
\end{proof}

We use the terminology \textit{weakly unramified extension} for an extension of valuation rings $V\to V'$ with ramification index $1$.

\begin{lem}[{\cite[Lemma 0ASK]{Stacks}}]
	\label{ExtensionValHenselisé}
Let $V$ be a valuation ring, $V^h$ $($resp. $V^{\rm sh})$ the henselization $($resp. a strict henselization$)$ of $V$. Then, the inclusions $V\subseteq V^h\subseteq V^{\rm sh}$ are weakly unramified extensions of valuation rings.
\end{lem}

\begin{defi}
	\label{ZZVR}
	A valuation ring $V$ is called a $\Z^2$-\textit{valuation ring} if its value group $\Gamma_V$ is isomorphic to the lexicographically ordered group $\Z\times\Z$.
\end{defi}

\begin{rem}
If $V$ is a $\Z^2$-valuation ring, then $V$ has height two. Indeed the lexicographically ordered group $\Z\times\Z$ has exactly two isolated subgroups, the trivial subgroup and the second factor. Hence, $V$ has exactly two non zero prime ideals $\p \subsetneq\m$ \cite[Chap. VI, § 4, n°4, Proposition 5]{Bourbaki1}.
\end{rem}

\begin{lem}
\label{LocaliseDVR}
Let $K$ be the field of fractions of a $\Z^2$-valuation ring $V$ whose prime ideals are $(0)\subsetneq \p\subsetneq \m$ and let $v: K^{\times}\to \Gamma_V$ be its valuation map. Then, the localization $V_{\p}$ is a discrete valuation ring whose valuation map is given by the composition
\begin{equation}
	\label{LocaliseDVR 1}
K^{\times}\to  \Gamma_V\to \Gamma_V/H,
\end{equation}
where $H$ is the unique non-trivial isolated subgroup of $\Gamma$. If we choose an isomorphism of ordered groups $\Gamma_V\cong \Z\times \Z$, then this composition is $K^{\times}\to  \Gamma_V\to\Z$, where the second homomorphism is the first projection. 
\end{lem}

\begin{proof} The lemma follows from combining \cite[Chap. VI, § 4, n°1, Proposition 1]{Bourbaki1} and \cite[Chap. VI, § 4, n°3, Proposition 4]{Bourbaki1}, where we also see that the valuation of $V_{\p}$ is given (up to equivalence) by the composition $K^{\times}\to  \Gamma_V\to \Gamma_V/H$.
\end{proof}

\begin{defi}
	\label{AlphaEtBeta}
Let $K$ be the field of fractions of a $\Z^2$-valuation ring $V$ and let $v: K^{\times}\to \Gamma_V$ be its valuation map. Let $\varepsilon_K$ be the minimum element of $\{t\in \Gamma_V ~|~ t>0\}$, i.e a generator of the non-trivial isolated subgroup of $\Gamma_V$. For a given uniformizer $\pi$ of the discrete valuation ring $V_{\p}$, $(v(\pi), \varepsilon_K)$ is an ordered generating family for the abelian group $\Gamma_V$ (\ref{LocaliseDVR 1}). Let $\alpha, \beta:\Gamma_V \to \Z$ be the group homomorphisms characterized respectively by $\alpha(v(\pi))=1$, $\alpha(\varepsilon_K)=0$, and $\beta(v(\pi))=0$, $\beta(\varepsilon_K)=1$. They induce an isomorphism of ordered abelian groups $(\alpha, \beta): \Gamma_V\xrightarrow{\sim} \Z\times \Z$. The composition
\begin{equation}
	\label{AlphaEtBeta 1}
K^{\times}\to \Gamma_V \xrightarrow{(\alpha, \beta)} \Z\times \Z
\end{equation}
is \emph{the normalized $\Z^2$-valuation map of $V$}. We set $v^{\alpha}=\alpha\circ v$ and $v^{\beta}=\beta\circ v$. Notice that, while $v^{\alpha}$ does not depend on the chosen $\pi$, $v^{\beta}$ does. 
\end{defi}

\begin{lem}
	\label{Extension de Z2Valuations}
Let $V$ be a henselian $\Z^2$-valuation ring with field of fractions $K$, $L$ a finite extension of $K$ and $W$ the integral closure of $V$ in $L$. 
Then, the valuation ring $W$  {\rm (cf. \ref{W valuation})} is also a $\Z^2$-valuation ring.
\end{lem}
\begin{proof}
The group $\Gamma_W$ is finitely generated with height and rational rank equal $2$ (\textit{cf.} \ref{remrank}). So the lemma follows from \cite[Chap. VI, § 10, n°3, Proposition 4]{Bourbaki1}.
\end{proof}

\begin{lem}[{\cite[(3.9)]{K1}}]
	\label{Monogene}
Let $V$ be a $\Z^2$-valuation ring, $K$ its field of fractions, $L$ a finite extension of $K$ and $W$ the integral closure of $V$ in $L$. Assume that $W$ is a valuation ring. Let $\m \supsetneq\p \supsetneq (0)$ $($resp. $\m' \supsetneq \p' \supsetneq (0))$ be all distinct prime ideals of $V$ $($resp. $W)$. Assume further that $W/\p'$ is of finite type as a $V/\p$-module. Then, the following conditions are equivalent.
	\begin{enumerate}
	\item[$(i)$] $W$ is a $V$-module of finite type.
		\item[$(ii)$] $W$ is a free $V$-module.
		\item[$(iii)$] $[L:K]=[\kappa(\p'):\kappa(\p)]$.
		\item[$(iii')$] The discrete valuation ring $W_{\p'}$ is weakly unramified over $V_{\p}$.
	\end{enumerate}  
 Moreover, if in addition to these conditions the extension $\kappa(\m')/\kappa(\m)$ is separable, then $W=V[a]$ for any $a$ in $W$ whose image in $\kappa(\p')$ generates $W/\p'$ over $V/\p$ $($such an $a$ exists by {\rm \cite[Chap. III, Proposition 12]{Serre1})}.
\end{lem}

\begin{proof}
The valuation rings $V$ and $V/\p$  have the same maximal ideal $\m$, the same residue field $\kappa(\m)$ and $W/\m W\cong (W/\p')/\m(W/\p')$, so $[(W/\p')/\m(W/\p'): (V/\p)/\m]=[W/\m : V/\m]$. Hence, $(i)\Leftrightarrow (ii) \Leftrightarrow (iii)$ follows from \cite[Chap. VI, § 8, n°5, Théorème 2]{Bourbaki1} applied to the couples $(A, B)=(V/\p, W/\p')$ and $(A,B)=(V,W)$. We also see from $[L:K]=e(W_{\p'}/V_{\p})[\kappa(\p'):\kappa(\p)]$ that $(iii')$ is a rewording of $(iii)$. The remaining assertion follows from Nakayama's lemma.
\end{proof}

\begin{rem}\label{RemCompSep}
 The finiteness condition on $W/\p'$ in \ref{Monogene} is satisfied if the field of fractions of the $\m$-adic completion $\widehat{V/\p}$ of $V/\p$ is separable over $\kappa(\p)$ \cite[Chap. 0, 23.1.7 (i)]{EGA.IV}.
\end{rem}

\begin{lem} \label{RemPdim} 
Let $V$ be a henselian $\Z^2$-valuation ring, $0\subsetneq \p\subsetneq \m$ its prime ideals. Assume that the residue field $\kappa(\p)$ at $\p$ has characteristic $p>0$ and that $[\kappa(\p):\kappa(\p)^p]=p$. Then, the field of fractions $\widehat{\kappa(\p)}$ of the $\m$-adic completion $\widehat{V/\p}$ is a separable extension of $\kappa(\p)$. In particular, by {\rm \ref{RemCompSep}}, $V$ satisfies the hypotheses of {\rm Lemma \ref{Monogene}}.
\end{lem}

\begin{proof}
By \cite[Chap. 0, 21.4.1]{EGA.IV}, if $\kappa(\p)$ has characteristic $p$, then every uniformizer of the discrete valuation ring $V/\p$ (which is also a uniformizer of its completion) is a $p$-basis of $\widehat{\kappa(\p)}$ over $\widehat{\kappa(\p)}^p$. Hence, $\kappa(\p)(\widehat{\kappa(\p)}^p)=\widehat{\kappa(\p)}$, and $\Omega_{\widehat{\kappa(\p)}/\kappa(\p)}=0$, and the separability claim follows from \cite[Chap. 0, 20.6.3]{EGA.IV}. The last statement follows from \ref{RemCompSep}.
\end{proof}

\begin{defi}
	\label{MonogenicIntegral}
Let $V$ be a valuation ring with field of fractions $K$. Let $L$ be a finite extension of $K$ and $W$ a valuation ring of $L$ extending $V$. We say that $W/V$ is a \textit{monogenic integral extension} of valuations rings if $W$ is the integral closure of $V$ in $L$ and $W=V[a]$ for some element $a$ of $W$. 
\end{defi}

\begin{prop}
	\label{Proprietes Objets}
Let $\O_K$ be an excellent henselian discrete valuation ring with field of fractions $K$, maximal ideal $\m_K$ and \emph{algebraically closed} residue field $k$. Denote $S=\Spec(\O_K)$ and $s$ the closed point of $S$.
Let $X\to S$ be a relative curve, i.e. a separated and flat morphism of finite type with relative dimension $1$. Let $\overline{x}\to X$ be a geometric point with image a closed point $x$ of the special fiber $X_s$. Assume that $X$ is normal at $x$ and $X-\{x\}$ is smooth over $S$.
Let $A_{\overline{x}}$ be the étale local ring $\O_{X, \overline{x}}$ and let $\m_{\overline{x}}$ be its maximal ideal. Then,
\begin{enumerate}
\item[$(i)$] $A_{\overline{x}}$ is a normal and excellent two-dimensional Cohen-Macaulay ring.
\item[$(ii)$] The residue field $A_{\overline{x}}/\m_{\overline{x}}$ is isomorphic to $k$ and the quotient $A_{\overline{x}}/\m_K A_{\overline{x}}$ is reduced.
\item[$(iii)$] Let $K'$ be a finite extension of $K$, $S'=\Spec(\O_{K'})$, $s'$ the closed point of $S'$ and $X'/S'$ the base change of $X/S$. Since $k$ is algebraically closed, the special fibers $X_s$ and $X'_{s'}$ are canonically isomorphic. Let $\overline{x}'\to X'$ be the lift of $\overline{x}$. Then, $X'$ is normal at $x'$ and $X'-\{x'\}$ is smooth over $S'$. Moreover $A'_{\overline{x}'}$ satisfies
\begin{equation}
	\label{Proprietes Objets 1}
A'_{\overline{x}'}\cong A_{\overline{x}}\otimes_{\O_K} \O_{K'}
\end{equation}
\end{enumerate}
\end{prop}
\begin{proof}
$(i)$ Since $S$ is excellent and $X$ is of finite type over $S$, $X$ is excellent, hence $\O_{X, x}$ is also excellent. Then, $A_{\overline{x}}$ is excellent \cite[18.7.6]{EGA.IV}. By the permanence properties of strict henselization, $A_{\overline{x}}$ is a two dimensional noetherian normal local ring, hence an integrally closed domain. So $A_{\overline{x}}$ is Cohen-Macaulay \cite[Discussion below 16.5.1]{EGA.IV}.

$(ii)$ The residue field $A_{\overline{x}}/\m_{\overline{x}}$ is also the residue field of $\O_{X, x}$ which is an algebraic extension of $k$, hence is $k$. The assumption that $X-\{x\}$ is smooth implies  that $A_{\overline{x}}/\m_K A_{\overline{x}}$ is reduced.

$(iii)$ Since $X-\{x\}$ is smooth over $S$, $X'-\{x'\}$ is smooth over $S'$ by base change. The ring $\O_{X, x}\otimes_{\O_K}\O_{K'}$ is local because its maximal ideals are the closed points of
\begin{equation}
	\label{Proprietes Objets 2}
\Spec\left((\O_{X, x}\otimes_{\O_K}\O_{K'})/\m_{K'}\right) \xrightarrow{\sim} \Spec(\O_{X, x}/\m_K).
\end{equation}
Hence, it follows that
\begin{equation}
	\label{Proprietes Objets 3}
\O_{X', x'}\cong\O_{X, x}\otimes_{\O_K}\O_{K'}.
\end{equation}
By the same argument, the ring $A_{\overline{x}}\otimes_{\O_K} \O_{K'}$ is also local; and, being a finite algebra over the henselian ring $A_{\overline{x}}$, it is also henselian, which proves \eqref{Proprietes Objets 1}. The normality of  $\O_{X', x'}$ follows from Serre's criterion \cite[IV.D.4, Théorème 11]{Serre3} as follows. Being smooth, $X'-\{x'\}$ is regular; if $y'$ is a height 1 prime ideal of $\O_{X', x'}$, it is also a codimension $1$ point of $X'-\{x'\}$ and the localization of $\O_{X', x'}$ at $y$ is $\O_{X', y'}$ which is regular. Thus $\O_{X', x'}$ is $(R_1)$. It is also $(S_2)$ because it has the same depth as $\O_{X, x}$, as shown by the following argument. Let$f_1,\ldots, f_r$ be generators of the maximal ideal $\m_x$ of $\O_{X, x}$ and denote by $\m_{x'}$ the maximal ideal of $\O_{X', x'}$. Then, we have $V((f_1,\ldots, f_r)\O_{X', x'})= V(m_{x'})$ \eqref{Proprietes Objets 3} and thus, by \cite[II.5]{SGA2}, the local cohomology group $H^q_{\m_{x'}}(\O_{X', x'})$ is the $q$-th cohomology of the complex
\[C^{\bullet}(\O_{X', x'}, f_{\bullet}): 
0\to \O_{X', x'} \to \prod_i \O_{X', x'}[\frac{1}{f_i}]\to \prod_{i<j}\O_{X', x'}[\frac{1}{f_i f_j}]\to \cdots\to \O_{X', x'}[\frac{1}{f_1\cdots f_r}]\to 0.\]
By \eqref{Proprietes Objets 3} and faithful flatness of $\O_{K'}$ over $\O_K$, $C^{\bullet}(\O_{X', x'}, f_{\bullet})$ is isomorphic to $C^{\bullet}(\O_{X, x}, f_{\bullet})\otimes_{\O_K}\O_{K'}$ and this induces the isomophism
\begin{equation}
	\label{Proprietes Objets 4}
H^q_{\m_{x}}(\O_{X, x})\otimes_{\O_K}\O_{K'} \xrightarrow{\sim} H^q_{\m_{x'}}(\O_{X', x'}).
\end{equation}
We conclude by \cite[III.3.4]{SGA2} (and faithful flatness of $\O_K\to \O_{K'}$) that $\O_{X, x}$ and $\O_{X', x'}$ have the same depth (which is $2$ since $\O_{X, x}$ is Cohen-Macaulay).
\end{proof}

\subsection{} \label{(V)}
Let $\O_K$ be an excellent henselian discrete valuation ring with field of fractions $K$ and \emph{algebraically closed} residue field $k$. Set $S=\Spec(\O_K)$, with closed $s$, and let $X/S$ be a relative curve. Let $\overline{x}=\Spec(k)\to X$ be a geometric point with closed image in $X_s$. Assume that the couple $(X/S, \overline{x})$ satisfies the following property
\begin{itemize}
	\item[] {\hfil $(P)$ \quad $X$ is normal at $x$ and $X-\{x\}$ is smooth over $S$.}
\end{itemize}
Let $X_{(\overline{x})}=\Spec(A_{\overline{x}})$ be the strict localization of $X$ at $\overline{x}$. Let $\overline{y}$ be a geometric generic point of $X_s$ and let $\overline{y}\rightsquigarrow \overline{x}$ be a specialization map, i.e an $X$-morphism $\overline{y}\to X_{(\overline{x})}$. Its image corresponds to a height $1$ prime ideal $\p$ of $A_{\overline{x}}$, and it factors through an $X$-morphism $X_{(\overline{y})}\to X_{(\overline{x})}$. The corresponding homomorphism $A_{\overline{x}}\to \O_{X, \overline{y}}$ factors as $A_{\overline{x}}\to (A_x)_{\p}\to \O_{X, \overline{y}}$ which induces an isomorphism $(A_{\overline{x}})_{\p}^{\rm sh}\xrightarrow{\sim} \O_{X, \overline{y}}$ \cite[VIII, 7.6]{SGA4}.
It is clear that $(A_{\overline{x}})_{\p}$ is a discrete valuation ring by \ref{Proprietes Objets} $(i)$. We define $V_X(\overline{y}\rightsquigarrow \overline{x})$ to be the subring of $(A_{\overline{x}})_{\p}$ consisting of elements whose images in the residue field $\kappa(\p)$ belong to the normalization of $A_{\overline{x}}/\p$ in $\kappa(\p)$. Since $A_{\overline{x}}$ is excellent (\ref{Proprietes Objets} $(i)$), it is a universally Japanese ring \cite[7.8.3 (vi)]{EGA.IV}. Thus $A_{\overline{x}}/\p$ is Japanese. 
The normalization of $A_{\overline{x}}/\p$ is therefore a finite algebra over the henselian ring $A_{\overline{x}}/\p$ and is also an integral domain, hence it is a discrete valuation ring. Therefore $V_X(\overline{y}\rightsquigarrow \overline{x})$ is a normalized $\Z^2$-valuation ring of ${\rm Frac}(A_{\overline{x}})$ by \ref{Conrad} below.

\begin{lem} \label{Conrad}
Let $L$ be a valuation field, with valuation ring $R$, maximal ideal $\p$, residue field $F$ and value group $\Gamma_R$. We assume that $F$ is also a valuation field with valuation ring $\overline{R}$ and value group $\Gamma_{\overline{R}}$. We put
\begin{equation}
	\label{Conrad 1}
V=\{x\in R~\lvert ~ (x ~ {\rm mod} ~ \p) \in \overline{R}\}.
\end{equation}
\begin{itemize}
\item[(i)] The subset $V$ of $R$ is a valuation ring of $L$ whose value group $\Gamma_V$ contains $\Gamma_{\overline{R}}$ as an isolated subgroup, and we have an order-preserving short exact sequence of ordered abelian groups
\begin{equation}
	\label{Conrad 2}
0\to \Gamma_{\overline{R}}\to\Gamma_V\to\Gamma_R\to 0.
\end{equation}
\item[(ii)] If $R$ is a discrete valuation ring, the choice of a uniformizer $\pi$ of $R$ induces a canonical splitting of {\rm \eqref{Conrad 2}}
\begin{equation}
	\label{Conrad 3}
\Gamma_V\cong \Z\oplus \Gamma_{\overline{R}}
\end{equation}
which is an isomorphism of ordered abelian groups.
\item[(iii)] Assume that $R$ and $\overline{R}$ are discrete valuation rings with normalized valuation maps $v_R: L^{\times}\to \Z$ and $v_{\overline{R}}: F^{\times}\to \Z$ respectively. Then, $V$ is a $\Z^2$-valuation ring and $\p$ is its height $1$ prime ideal. For a given uniformizer $\pi$ of $R$, the normalization map of $V$ induced by \eqref{Conrad 3} is
\begin{equation}
	\label{Conrad 4}
v: L^{\times}\to \Z^2, \quad x\mapsto (v_R(x), v_{\overline{R}}(x\pi^{-v_R(x)} \mod \p))
\end{equation}
and it coincides with the normalized $\Z^2$-valuation map of $V$ induced by $\pi$ \eqref{AlphaEtBeta 1}.
\end{itemize}
\end{lem}

\begin{proof}
(i)\quad Clearly $V$ is a subring of $R$. If $a$ and $b$ are elements of $R$ with $b$ non zero, then $a/b=ac/bc$ for any non zero element $c$ of $\p$, and $\p\subsetneq V$ (and is clearly a prime ideal of $V$). Thus, the field of fractions of $V$ is $L$. Now let $x\in L^{\times}$. We suppose $x$ is not in $R'$ and show that $x^{-1}$ is in $V$. If $x$ is not in $R$, then $x^{-1}$ is a non unit element of $R$, i.e an element of $\p\subsetneq V$. If $x$ is in $R-V$, then it a unit in $R$ and $\overline{x}=x$ mod $\m$ is not in $\overline{R}$, so $\overline{x}^{-1}=\overline{x^{-1}}$ is in $\overline{R}$, hence $x^{-1}$ is in $V$. 
This proves that $V$ is a valuation ring of $L$. 
The reduction map $R^{\times}\to F^{\times}$ sends $V^{\times}$ in $\overline{R}^{\times}$ and the induced quotient map $R^{\times}/V^{\times}\to F^{\times}/\overline{R}^{\times}$ is clearly an isomorphism of ordered groups. Hence, we have an order-preserving injective homomorphism $F^{\times}/\overline{R}^{\times}=\Gamma_{\overline{R}}\hookrightarrow \Gamma_V=L^{\times}/V^{\times}$. Then, by  \cite[VI, § 4, n°3, Remarque]{Bourbaki1}, the induced short exact sequence
\eqref{Conrad 2} is order-preserving.

(ii) \quad If $R$ is a discrete valuation ring with uniformizer $\pi$, then the group homomorphism $\Z\to \Gamma_{R'}$ sending $1$ to $\pi$ gives the splitting $\Gamma_V\cong \Z\oplus \Gamma_{\overline{R}}$ of \eqref{Conrad 2}. 

(ii) \quad By the construction of $V$, we have $V/\p=\overline{R}$. Therefore, if, for a totally ordered group $G$, ${\rm ht}(G)$ denotes the height of $G$, i.e. the number of its isolated proper subgroups, \cite[VI, \S, n°4, Prop. 5]{Bourbaki1} yields ${\rm ht}(\Gamma_V)={\rm ht}(\Gamma_R) + {\rm ht}(\Gamma_{\overline{R}})$. In particular, if $R$ and $\overline{R}$ are discrete valuation rings, then ${\rm ht}(\Gamma_V)=2={\rm rk}(\Gamma_V)$ \cite[V, \S 10, n° 2, Prop. 3, Cor.]{Bourbaki1} and thus $\Gamma_V$ is isomorphic to $\Z^2$ with the lexicographic order \cite[V, \S 10, n° 2, Prop. 4]{Bourbaki1}. Also, as a prime of $V$, $\p$ is clearly of height $1$. That the normalized valuation of $V$ is the map \eqref{Conrad 4} follows from the definition of $V$ \eqref{Conrad 1}; combined with the splitting \eqref{Conrad 3}, it yields the last claim since $v(\pi)=(1, 0)$ and $v(\varepsilon_L)=(0, 1)$ (notation of \ref{AlphaEtBeta}).
\end{proof}

\subsection{}\label{Category CK}
Let $\O_K$ be an excellent henselian discrete valuation ring with field of fractions $K$ and \emph{algebraically closed} residue field $k$.  Following Kato \cite[5.5]{K1}, we denote by $\mathcal{C}_K$ the category whose objects are the rings $A$ such that $A$ is isomorphic over $\O_K$ to $A_{\overline{x}}$ for some couple $(X/S, \overline{x})$ satisfying property $(P)$ in \ref{(V)}, and whose morphisms are finite $\O_K$-homomorphisms inducing separable extensions of fractions fields. In particular, morphisms in $\mathcal{C}_K$ are injective local homomorphisms.

\begin{prop}
	\label{Proprietes formelles objets}
Let $\O_K$ be a complete discrete valuation ring with field of fractions $K$, maximal ideal $\m_K$ and \emph{algebraically closed} residue field $k$. Let $\mathcal{S}$ be $\Spf(\O_K)$, $s$ its unique point $\pi$ a uniformizer of $\O_K$. Let $\mathfrak{X}/\mathcal{S}$ be a formal relative curve {\rm (see \ref{Norme et Ordre})} and $\overline{x}$ a geometric point of $\mathfrak{X}$ with image a closed point $x$ of the special fiber $\mathfrak{X}_s$. Assume that $\mathfrak{X}-\{x\}$ is smooth over $\mathcal{S}$ and $\mathfrak{X}$ is normal at $x$.
Let $A_{\overline{x}}$ be the formal étale local ring $\O_{\mathfrak{X}, \overline{x}}$ {\rm \eqref{ALFormelEtale 1}} and let $\m_{\overline{x}}$ be its maximal ideal. Then,
\begin{enumerate}
\item[$(i)$] The local ring $A_{\overline{x}}$ is normal and two-dimensional, thus Cohen-Macaulay.
\item[$(ii)$] The residue field $A_{\overline{x}}/\m_{\overline{x}}$ is isomorphic to $k$ and the ring $A_{\overline{x}}/\m_K A_{\overline{x}}$ is reduced and excellent.
\item[$(iii)$] Let $K'$ be a finite extension of $K$, $\mathcal{S}'=\Spf(\O_{K'})$; $s'$ the unique point of $\mathcal{S}'$ and $\mathfrak{X}'/\mathcal{S}'$ the formal base change of $\mathfrak{X}/\mathcal{S}$. Since $k$ is algebraically closed, the special fibers $\mathfrak{X}_s$ and $\mathfrak{X}'_{s'}$ are canonically isomorphic. Let $x'\in \mathfrak{X}'_{s'}$ be the image of $x$. Then, the couple $(\mathfrak{X}'/\mathcal{S}', x')$ satisfies the property $(P)$. Moreover, the $\m_K$-adic completion $\widehat{A'_{\overline{x}'}}$ of  $A'_{\overline{x}'}$ satisfies
\begin{equation}
	\label{Proprietes formelles objets 1}
\widehat{A'_{\overline{x}'}}\cong A_{\overline{x}}\widehat{\otimes}_{\O_K} \O_{K'},
\end{equation}
where $\widehat{\otimes}$ denotes the $\m_K$-completed tensor product.
\end{enumerate}
\end{prop}

\begin{proof}
$(i)$ We can assume that $\mathfrak{X}=\Spf(A)$ is affine, where $A$ is an $\m_K$-adic $\O_K$-algebra which is topologically of finite type \cite[10.13.4]{EGA.I}. Then, $x$ corresponds to an open prime ideal $\p$ of $A$. On the one hand, since $A/\m_K$ is of finite type over the excellent ring $\O_K$, it is an excellent ring \cite[7.8.3 (ii)]{EGA.IV}. Since $A$ is also $\m_K$-adically complete, it follows from a result of Gabber \cite[I, Théorème 9.2]{TdeGabber} that $A$ is itself quasi-excellent. Hence, $A_{\p}$ is also quasi-excellent \cite[I, Théorème 5.1]{TdeGabber} and thus its strict henselization $A_{\p}^{\rm sh}$ is quasi-excellent \cite[I, Théorème 8.1 (iii)]{TdeGabber}, hence excellent \cite[I, Corollaire 6.3 (ii)]{TdeGabber}. On the other hand, since $A_{\p}\to \O_{\mathfrak{X}, x}$ is faithfully flat (\ref{Completion formelle}) and $\O_{\mathfrak{X}, x}$ is normal, $A_{\p}$ is also normal by \cite[6.5.4]{EGA.IV}. Hence, $A_{\p}^{\rm sh}$ is a normal and excellent local ring. As the $\m_K$-adic completion $\widehat{A_{\p}^{\rm sh}}\to \widehat{\O_{\mathfrak{X}, \overline{x}}}$ is an isomorphism (\ref{ALFE}), it thus follows that $\widehat{\O_{\mathfrak{X}, \overline{x}}}$ is normal \cite[7.8.3 (v)]{EGA.IV} and excellent \cite[I, 9.1 (i)]{TdeGabber}. Now, as $\O_{\mathfrak{X}, \overline{x}}\to \widehat{\O_{\mathfrak{X}, \overline{x}}}$ is flat and local, hence faithfully flat, we conclude that $A_{\overline{x}}=\O_{\mathfrak{X}, \overline{x}}$ is also normal \cite[6.5.4]{EGA.IV}. By \eqref{ALFE 2}, $\O_{\mathfrak{X}, \overline{x}}/\m_K\O_{\mathfrak{X}, \overline{x}}\xrightarrow{\sim} \O_{\mathfrak{X}_s, \overline{x}}$; hence, $\dim(A_{\overline{x}})=\dim(\O_{\mathfrak{X_s}, \overline{x}}) + 1=2$ \cite[Chap. 0, 16.3.4]{EGA.IV}.
It then follows from \cite[Chap. 0, discussion below 16.5.1]{EGA.IV} that $A_{\overline{x}}$ is also a Cohen-Macaulay ring.

$(ii)$ We can assume that $\mathfrak{X}=\Spf(A)$ is affine. From \ref{ALFE}, we see that the residue field $\kappa(\overline{x})$ of $A_{\overline{x}}$ is algebraic over the residue field of $A_{\p}$. The latter is $k$ since $x$ is a closed point, hence $\kappa(\overline{x})\cong k$. From the definition \eqref{ALFormelEtale 1}, we see that 
\begin{equation}
	\label{Proprietes formelles objets 2}
A_{\overline{x}}/\m_K A_{\overline{x}}\cong A_{\overline{x}}\otimes_{\O_K} k \xrightarrow{\sim} \varinjlim_{\Spf(B)\in {\rm Nbh}_{\overline{x}}^{\rm aff}(\mathfrak{X})} B/\m_K B.
\end{equation}
Since $\Spf(B)$ is étale over $\mathfrak{X}$ which is smooth outside $x$, $\Spf(B)$ is smooth outside the inverse image of $x$. We deduce that $B/\m_K B$ is reduced, hence $A_{\overline{x}}/\m_K A_{\overline{x}}$ is also reduced. From \eqref{ALFE 2}, we also have $A_{\overline{x}}/\m_K A_{\overline{x}}\xrightarrow{\sim} \O_{\mathfrak{X}_s, \overline{x}}$. Since $\mathfrak{X}_s$ is of finite type over over $k$ \cite[10.13.1]{EGA.I}, its local ring $\O_{\mathfrak{X}_s, x}$ is excellent, hence $A_{\overline{x}}/\m_K A_{\overline{x}}$ is excellent \cite[18.7.6]{EGA.IV}.

$(iii)$ Since $\mathfrak{X}-\{x\}/\mathcal{S}$ is smooth, so is $\mathfrak{X}'-\{x'\}/\mathcal{S}'$. From \eqref{Changement sites} we have
\begin{equation}
	\label{Proprietes formelles objets 3}
A'_{\overline{x}'}=\varinjlim_{\mathcal{U}'_s\in  {\rm Nbh}_{\overline{x'}}^{\rm aff}(\mathfrak{X}'_s)} \O_{\mathfrak{X}'}(\mathcal{U}'_s ).
\end{equation}
Now from the isomorphism of special fibers $\mathfrak{X}'_s \simeq\mathfrak{X}_s$, we deduce
\begin{equation}
	\label{Proprietes formelles objets 4}
A'_{\overline{x}'}\cong \varinjlim_{\mathcal{U}\in  {\rm Nbh}_{\overline{x}}^{\rm aff}}\left(\O_{\mathfrak{X}} (\mathcal{U}) \widehat{\otimes}_{\O_K}\O_{K'}\right).
\end{equation}
Completing $\m_K$-adically (and using the fact that colimits commute with quotients and tensor products), we get \eqref{Proprietes formelles objets 1}.
\end{proof}

\subsection{} \label{(V) formel}
Let $\O_K$ be a complete discrete valuation ring with field of fractions $K$ and \emph{algebraically closed} residue field $k$, let $\pi$ be a uniformizer. Let $\mathfrak{X}/\mathcal{S}$ be a formal relative curve and $\overline{x}$ a geometric point of $\mathfrak{X}$ with image a closed point $x$ of the special fiber  $\mathfrak{X}_s$. 
Assume that the couple $(\mathfrak{X}/\mathcal{S}, \overline{x})$ satisfies the following property
\begin{itemize}
\item[] {\hfil $(P)$\quad $\mathfrak{X}-\{x\}$ is smooth over $\mathcal{S}$ and $\mathfrak{X}$ is normal at $x$.}
\end{itemize}
Let $\overline{y}$ be a geometric generic point of $\mathfrak{X}_s$. Then, the noetherian local ring $\O_{\mathfrak{X},\overline{y}}$ is normal by the same argument for \ref{Proprietes formelles objets} $(i)$, with maximal ideal $\p_{\overline{y}}=\pi \O_{\mathfrak{X},\overline{y}}$. Indeed, modulo $\pi$, it is the field $\O_{\mathfrak{X}_s, \overline{y}}$ \eqref{ALFE 2}. Hence, it is a discrete valuation ring. Let $\overline{y}\to (\mathfrak{X}_s)_{(\overline{x})}$ be a specialization map. It induces a homomorphism $A_{\overline{x}}\to \O_{\mathfrak{X},\overline{y}}$ \eqref{spécialisation 2}. The inverse image $\p$ of $\p_{\overline{y}}$ is a height $1$ prime ideal of $A_{\overline{x}}$. Since $A_{\overline{x}}$ is normal (\ref{Proprietes formelles objets} $(i)$), we hence see that $(A_{\overline{x}})_{\p}$ is a discrete valuation ring. Moreover, $A_{\overline{x}}/\p$ is Japanese since it is a quotient of the excellent ring $A_{\overline{x}}/\m_K A_{\overline{x}}$ (\ref{Proprietes formelles objets} $(ii)$). Hence, the normalization of $A_{\overline{x}}/\p$ in $\kappa(\p)$ is a finite algebra over $A_{\overline{x}}/\p$. We define $V_{\mathfrak{X}}(\overline{y}\rightsquigarrow \overline{x})$ to be the subring of $(A_{\overline{x}})_{\p}$ consisting of elements whose images in the residue field $\kappa(\p)$ belong to this normalization. By the same argument as in \ref{(V)}, we conclude that $V_{\mathfrak{X}}(\overline{y}\rightsquigarrow \overline{x})$ is a normalized $\Z^2$-valuation ring.

\begin{rem}
	\label{(V) Fonctoriel}
The construction of $V_X(\overline{y}\rightsquigarrow \overline{x})$ in \ref{(V)} (resp. $V_{\mathfrak{X}}(\overline{y}\rightsquigarrow \overline{x})$ in \ref{(V) formel}) is functorial in the following sense. Let $(X/S, \overline{x})$ and $(X'/S, \overline{x}')$ (resp. $(\mathfrak{X}/\mathcal{S}, \overline{x})$ and $(\mathfrak{X}'/\mathcal{S}, \overline{x}')$) be couples as in \ref{(V)} (resp. \ref{(V) formel}) satisfying property $(P)$ therein, $\overline{y}\to X_s$ and $\overline{y}'\to X'_s$ (resp. $\overline{y}\to \mathfrak{X}_s$ and $\overline{y}'\to \mathfrak{X}'_s$) geometric generic points, $\overline{y}\rightsquigarrow \overline{x}$ and $\overline{y}'\rightsquigarrow \overline{x}'$ specialization maps. Let $X'\to X$ (resp. $\mathfrak{X}'\to \mathfrak{X}$) be an $S$-morphism (resp. an adic $\mathcal{S}$-morphism) compatible with the geometric points  and the specialization maps in the sense that we have the following commutative diagram
\begin{equation}
	\label{(V) Fonctoriel 1}
\xymatrix{
\overline{y}' \ar[r] \ar[d] & X'_{(\overline{x}')} \ar[r] \ar[d] & X' \ar[d] \\
\overline{y} \ar[r] & X_{(\overline{x})} \ar[r] & X.
}
\end{equation}
(resp. a similar diagram with $\mathfrak{X}_s$ and $\mathfrak{X}'_s$ in lieu of $X$ and $X'$ respectively). Then, we get an extension of $\Z^2$-valuation rings $V_X(\overline{y}\rightsquigarrow \overline{x})\to V_{X'}(\overline{y}'\rightsquigarrow \overline{x}')$ (resp $V_{\mathfrak{X}}(\overline{y}\rightsquigarrow \overline{x})\to V_{\mathfrak{X}'}(\overline{y}'\rightsquigarrow \overline{x}')$.
\end{rem}

\subsection{} \label{Category CK formelle}
Let $\O_K$ be a complete discrete valuation ring with field of fractions $K$ and \emph{algebraically closed} residue field $k$. We denote by $\widehat{\mathcal{C}}_K$ the category whose objects are the rings $A$ such that $A$ is isomorphic over $\O_K$ to $A_{\overline{x}}$ for some couple $(\mathfrak{X}/S, \overline{x})$ satisfying property $(P)$ in \ref{(V) formel}, and whose morphisms are finite $\O_K$-homomorphisms inducing separable extensions of fractions fields. In particular, morphisms in $\widehat{\mathcal{C}}_K$ are injective local homomorphisms.

\subsection{}  \label{(V) Independant y}
We keep the notation of \ref{Category CK} (resp. \ref{Category CK formelle}).
Let $A$ be an object of $\mathcal{C}_K$ (resp. $\widehat{\mathcal{C}}_K$) with field of fractions $\K$ and $\p$ a height $1$ prime ideal of $A$ above $\m_K$.  The localization map $A\to A_{\p}$ induces a homomorphism $A/\m_K A \to (A_{\p}/\m_K A_{\p})^{\rm sh}$, where the strict henselization is with respect to a separable closure $\kappa(\p)^{\rm sep}$ of the residue field $\kappa(\p)$ of $A_{\p}$, and $\p$ corresponds to a minimal prime ideal of $A/\m_K A$. This yields a geometric generic point and a specialization map
\begin{equation}
	\label{(V) Independant y 1}
\overline{y}=\Spec(\kappa(\p)^{\rm sep})\to \Spec\left((A_{\p}/\m_K A_{\p})^{\rm sh}\right)\to \Spec(A/\m_K A).
\end{equation}
Indeed, if $A$ is isomorphic to $A_{\overline{x}}$ for some couple $(X/S, \overline{x})$ (resp. $(\mathfrak{X}/\mathcal{S}, \overline{x})$) satisfying property $(P)$ in \ref{(V)} (resp. \ref{(V) formel}), then \eqref{(V) Independant y 1} translates as $\overline{y}\to (X_s)_{(\overline{x})} \to X_s$ (resp. $\overline{y}\to (\mathfrak{X}_s)_{(\overline{x})}\to \mathfrak{X}_s$). In conclusion, the construction of $V_{X}(\overline{y}\rightsquigarrow \overline{x})$ in \ref{(V)} (resp. $V_{\mathfrak{X}}(\overline{y}\rightsquigarrow \overline{x})$ in \ref{(V) formel}) depends only on the choice of a couple $(A, \p)$, where $A$ is an object of $\mathcal{C}_K$ (resp. $\widehat{\mathcal{C}}_K$) corresponding to $(X/S, \overline{x})$ (resp. $(\mathfrak{X}/\mathcal{S}, \overline{x})$) and  $\p$ a height $1$ prime ideal of $A$ above $\m_K$. For such a couple $(A, \p)$, the normalized $\Z^2$-valuation ring obtained is denoted by $V_A(\p)$, its field of fractions by $\K_{\p}=\K$ and its valuation map by $v_{\p}$. By Lemma \ref{ExtensionValHenselisé}, the henselization $V^h_A (\p)$ of $V_A (\p)$ is also a $\Z^2$-valuation ring; let $\K^h$ be its field of fractions. If $\pi$ is a uniformizer of $\O_K$, then its image by $\O_K\to A_{\p}=(V_A(\p))_{\p}$ is a uniformizer of $(V_A(\p))_{\p}$, and, coupled with $\varepsilon_{\K}$ (\ref{AlphaEtBeta}), it gives an isomorphism $\Gamma_{V_A(\p)}\xrightarrow{\sim} \Z^2$ (\ref{AlphaEtBeta}) and thus induces a valuation map $\K^{h\times} \to \Z^2$ (\ref{ExtensionValHenselisé}) which coincides with the normalized valuation map of $V^h_A (\p)$ \eqref{AlphaEtBeta 1}. We note that, as the chosen uniformizer $\pi$ comes from $K$, these normalized valuation maps don't depend on $\pi$.

\subsection{} \label{(V) KatoGeneral}
We keep the notation of \ref{Category CK} (resp. \ref{Category CK formelle}) and let $A$ be an object of $\mathcal{C}_K$ (resp. $\widehat{\mathcal{C}}_K$). By \ref{Proprietes Objets} $(i)$ (resp. \ref{Proprietes formelles objets} $(i)$), the ring $A_K=A\otimes_{\O_K} K$ is a Dedekind domain. By \ref{Proprietes Objets} $(ii)$ (resp \ref{Proprietes formelles objets} $(ii)$), $A_0=A/\m_KA$ is a reduced ring. We denote by $\widetilde{A_0}$ the normalization of $A_0$ in its total ring of fractions. We denote by $\delta(A)$ the dimension of the quotient  $\widetilde{A_0}/A_0$ of $k$-vector spaces.

For a finite extension $K'/K$, $A'=A\otimes_{\O_K} \O_{K'}$ (resp. $A\widehat{\otimes}_{\O_K} \O_{K'}$) is an object of $\mathcal{C}_{K'}$  \eqref{Proprietes Objets 1} (resp. $\widehat{\mathcal{C}}_{K'}$ (resp. \eqref{Proprietes formelles objets 1}) and we have 
\begin{equation}
	 \label{(V) KatoGeneral 1}
	A/\m_K A\xrightarrow{\sim} A'/\m_{K'} A', \quad {\rm hence}\quad \delta(A)=\delta(A').
\end{equation} 

We denote by $P(A)$ the set of height $1$ prime ideals of $A$, by  $P_s (A)$ the subset of $P(A)$ of prime ideals above $\m_K$ and by $P_{\eta}(A)$ the complement $P(A)-P_{s}(A)$ (of primes above $0$). Since $P_s(A)$ corresponds to the generic points of $\Spec(A_0)$, it is a finite set. We note also that $P_{\eta}(A)$ identifies with the set of maximal ideals of the Dedekind domain $A_K=A_{(0)}$. For $\p\in P_{\eta}(A)$, we denote by ${\rm ord}_{\p}$ the discrete valuation defined by $A_{\p}=(A_K)_{\p}$ and note that the residue field $\kappa(\p)$ is a finite extension of $K$.

Let $\varphi: A\to B$ be a morphism in $\mathcal{C}_K$ (resp. $\widehat{\mathcal{C}}_K$). For $\p\in P_s(A)$ and $\q\in P_s (B)$ above $\p$, $\varphi$ induces an extension of $\Z^2$-valuation rings $V_A (\p)\to V_B (\q)$.

\begin{lem}
	\label{V^h MonogenicLibre}
The above extension induces a monogenic integral extension of $\Z^2$-valuation rings $V^h_B (\q)/V^h_A (\p)$ {\rm (\ref{MonogenicIntegral})}. Moreover, the finitely generated $V^h_A (\p)$-module $V^h_B (\q)$ is free.
\end{lem}

\begin{proof}
By \cite[Thm 17.17]{Endler}, the extension of fields of fractions ${\rm Frac}(V^h_A (\p))\to {\rm Frac}(V^h_B (\q))$ is finite. Since $V^h_A (\p)$ is henselian by \ref{W valuation}, we deduce that $V^h_B (\q)$ is its integral closure in ${\rm Frac}(V^h_B (\q))$. By Lemma \ref{ExtensionValHenselisé}, $V^h_B (\q)/V^h_A (\p)$ is an extension of normalized $\Z^2$-valuation rings.
Since $B$ is finite over $A$, $B/\q$ is also finite over $A/\p$. As seen in \ref{(V)} (resp. \ref{(V) formel}), $A/\p$ and $B/\q$ are Japanese rings, and thus $A/\p \to \widetilde{A/\p}$ and $B/\q \to \widetilde{B/\q}$ are finite. Hence, $\widetilde{A/\p}\subseteq \widetilde{B/\q}$ is a finite extension of henselian rings. Now, since we have 
\begin{equation}
	\label{V^h MonogenicLibre 1} 
V^h_A (\p)/\p=(V_A (\p)/\p)^h=(\widetilde{A/\p})^h=\widetilde{A/\p},
\end{equation}
and the same holds with $B$ and $\q$, this proves that the extension of rings $V^h_A (\p)/\p \subseteq V^h_B (\q)/\q$ is finite.
The extension $(V_A (\p))_{\p} \subseteq (V_B (\q))_{\q}$ coincides with $A_{\p}\subseteq B_{\q}$ and is thus weakly unramified since the maximal ideals of both rings are generated by the uniformizer $\pi$ of $\O_K$. Hence, by \ref{ExtensionValHenselisé}, so is the extension of their henselizations $(V_A (\p))_{\p}^h \subseteq (V_B (\q))_{\q}^h$. The latter coincides with the extension $\left((V^h_A (\p))_{\p}\right)^h\subseteq \left((V^h_B (\q))_{\q}\right)^h$ \cite[VIII, 7.6]{SGA4}, which is thus weakly unramified. Hence, $(V^h_A (\p))_{\p}\subseteq (V^h_B (\q))_{\q}$ is also weakly unramified by \ref{ExtensionValHenselisé} again.
Since the residue field of $V^h_A (\p)$ is the algebraically closed field $k$, \ref{Monogene} applies and proves our claims. 
\end{proof}

\subsection{} \label{trace} Let $A$ be a ring and $B$ an $A$-algebra which is a projective $A$-module of finite rank $r$. Then, the symmetric bilinear trace map ${\rm tr}: B\times B\to A$, $(x,y)\mapsto {\rm tr}_{B/A}(xy)$ induces a morphism of $A$-modules
\begin{equation}
	\label{Trace map}
 T_{B/A}: {\rm det}_A (B)\otimes_A {\rm det}_A (B)\to A,
\end{equation} 
where $\det_A (B)$ is the invertible $A$-module $\bigwedge_A^r B$ \cite[Chap. III, \S 2]{Serre1}. If $A$ and $B$ are domains and the induced extension of their fields of fractions is separable, then, the image of $T_{B/A}$ is a non-zero principal ideal of $A$, the \textit{discriminant ideal of} the extension $B/A$.

\subsection{}\label{ds (BA)}
For the rest of this section, we resume with the notation of \ref{Category CK} (resp. \ref{Category CK formelle}) and let $\pi$ be a uniformizer of $\O_K$.
Let $A\to B$ be a morphism in $\mathcal{C}_K$ (resp. $\widehat{\mathcal{C}}_K$). On the one hand, since the induced homomorphism $A_K\to B_K$ is an extension of Dedekind domains (\ref{(V) KatoGeneral}), the $K$-linear map $T_{B_K/A_K}$ is well-defined. Following Kato \cite[5.6]{K1}, we define the integer
\begin{equation}
	\label{d eta (BA)}
d_{\eta}(B/A)=\dim_K({\rm Coker}(T_{B_K/A_K})).
\end{equation}
On the other hand, for $\p\in P_s(A)$ and $\q\in P_s(B)$ above $\p$, $T_{V^h_B (\q)/V^h_A (\p)}$ is well defined (\ref{V^h MonogenicLibre}). Let $c(B/A,\p,\q) \in V^h_A (\p)$ be a generator of the image of $T_{V^h_B (\q)/V^h_A (\p)}$. The choice of $\pi$ induces a $\Z^2$ normalization $(\alpha, \beta): \Gamma_{V^h_A (\p)}\xrightarrow{\sim} \Z^2$ (\ref{(V) Independant y}). Then, as in \cite[5.6]{K1}, we define the integer
\begin{equation}
	\label{ds (BA) 1}
d_s(B/A)=\sum_{(\p, \q)} v^{\beta}_{\p}(c(B/A,\p,\q)), 
\end{equation}
where $\p$ runs over $P_s(A)$ and $\q$ runs over the elements of $P_s(B)$ above $\p$. Since any two generators of the image of $T_{V^h_B (\q)/V^h_A (\p)}$ differ by a scalar unit of $V^h_A(\p)$, $d_s(B/A)$ is independent of the choices of such of generators. We note that if, for all $\p$ and $\q$ as above, the extension of discrete valuation rings $B_{\q}/A_{\p}$ is unramified (i.e. $\kappa(\q)/\kappa(\p)$ is a separable extension), then $v^{\beta}_{\p}(c(B/A,\p,\q))$ is the valuation in $\widetilde{A/\p}$ of a generator of the discriminant ideal of $\widetilde{B/\q}$ over $\widetilde{A/\p}$, where $\widetilde{A/\p}$ and $\widetilde{B/\q}$ are the normalizations of $A/\p$ and $B/\q$ in their respective fields of fractions $\kappa(\p)$ and $\kappa(\q)$;
and thus $d_s(B/A)$ is the $k$-dimension of the cokernel of $T_{\widetilde{B_0}/\widetilde{A_0}}$, where $A_0=A/\m_K A$ and $B_0=B/\m_K B$.

\begin{lem}[{\cite[Lemma 5.8]{K1}}]
	\label{d eta ord p}
Let $A$ be an object of $\mathcal{C}_K$ $($resp. $\widehat{\mathcal{C}}_K)$ and denote by $\K$ its field of fractions. Then, for any $x\in \K^{\times}$, we have
\begin{equation}
	\label{d eta ord p 1}
\sum_{\p\in P_s(A)} v_{\p}^{\beta}(x)=\sum_{\p\in P_{\eta}(A)} [\kappa(\p): K] {\rm ord}_{\p}(x).
\end{equation}
\end{lem}

\begin{proof}
We note first that both sums in \eqref{d eta ord p 1} are finite as $P_s(A)$ is finite and $A_K$ is a Dedekind domain with $P_{\eta}(A)$ identified with the set of its non-zero prime ideals.
The $K$-theoretic proof given by Kato in \textit{loc. cit.} when $A$ is an object of $\mathcal{C}_K$ applies also when $A\in {\rm Obj}(\widehat{\mathcal{C}}_K)$. 
\end{proof}

\begin{prop}[{\cite[proof of (5.7)]{K1}}]
	\label{Formule LocalGlobal de Kato}
Let $A\to B$ be a morphism in $\mathcal{C}_K~ ($resp. $\widehat{\mathcal{C}}_K)$. Assume that $A$ and $A_0=A/\m_K A$ are regular. Then, $B$ is a free $A$-module of finite rank and 
\begin{equation}
	\label{Formule LocalGlobal 1}
d_{\eta}(B/A)=\sum_{\p_{\eta}\in P_{\eta}(A)} [\kappa(\p_{\eta}): K] {\rm ord}_{\p_{\eta}}(c),
\end{equation}
\begin{equation}
	\label{Formule LocalGlobal 2}
d_s(B/A) + 2\delta(B)= v^{\beta}_{\p}(c),
\end{equation}
where $\p$ is the unique element of $P_s(A)$ and $c$ is a generator in $A$ of the image of $T_{B/A}$.
\end{prop}
 
\begin{proof}
Recall that $A\to B$ is a finite injective local homomorphism between local domains. As $B$ is Cohen-Macaulay (\ref{Proprietes formelles objets} $(i)$), we deduce from \cite[Chap. 0, 17.3.5 (ii)]{EGA.IV} that $B$ is indeed a free $A$-module. Let $c$ be a generator of the image in $A$ of the well defined homomorphism $T_{B/A}$ \eqref{Trace map}. Then, $c$ is also a generator of the image of $T_{B_K/A_K}$ in the Dedekind domain $A_K$ whose non-zero prime ideals are the elements of $P_{\eta}(A)$; hence, \eqref{Formule LocalGlobal 1} follows. Let $n$ be the maximum integer such that ${\rm Im}(T_{B/A})\subseteq \pi^n A$ and put $T'=\pi^{-n} T_{B/A}$. With the notation of \ref{(V) KatoGeneral}, let $T'_0: \det_{A_0}(B_0)\otimes_{A_0} \det_{A_0}(B_0) \to A_0$ be the homomorphism induced by $T'$ and let $\widetilde{T}'_0: \det_{A_0}(\widetilde{B}_0)\otimes_{A_0} \det_{A_0}(\widetilde{B}_0) \to {\rm Frac}(A_0)$ be the homomorphism induced by $T'_0$. As $\pi$ remains a uniformizer through the composition map $\O_K \to A_{\p}=(V_A(\p))_{\p}\to (V^h_A(\p))_{\p}$ between discrete valuation rings, we see that $n$ is also the maximum integer such that ${\rm Im}(T_{V^h_B(\q)/V^h_A(\p)})\subsetneq \pi^n V^h_A(\p)$.  
It then follows from \ref{Conrad} that $d_s (B/A)$ is the integer $i_{\p}$ such that ${\rm Im}(\widetilde{T}'_0) =\m_{A_0}^{i_{\p}}$. Now from \cite[III, Prop. 5]{Serre1}, we see that 
\begin{equation}
	\label{Formule LocalGlobal 3}
2 \delta(B)= \dim_k \left({\rm Im}(\widetilde{T}'_0)/{\rm Im}(T'_0)\right).
\end{equation}
We also get, looking at the definition of $v^{\beta}_{\p}$, that
\begin{equation}
	\label{Formule LocalGlobal 4}
v^{\beta}_{\p}(c)= \dim_k (A_0/{\rm Im}(T'_0)).
\end{equation}
Putting together $d_s (B/A)= i_{\p}$, \eqref{Formule LocalGlobal 3} and \eqref{Formule LocalGlobal 4} gives us the formula \eqref{Formule LocalGlobal 2}.
\end{proof}

\section{Variation of the discriminant of a rigid morphism.}\label{Discriminant}
For the rest of the article, if $S$ is a finite set, $\lvert S \lvert$ will denote its cardinal and the context will help not to confuse it with the absolute value here and elsewhere.

\subsection{} \label{NotationsRigid}
We let $K$ be a complete discrete valuation field, $\mathcal{O}_K$ its valuation ring, $\m_K$ its maximal ideal, $k$ its residue field, assumed to be algebraically closed of characteristic $p>0$, and $\pi$ a uniformizer of $\O_K$. Let $\overline{K}$ be an algebraic closure of $K$, $v_K: \overline{K}^{\times}\to \Q$ the valuation of $\overline{K}$ normalized by $v_K(K^{\times})=\Z$, and let $C=\widehat{\overline{K}}$ be the completion of $\overline{K}$ with respect to $v$. We denote by $D=D_K=\Sp(K\{\xi\})$ the rigid unit disc over $K$ centered at the point $0$ corresponding to the maximal ideal $(\xi)$ of $K\{\xi\}$, and by $\mathfrak{D}=\mathfrak{D}_K=\Spf(\O_K\{\xi\})$ the formal unit disc, an admissible formal model of $D$ over $\mathcal{S}=\Spf(\O_K)$, with special fiber $\mathbb{A}^1_k=\Spec(k[\xi])$. For $r\in \Q$, we denote by $D^{(r)}=D_K^{(r)}$ the $0$-centered $K$-subdisc of $D$ of radius $\lvert \pi\lvert^r$, defined by $v(\xi)\geq r$. For rational numbers $r \geq r'$, we denote by $A(r, r')$ (resp. $A^{\circ}(r, r')$) the closed (resp. open) annulus in $D$, centered at $0$, with inner radius $\lvert \pi\lvert^{r}$ and outer radius $\lvert \pi\lvert^{r'}$. The affinoid $K$-algebra of $A(r, r')$ is the set of power series 
\begin{equation}
	\label{NotationsRigid 1}
f(\xi)=\sum_{i\in \Z} a_i\xi^{i} \quad \in K[[\xi, \xi^{-1}]]
\end{equation}
such that $f(x)$ converges in $\overline{K}$ for any $x$ in $A(r, r')$.

\subsection{}\label{Weierstrass}
By the Weierstrass preparation theorem \cite[Corollaire 1.5]{Henrio}, if a function $f$ on $A(r, r')$ is invertible, it can be written in the form
\begin{equation}
	\label{Weierstrass 1}
	f(\xi)=c \xi^d (1+ h(\xi)), \quad {\rm where}\quad h(\xi)=\sum_{i\in\Z-\{0\}} h_i\xi^{i},
\end{equation}
$c\in K^{\times}$, $d\in \Z$ and $h$ is a function on $A(r, r')$ such that $\lvert h\rvert_{{\sup}}< 1$. (Therefore, for any $i\in \Z-\{0\}$ and any $r\geq t\geq r'$, we have $\lvert h_i\pi^{it}\lvert < 1$). The integer $d$ is called \textit{the order of} $f$.
After we renormalize by dividing by $c$ in \eqref{Weierstrass 1}, $f$ defines a rigid morphism from $A(r, r')$ to $A(d r, d r')$. If, moreover, $f$ is étale, then, by the jacobian criterion, its derivative $f'(\xi)=\frac{d f(\xi)}{d \xi}$ is also an invertible power series and thus has a well-defined order $\sigma$; we put
\begin{equation}
	\label{orderderivative}
\nu= \sigma - d + 1\in \Z.
\end{equation}

\subsection{}\label{WeierstrassDiscOuvert}
The open annulus $A^{\circ}(r, r')$ is the increasing union of the closed annuli $A(\varepsilon, \varepsilon')$, taken over the rational numbers $\varepsilon$ and $\varepsilon'$ satisfying $r>\varepsilon\geq \varepsilon'> r'$. Hence, the ring $\O(A^{\circ}(r, r'))$ of functions on $A^{\circ}(r, r')$ is the projective limit of the rings $\O(A(\varepsilon, \varepsilon'))$ with transition maps the restrictions to smaller annuli. If $f=(f_{\varepsilon, \varepsilon'})$ is an invertible function on $A^{\circ}(r, r')$, then each $f_{\varepsilon, \varepsilon'}$ is an invertible function on $A(\varepsilon, \varepsilon')$ and thus has a well defined order $d_{\varepsilon, \varepsilon'}$ (\ref{Weierstrass}). As the formula \eqref{Weierstrass 1} is invariant under restriction, we see that $d_{\varepsilon, \varepsilon'}$ is independent of $\varepsilon$ and $\varepsilon'$. Thus, it is a well-defined order of $f$, which we denote $d$. If, moreover, $f$ is étale on $A^{\circ}(r, r')$, likewise, we get a well-defined order $\sigma$ for its derivative, and put $\nu= \sigma - d + 1$.

\begin{lem} [{\cite[Lemma 2.3]{Lutke}}]
		\label{DecompositionCouronnes}
Let $X$ be a smooth rigid space over $K$ and $f: X \to A(r, r')$ be a finite flat morphism, étale over a nonempty open subset of $A(r, r')$. Then, there exist a finite extension $K'$ of $K$, and a finite sequence of rational numbers
$r=r_0 > r_1 > \cdots > r_n > r_{n+1}=r'$ in $v_K(K')$ such that, denoting by $f_{K'}: X_{K'}\to A_{K'}(r, r')$ the base change of $f$ to $K'$, the following holds for each $i=1,\ldots, n+1$.
\begin{enumerate}
\item[$(1)$] The inverse image $\Delta_i=f_{K'}^{-1}\left( A_{K'}^{\circ}(r_{i-1}, r_i)\right)$ decomposes into a finite disjoint union of rigid open annuli $\Delta_i=\Delta_{i 1}\cup\cdots \cup \Delta_{i \delta(i)}$.
\item[$(2)$] The restriction of $f_{K'}$ to each annulus $\Delta_{ij}$ is an étale morphism
\begin{equation}
	\label{DecompositionCouronnes 1}
f_{ij}: \Delta_{ij}\to A^{\circ}(r_{i-1}, r_i), \quad \xi_{ij}\mapsto \xi_{ij}^{d_{ij}}(1+ h_{ij}),
\end{equation}
where $h_{ij}$ is a function on $\Delta_{ij}$ satisfying $\lvert h_{ij} \rvert_{\sup} < 1$ and $d_{ij}\geq 1$ is an integer such that $\Delta_{ij}\xrightarrow{\sim} A_{K'}^{\circ}(r_{i-1}/d_{ij}, r_i /d_{ij})$,
\item[$(3)$] the sum $d=d_{i1}+\cdots + d_{i\delta(i)}$ is independent of $i$.
\end{enumerate}
\end{lem}

\begin{proof}
The main ingredient in the proof given in \textit{loc. cit.} is the semi-stable reduction theorem \cite[Theorem 7.1]{BoschLutke}. There, the lemma is proved, over the complete algebraically closed field $C$, for the base change $X_C=X\widehat{\otimes}_K C$. This is achieved by using a semi-stable formal model $\mathfrak{X}_C$ of $X_C$ over $\O_C$ to produce the rational numbers $r_1, \ldots, r_n$ and the étale morphisms $f_{ij}$. (See \cite[Lemma 2.3]{Lutke} for more details.) The model $\mathfrak{X}_C$ descends to a formal model $\mathfrak{X}_{K'}$ of $X_{K'}$, for some finite extension $K'$ of $K$ which we can take to be large enough to contain all $\pi^{r_i}$ and the coefficients defining the morphisms $f_{ij}$. As the latter morphisms are obtained first by restricting to $\Delta_i$ the same morphism $f_{K'}$, the independence statement of $(3)$ follows.
\end{proof}

\subsection{} \label{total order}
We keep the notation of Lemma \ref{DecompositionCouronnes}. From the well-defined orders $\sigma_{ij}=d_{ij}-1+ \nu_{ij}$ of the derivatives of the $f_{ij}$ (\ref{WeierstrassDiscOuvert}), we can define the total order of the the derivative of $f$ on $\Delta_i$
\begin{equation}
	\label{total order 1}
\sigma_i=\sigma_{i1}+\cdots + \sigma_{i\delta(i)}=d-\delta(i)+ \nu_i, \quad \nu_i=\nu_{i1}+ \cdots + \nu_{i\delta(i)}.
\end{equation}

\subsection{} \label{Unit Balls}
 Let $X$ be a smooth $K$-affinoid curve and let $f:X\to D$ be a finite flat morphism of degree $d$, which is étale over a nonempty open subset of $D$ containing $0$. Let $r\in\Q$, denote by $X^{(r)}$ the inverse image of $D^{(r)}$ by $f$, by $f^{(r)}: X^{(r)}\to D^{(r)}$ the induced morphism and set
\begin{equation}
	\label{Unit Balls 1}
\O^{\circ}(D^{(r)})=\{ h\in \O_D(D^{(r)}) ~ \lvert ~ \lvert h \rvert_{\sup}\leq 1\} , \quad \O^{\circ}(X^{(r)})=\{ h\in \O_X(X^{(r)}) ~ \lvert ~ \lvert h \rvert_{\sup}\leq 1\},\
\end{equation} 
for the unit balls of the affinoid algebras $\O(D^{(r)})$ and $\O(X^{(r)})$. If $K'$ is a finite extension of $K$, $D_{K'}^{(r)}$ (resp. $X_{K'}^{(r)}$) denotes the base change $D^{(r)}\widehat{\otimes}_K K'$ (resp. $X^{(r)}\widehat{\otimes}_K K'$). We denote by $D_C, X_C$ and $f_C: X_C\to D_C$ the base change to $C$ of $D, X$ and $f$ respectively.

\begin{lem}
	\label{Adic Unit Balls}
The unit balls $\O^{\circ}(D^{(r)})$ and $\O^{\circ}(X^{(r)})$ are $(admissible)$ $\pi$-adic rings and the homomorphism $\O^{\circ}(D^{(r)})\to \O^{\circ}(X^{(r)})$ induced by $f^{(r)}$ is finite and generically étale.
\end{lem}
\begin{proof}
Indeed, writing $r=\frac{a}{b}$ with $a, b$ integers, we have $\O^{\circ}(D^{(r)})=\O_K\{\xi, \zeta\}/(\xi^b-\zeta\pi^{a})$. Being topologically of finite type over $\O_K$, $\O^{\circ}(D^{(r)})$ is an adic ring. Since $f$ is finite, so is $f^{(r)}$. As $D^{(r)}$ is reduced, \cite[6.4.1/6]{BGR} implies that $\O^{\circ}(X^{(r)})$ is finite over $\O^{\circ}(D^{(r)})$, hence also an adic ring \cite[1.8.29]{EGR}. The generic étaleness follows from the étaleness hypothesis on $f$.
\end{proof}

\begin{lem} 
	\label{ModelsLibre}
With the notation of {\rm \ref{Unit Balls}}, $\O^{\circ}(X_C^{(r)})$ is a free $\O^{\circ}(D_C^{(r)})$-module of finite rank.
\end{lem}

\begin{proof}
As $\O^{\circ}(D_C^{(r)})/\m_C\O^{\circ}(D_C^{(r)})\cong k[\xi/\pi^r]$ and $\O^{\circ}(X_C^{(r)})/\m_C\O^{\circ}(X_C^{(r)})$ are the rings of the special fibers of the normalized integral models of $D_C^{(r)}$ and $X_C^{(r)}$ respectively (\ref{Normalized integral model}), they are reduced (\ref{Reduced Fiber Theorem}). Hence, they coincide with $\O^{\circ}(D_C^{(r)})/\O^{\circ\circ}(D_C^{(r)})$ and $\O^{\circ}(X_C^{(r)})/\O^{\circ\circ}(X_C^{(r)})$ respectively, where $\O^{\circ\circ}(D_C^{(r)})$ and $\O^{\circ\circ}(X_C^{(r)})$ are the sets of all topologically nilpotent elements of $\O(D_C^{(r)})$ and $\O(X_C^{(r)})$ respectively. It follows that $\O^{\circ}(X_C^{(r)})/\m_C\O^{\circ}(X_C^{(r)})$ is a finitely generated module over $k[\xi/\pi^r]$ \cite[Theorem 6.3.4/2]{BGR}. It is also flat by the following argument. As $X_C^{(r)}\to D_C^{(r)}$ is finite and flat, it is surjective and thus its image is not sent to the tube of a point of the special fiber of the formal model $\Spf(\O^{\circ}(D_C^{(r)}))$ of $D_C^{(r)}$. Hence, the homomorphism $k[\xi/\pi^r]\to \O^{\circ}(X_C^{(r)})/\m_C\O^{\circ}(X_C^{(r)})$ is injective. To see that it is flat, it is enough to show that every localization of $\O^{\circ}(X_C^{(r)})/\m_C\O^{\circ}(X_C^{(r)})$ at a prime ideal is Cohen-Macaulay \cite[Chap 0, 17.3.5]{EGA.IV}, hence it is enough to show that every localization of $\O^{\circ}(X_C^{(r)})$ at a prime ideal is Cohen-Macaulay \cite[Chap 0, 16.5.5]{EGA.IV}. By \ref{Reduced Fiber Theorem} and \ref{A.S.1 Lemme 4.1}, the latter ring is integrally closed in the normal ring $\O(X_C^{(r)})$, hence it is normal. Thus, its localizations at prime ideals are two-dimensional normal rings, hence Cohen-Macaulay \cite[Chap. 0, discussion below 16.5.1]{EGA.IV}.
It follows that $\O^{\circ}(X_C^{(r)})/\m_C\O^{\circ}(X_C^{(r)})$ is a torsion-free $k[\xi/\pi^r]$-module.
Therefore, $\O^{\circ}(X_C^{(r)})/\m_C\O^{\circ}(X_C^{(r)})$ is free of finite rank over $k[\xi/\pi^r]$. Then, by \cite[6.4.2/3]{BGR}, $\O^{\circ}(X_C^{(r)})$ is also free of finite type over $\O^{\circ}(D_C^{(r)})$.
\end{proof}
 
\begin{prop} 
	\label{ModelsFormels}
There exists a finite extension $K'$ of $K$, containing an element $\pi^r$ of valuation $r$, such that
\begin{enumerate}
\item[$(1)$] $\O^{\circ}(X_{K'}^{(r)})$ is a finite free $\O^{\circ}(D_{K'}^{(r)})$-module.
\item[$(2)$] $\mathfrak{D}_{K'}^{^{(r)}}=\Spf\left(\O^{\circ}(D_{K'}^{(r)})\right)$ and $\mathfrak{X}_{K'}^{^{(r)}}=\Spf\left(\O^{\circ}(X_{K'}^{(r)})\right)$ are admissible formal models over $\mathcal{S}'=\Spf(\O_{K'})$ of $D_{K'}^{(r)}$ and $X_{K'}^{(r)}$ respectively, with geometrically reduced special fibers.
\end{enumerate}
Moreover, if $K'$ is a finite extension of $K$, containing an element $\pi^r$ of valuation $r$, satisfying $(2)$, then we also have
\begin{enumerate}
\item[$(3)$] $\mathfrak{D}_{K'}^{(r)}$ is smooth over $\mathcal{S}'$; $\mathfrak{X}_{K'}^{(r)}$ is smooth over $\mathcal{S}'$ outside a finite number of closed points in its special fiber, and is normal.
\item[$(4)$] If $K''$ is a finite extension of $K'$ and $\mathcal{S}''=\Spf(\O_{K''})$, then
 \begin{equation}
	\label{ModelsFormels 1}
\mathfrak{X}_{K''}^{(r)}\cong \mathfrak{X}_{K'}^{(r)}\times_{\mathcal{S}'}\mathcal{S}''.
\end{equation}
\end{enumerate}
\end{prop}

\begin{proof}
The $K$-affinoid spaces $D_{K}^{(r)}$ and $X_{K}^{(r)}$ are smooth, hence geometrically reduced. Then, by Theorem \ref{Reduced Fiber Theorem}, we see that there exists a finite extension $K'$ of $K$ such that $\O^{\circ}(D_{K'}^{(r)})$ and $\O^{\circ}(X_{K'}^{(r)})$ have geometrically reduced special fibers and and their formation commutes with finite extensions $K''/K'$, namely
\begin{equation}
	\label{ModelsFormels 2}
\O^{\circ}(D_{K''}^{(r)})\cong \O^{\circ}(D_{K'}^{(r)})\otimes_{\O_{K'}} \O_{K''} \quad {\rm and}\quad \O^{\circ}(X_{K''}^{(r)})\cong \O^{\circ}(X_{K'}^{(r)})\otimes_{\O_{K'}} \O_{K''}.
\end{equation}
It then follows from \ref{A.S.1 Lemme 4.1} that $\O^{\circ}(D_{K'}^{(r)})$ and $\O^{\circ}(X_{K'}^{(r)})$ are formal models of, and integrally closed in, $\O(D_{K'}^{(r)})$ and $\O(X_{K'}^{(r)})$ respectively (and remain so after finite extension \eqref{ModelsFormels 2}). In particular, they are also normal, which implies, by Lemma \ref{normal}, that $\mathfrak{D}_{K'}^{(r)}$ and $\mathfrak{X}_{K'}^{(r)}$ are normal.
Taking the colimit over $K''$ in \eqref{ModelsFormels 2} and completing $\pi$-adically gives
\begin{equation}
	\label{ModelsFormels 3}
\O^{\circ}(D_C^{(r)})=\O^{\circ}(D_{K'}^{(r)})\widehat{\otimes}_{\O_{K'}} \O_C \quad {\rm and}\quad \O^{\circ}(X_C^{(r)})=\O^{\circ}(X_{K'}^{(r)})\widehat{\otimes}_{\O_{K'}} \O_C.
\end{equation}
From Lemma \ref{ModelsLibre}, we know that $\O^{\circ}(X_C)$ is a finite free algebra over $\O^{\circ}(D_C^{(r)})$. Hence, by \eqref{ModelsFormels 2} and Lemma \ref{FreeAdic}, possibly after a finite extension of $K'$, we can assume that $\O^{\circ}(X_{K'}^{(r)})$ is a finite free algebra	 over $\O^{\circ}(D_{K'}^{(r)})$.
Smoothness of  $\mathfrak{D}_{K'}^{(r)}$ (resp. $\mathfrak{X}_{K'}^{(r)}$) over $\mathcal{S}'$ is tested on the special fiber \cite[2.4.6]{EGR}. Since the latter is a smooth curve (resp. reduced curve), it is smooth outside a finite set of closed points.
\end{proof}

\begin{defi} \label{radmissible}
In the situation of \ref{ModelsFormels}, we say that the field $K'$ is $r$-\emph{admissible for} $f$ and that the adic $\O_{K'}$-morphism $\widehat{f}^{(r)}: \mathfrak{X}_{K'}^{(r)}\to\mathfrak{D}_{K'}^{(r)}$ induced by $f^{(r)}$ is \textit{the normalized integral model of} $f^{(r)}$ \textit{over} $K'$ (see also Definition \ref{Normalized integral model}). We then see from \ref{ModelsFormels} (4) that any further finite extension $K''$ is also $r$-admissible for $f$ and that the base change $\widehat{f}^{(r)}\widehat{\otimes}_{\O_{K'}} \O_{K''}$ is the normalized integral model of $f^{(r)}$ over $K''$. Sometimes we drop $K'$ to lighten the notation, especially when dealing with the local rings of these models at some points.
\end{defi}

\begin{rem}
	\label{Functoriality of Models}
As any homomorphism $h:A\to B$ of $K$-affinoid algebras satisfy the inequality $\lvert h(a)\lvert_{\rm sup}\leq \lvert a\lvert_{\rm sup}$, for any $a\in A$, the construction of the normalized integral models is functorial: if $X'$ a smooth $K$-affinoid curve and $g:X'\to X$ is a $K$-morphism such that the composition $f'=f\circ g$ is finite, flat, and étale over a nonempty open subset of $D$ containing $0$, then we have an induced adic morphism $\widehat{g}^{(r)}: \mathfrak{X}_{K'}^{'(r)}\to \mathfrak{X}_{K'}^{(r)}$ such that $\widehat{f}^{'(r)}=\widehat{f}^{(r)}\circ \widehat{g}^{(r)}$, where $K'$ is $r$-admissible for $f$ and $f'$.
\end{rem}

\begin{rem}\label{etaleseparable}
With the notation of {\rm \ref{Unit Balls}}, assume that $K'$ is $r$-admissible. Since, by Lemma \ref{Adic Unit Balls}, $\O^{\circ}(X_{K'}^{(r)})$ is a generically étale $\O^{\circ}(D_{K'}^{(r)})$-algebra, which is also finite free by \ref{ModelsFormels} $(4)$, it is a finite product of domains which are free over $\O^{\circ}(D_{K'}^{(r)})$ and whose fields of fractions are separable over the field of fractions of $\O^{\circ}(D_{K'}^{(r)})$.
\end{rem}

\begin{defi}
	\label{DiscrimRig}
We keep the notation of \ref{Unit Balls} and assume that $K'$ is $r$-admissible for $f$. We let $\mathfrak{d}_{f}(r, K')$ be the discriminant ideal of $\O^{\circ}(X_{K'}^{(r)})$ over $\O^{\circ}(D_{K'}^{(r)})$; it is an invertible (i.e. locally monogenic) ideal of $\O^{\circ}(D_{K'}^{(r)})$ (cf. \ref{trace} and \ref{etaleseparable}). For a finite extension $K''/K'$, we have the inclusion $\O^{\circ}(D_{K'}^{(r)}) \hookrightarrow \O^{\circ}(D_{K''}^{(r)})$, and \ref{ModelsFormels} $(4)$ implies that 
\begin{equation}
	\label{DiscrimRig 1}
\mathfrak{d}_{f}(r, K')\O^{\circ}(D_{K''}^{(r)}) =\mathfrak{d}_{f}(r, K'').
\end{equation}
Hence, the discriminant ideal does not depend on the field of definition of the normalized integral model of $f^{^{(r)}}$. So we denote it simply by $\mathfrak{d}_f(r)$.
\end{defi}

\begin{rem}\label{DiscLutke}
From \ref{ModelsFormels} $(4)$ and the inclusion $\O^{\circ}(D_{K'}^{(r)}) \hookrightarrow \O^{\circ}(D_C^{(r)})$, we have
\begin{equation}
	\label{DiscLutke 1}
\mathfrak{d}_f(r)\O^{\circ}(D_C^{(r)})=\mathfrak{d}_{f_C}(r),
\end{equation} 
where $\mathfrak{d}_{f_C}(r)$ is the well-defined discriminant ideal of $\O^{\circ}(X_C^{(r)})$ over $\O^{\circ}(D_C^{(r)})$ (cf. \ref{ModelsLibre} and \ref{etaleseparable}).
\end{rem}

\subsection{}\label{Discr-Invertible-Ideal}
Let $r\geq 0$ be a rational number and $\mathfrak{d}$ an invertible ideal of $\O^{\circ}(A(r, r))$. Let $(U_i)$ be a formal open cover of $\Spf(\O^{\circ}(A(r, r)))\cong \widehat{\mathbb{G}}_{m, \O_K}$ that trivializes $\mathfrak{d}$ and $(g_i)$ a tuple of local generators of $\mathfrak{d}$ on $(U_i)$.
If $x: \Spf(O_K)\to \widehat{\mathbb{G}}_{m, \O_K}$ is a rig-point that lands in one of the opens $U_i$ of the cover, then $\lvert g_i(x)\lvert_x$ \eqref{supnorm} is independent of the chosen generator $g_i$ of $\mathfrak{d}$ on $U_i$ since another choice of local generator differs from $g_i$ by a factor which is a unit in $\O(U_i)$, hence a unit in the valuation ring of $\O(U_{i, \eta})/x$. We put $ \lvert \mathfrak{d}(x)\lvert_x=\lvert g_i(x)\lvert_x$ and define the sup-norm of $\mathfrak{d}$ as
\begin{equation}
	\label{Discr-Invertible-Ideal 1}
	\lvert \mathfrak{d}\lvert_{\rm sup}=\sup_x \lvert \mathfrak{d}(x)\lvert_x,
\end{equation}
where $x$ runs over the rig-points of $\widehat{\mathbb{G}}_{m, \O_K}$. It is clear that $\lvert \mathfrak{d}\lvert_{\rm sup}$ is independent of the cover trivializing $\mathfrak{d}$. Therefore, $\lvert \mathfrak{d}\lvert_{\rm sup}$ is well-defined.

\subsection{}\label{Discr-Couronne-Epaisseur-Nulle}
Let $r\geq 0$ be a rational number and $d\geq 1$ an integer. If $g: A(r/d, r/d)\to A(r, r)$ is a finite flat morphism of order $d$, then, $g$ is given by $\xi\mapsto \xi^d(1 + h(\xi))$, where $h$ is a function on $A(r/d, r/d)$ satisfying $\lvert h\lvert_{\rm sup} < 1$ (\ref{Weierstrass}); thus, $\O^{\circ}(A(r/d, r/d))$ is a locally free $\O^{\circ}(A(r, r))$-module of finite rank $d$. We define the discriminant ideal $\mathfrak{d}_g[r]$ of $g$ as the discriminant of $\O^{\circ}(A(r/d, r/d))$ over $\O^{\circ}(A(r, r))$. It is an invertible ideal of $\O^{\circ}(A(r, r))$, hence has a well-defined supremum norm \eqref{Discr-Invertible-Ideal}. By \ref{Weierstrass}, if $g: A(\varepsilon, \varepsilon')\to A(r, r')$ is a finite flat morphism and $r\geq t\geq r'$ is rational number, we have a well-defined discriminant ideal $\mathfrak{d}_g[t]$ by restricting $g$ over $A(t, t)$.

\subsection{} \label{Couronne Epaisseur Nulle}
We keep the notation of \ref{Unit Balls} and assume that $K'$ is $r$-admissible for $f$ (\ref{radmissible}). Let $D^{[r]}=A(r, r)\subset D^{(r)}$ be the annulus of radius $\lvert \pi\lvert^r$ with $0$-thickness and $X^{[r]}=f^{-1}(D^{[r]})$ its inverse image, a smooth $K$-affinoid space. Over $K'$, we have
\begin{equation}
	\label{Couronne Epaisseur Nulle 1}
\O(D_{K'}^{(r)})=K'\{T\}, \quad \O(D_{K'}^{[r]})=K'\{T, T^{-1}\}\quad {\rm and}\quad \O(X_{K'}^{[r]})=\O(X_{K'}^{(r)})\{T^{-1}\}.
\end{equation}
As $\O^{\circ}(D_{K'}^{(r)})\{T, T^{-1}\}$ (resp. $\O^{\circ}(X_{K'}^{(r)})\{T^{-1}\}$) is a formal model over $\O_{K'}$ of $\O(D_{K'}^{(r)})\{T^{-1}\}$ (resp. $\O(X_{K'}^{(r)})\{T^{-1}\}$) whose special fiber $\mathbb{A}_k^1$ (resp. $(\O^{\circ}(X_{K'}^{(r)})\otimes_{\O_{K'}} k)[T^{-1}])$ is geometrically reduced (\ref{ModelsFormels} $(2)$), Lemma \ref{A.S.1 Lemme 4.1}(ii) implies that
\begin{equation}
	\label{Couronne Epaisseur Nulle 2}
\O^{\circ}(D_{K'}^{[r]})=\O^{\circ}(D_{K'}^{(r)})\{T^{-1}\}=\O_{K'}\{T, T^{-1}\} \quad {\rm and}\quad \O^{\circ}(X_{K'}^{[r]})=\O^{\circ}(X_{K'}^{(r)})\{T^{-1}\}.
\end{equation}
Hence, $\mathfrak{D}_{K'}^{[r]}=\Spf(\O^{\circ}(D_{K'}^{[r]}))\simeq\widehat{\mathbb{G}}_{m, \O_{K'}}$ is an open subscheme of the formal affine line $\mathfrak{D}_{K'}^{(r)}\simeq \widehat{\mathbb{A}}^1_{K'}$. Moreover, $\mathfrak{D}_{K'}^{[r]}$ and $\mathfrak{X}_{K'}^{[r]}=\Spf(\O^{\circ}(X_{K'}^{[r]}))$ satisfy the statements $(1)$ to $(4)$ of Proposition \ref{ModelsFormels}, and the induced normalized integral model $\widehat{f}^{^{[r]}}: \mathfrak{X}_{K'}^{[r]}\to \mathfrak{D}_{K'}^{[r]}$ of the restriction $f^{^{[r]}}: X^{[r]}\to D^{[r]}$ of $f$ fits into a Cartesian square
\begin{equation}
	\label{Couronne Epaisseur Nulle 3}
\xymatrix{\ar @{} [dr] | {\Box} 
 \mathfrak{X}_{K'}^{[r]} \ar[r] \ar[d]_{\widehat{f}^{[r]}} & \mathfrak{X}_{K'}^{(r)} \ar[d]^{\widehat{f}^{(r)}} \\
 \mathfrak{D}_{K'}^{[r]} \ar[r] & \mathfrak{D}_{K'}^{(r)},
}
\end{equation}
where the bottom horizontal arrow is a formal open immersion.
In particular, $\mathfrak{X}_{K'}^{[r]}$ is a formal open subscheme of $\mathfrak{X}_{K'}^{(r)}$, whose complement lies over $\mathfrak{D}_{K'}^{(r)}-\mathfrak{D}_{K'}^{[r]}$.\newline
The discriminant ideal $\mathfrak{d}_f[r]$ of $\O^{\circ}(X_{K'}^{[r]})$ over $\O^{\circ}(D_{K'}^{[r]})$ is also well-defined (\ref{ModelsFormels} (1)), independent of the choice of $K'$ which is $r$-admissible for $f$ (\ref{ModelsFormels} (4)) and consistent with \ref{Discr-Couronne-Epaisseur-Nulle}. By \cite[III, \S 4, Prop. 9]{Serre1}, we have $\mathfrak{d}_f[r]=\mathfrak{d}_f(r)[\frac{1}{T}]$. In particular, with the notation of \ref{Discr-Invertible-Ideal}, we see that
\begin{equation}
	\label{Couronne Epaisseur Nulle 4}
\lvert \mathfrak{d}_f[r]\lvert_{\rm sup}=\lvert \mathfrak{d}_f(r)\lvert_{\rm sup}.
\end{equation}

\begin{lem} [{\cite[Lemma 1.7]{Lutke}}]
	\label{FormuleLutke}
With the notation of {\rm \ref{NotationsRigid}}, assume that the morphism $f$ is étale and given by
\begin{equation}
	\label{FormuleLutke 1}
A(r/d, r'/d)\to A(r, r'), \quad \xi\mapsto \xi^d(1+h(\xi)), \quad {\rm with}\quad h(\xi)=\sum_{i\neq 0} h_i\xi^{i}, ~ \lvert h(\xi) \rvert_{\sup}< 1,
\end{equation}
where $r\geq r'$ are in $\Q$. Denote by $\sigma= d-1+ \nu$, $\nu\in \Z$, the order of the derivative $f'(\xi)=\frac{d f(\xi)}{d \xi}$. Then, for a rational number $t$ such that $r\geq t\geq r'$, we have
\begin{equation}
	\label{FormuleLutke 2}
	\lvert \mathfrak{d}_f[t]\lvert_{\sup}=\lvert d \lvert ^d \quad {\rm if}\quad \nu=0 \quad {\rm and}\quad \lvert \mathfrak{d}_f[t]\lvert_{\sup}=\lvert \nu h_{\nu}\lvert^d \lvert \pi\lvert^{\nu t} \quad {\rm if}\quad \nu\neq 0.
\end{equation}
\end{lem}

\begin{proof}
For the sake of completeness, we reproduce here, with some more details, Lütkebohmert's proof.
We can write
\begin{equation}
	\label{FormuleLutke 3}
f(\xi)=\sum_{i\in \Z} a_i\xi^{i} =\xi^d (1+h(\xi)) ;
\end{equation}
whence we see, putting $h_0=1$, that $a_{d+i}=h_i$ for all $i\in \Z$. By the Weierstrass preparation theorem, recalled in \ref{Weierstrass}, the derivative $f'(\xi)$ can also be written in the form
\begin{equation}
	\label{FormuleLutke 4}
f'(\xi)=\left(\sum i a_i\xi^{i-1}\right)=(d+\nu) a_{d+\nu}\xi^{\sigma}(1+g(\xi)),
\end{equation}
where $g$ is a function on $ A(r/d, r'/d)$ such that $\lvert g\rvert_{{\sup}}<1$. It follows that $\lvert f'(\xi)\lvert_{\rm sup}=\lvert (d+\nu) a_{d+\nu}\lvert$ and thus $(d+\nu) a_{d+\nu}$ is the dominant coefficient of $f'(\xi)$. Then, by \cite[9.7.1/1]{BGR}, applied to $f'$, and the identity $a_{d+i}=h_i$ (with $a_d=h_0=1)$), we have the following. 
If $\nu =0$, then $\lvert (d+i) h_i \pi^{it}\lvert < \lvert d\lvert$ for all $i\in \Z-\{0\}$ and all $r \geq t\geq r'$.
If $\nu\neq 0$, then $\lvert (d+i) h_i \pi^{(i-\nu)t} \lvert < \lvert (d+\nu)h_{\nu}\lvert$ for all $i\neq \nu$ and all $r\geq t\geq r'$, and thus, taking $i=0$, we have $\lvert d\lvert < \lvert (d + \nu) h_{\nu} \pi^{\nu t}\lvert$ for all $r\geq t\geq r'$; in particular, as $\lvert h(\xi)\lvert_{\sup} < 1$, we get $\lvert d\lvert < \lvert d+\nu\lvert$ and thus $\lvert \nu\lvert =\lvert d + \nu\lvert$.  

Now, we make the change of coordinates $T=\xi/\pi^{t/d}$ and put $U(T)=f(\pi^{t/d}T)/\pi^t$, $\widetilde{g}(T)=g(\pi^{t/d}T)$. Then, using the identity $\sigma-1 + d=\nu$, \eqref{FormuleLutke 4} translates into
\begin{equation}
	\label{FormuleLutke 5}
dU(T)/dT=(d+\nu) a_{d+\nu}\pi^{\nu t/d}T^{\sigma}(1+g).
\end{equation}
It follows that the different ideal is generated by $(d+\nu) a_{d+\nu}\pi^{t\nu/d}T^{\sigma}$. Hence, the discriminant ideal $\mathfrak{d}_f[t]$ is generated by its $d$-th power $(d+\nu)^d a_{d+\nu}^d\pi^{\nu t}T^{d\sigma}$ and thus $\lvert \mathfrak{d}_f[t]\lvert_{\rm sup}=\lvert (d+\nu)^d h_{\nu}^d \pi^{\nu t}\lvert$. If $\nu=0$, then $\lvert \mathfrak{d}_f[t]\lvert_{\rm sup}=\lvert d\lvert^d$, since $h_0=1$. If $\nu\neq 0$, then $\lvert \mathfrak{d}_f[t]\lvert_{\rm sup}=\lvert \nu h_{\nu}\lvert^d \lvert \pi\lvert^{\nu t}$, since $\lvert d + \nu\lvert=\lvert \nu\lvert$.
\end{proof}

\subsection{} \label{ValuationDisqueUnité}
Let $x$ be the origin point in $\mathbb{A}^1_k$, corresponding to the open maximal ideal $\m_x=(\pi, \xi)$ of $\O_K\{\xi\}$, a closed point of $\mathfrak{D}$. Let $\overline{x}=\Spec(k)\to \mathbb{A}^1_k$ be the geometric point associated to $x$. We let $A=\O_{\mathfrak{D}, \overline{x}}$ be the étale local ring of $\mathfrak{D}$ at $\overline{x}$ \eqref{ALFormelEtale 1}. The ideal $\p=\pi A$ of $A$ is prime of height $1$, and $A/\p=k[\xi]_{(\xi)}^{\rm sh}$ \eqref{ALFE 2}.
These data produce a normalized $\Z^2$-valuation ring $V=V_A (\p)$  (see \ref{(V) Independant y}) with field of fractions $\K=\K(\p)$. We denote by $v: \K^{\times}\to \Z^2$ the corresponding normalized valuation map, and by $v^{\alpha}: \K^{\times}\to \Z$ and $v^{\beta}: \K^{\times}\to \Z$ the associated projections (\ref{AlphaEtBeta}).
We note that these normalized valuation maps, defined using the uniformizer $\pi$, don't in fact depend on $\pi$ (see the end of \ref{(V) Independant y}). 
We also note that the degree of imperfection of the residue field $\kappa(\p)$ of $V$ at $\p$ is $[\kappa(\p): \kappa(\p)^p]=[k((\xi)): k((\xi))^p]=p$ \cite[2.1.4]{GO}.

\begin{lem}
	\label{(V) Explicit}
\begin{enumerate}
\item[$(i)$] On $k[\xi]$, the valuation map for the discrete valuation ring $A/\p$ is given by 
\begin{equation}
	\label{(V) Explicit 1}
\sum_{n=0}^d a_n \xi^n \mapsto \min \{n ~ \lvert ~ a_n \neq 0\}.\
\end{equation}
\item[$(ii)$] Consider $\O_K\{\xi\}= \O_{\mathfrak{D}}(\mathfrak{D})$ as a subring of $A\subsetneq A_{\p}$. Then, the restriction to $K\{\xi\}^{\times}$ of the valuation map $v_{\p}: \K^{\times}\to \Z$,
associated to the discrete valuation ring $A_{\p}$ {\rm \eqref{(V) formel}}, is given by
\begin{equation}
	\label{(V) Explicit 2}
\sum_{n\geq 0} a_n \xi^n \mapsto \inf_n \{v_K(a_n)\}.
\end{equation}
\item[$(iii)$] If $g(\xi)=c \xi^d (1+h(\xi))$ is in $\O_K \{\xi\}$ with $c\in K^{\times}$ and $\lvert h\lvert_{\sup}<1$, then
\begin{equation}
	\label{(V) Explicit 3}
v^{\alpha}(g(\xi))=v_K (c) \quad {\rm and}\quad v^{\beta}(g(\xi))=d.
\end{equation}
\end{enumerate}
\end{lem}
 
 \begin{proof}
$(i)$ and $(ii)$  are clear since both $A/\p$ and $A_{\p}$ are discrete valuation rings with maximal ideals generated by $\xi$ and $\pi$ respectively. Item $(iii)$ follows from $(i)$, $(ii)$ and \ref{Conrad}.
 \end{proof}
 
\subsection{} \label{(Vr)}
We keep the notation of \ref{ValuationDisqueUnité}, let $r \in \Q$ and assume that $K'$ is a finite extension of $K$ containing an element $\pi^r$ of valuation $r$. In order to describe integral models of $D_{K'}^{(r)}$, we make the change of variable $T=\xi/\pi^r$, which reduces us to the formal unit disc $\Spf(\O_{K'}\{T\})$. We then apply the construction in \ref{ValuationDisqueUnité} to this disc and the origin of its special fiber. This gives a $\Z^2$-valuation ring $V_r$ with field of fractions $\K_r$ containing $K'\{T\}$ and normalized valuation map $v'_r: \K_r^{\times}\to \Z^2$. We denote by $v_r$ the composition $\K_r^{\times}\xrightarrow{v'_r} \Z^2 \hookrightarrow\Q\times \Z$, where the second map is given by $(a, b)\mapsto (a/e, b)$ with $e$ the ramification index of the extension $K'/K$, and we denote by $v_r^{\alpha}$ and $v_r^{\beta}$ the composition of $v_r$ with the projections $\Q\times \Z\to \Q$ and $\Q\times\Z\to \Z$ respectively. As in \ref{ValuationDisqueUnité}, we note the residue field of $V_r$ at the height $1$ prime ideal has degree of imperfection $p$.
For $g(\xi)=c \xi^d (1+h(\xi))=c\pi^{rd}T^d (1+ \widetilde{h}(T))$ in $\O_{K'}\{\xi\} \subset \O_{K'}\{T\}$ with $c\in K'$ non-zero, $h(\xi)$ a function on $D$ satisfying $\lvert h\lvert_{\sup}<1$ and $\widetilde{h}(T)=h(\pi^r T)$ (hence $\lvert \widetilde{h}(T)\lvert_{\rm sup} < 1$), \eqref{(V) Explicit 3} yields
\begin{equation}
	\label{(Vr) 1}
v_r^{\alpha}(g(\xi))=v_K (c) + rd \quad {\rm and}\quad v_r^{\beta}(g(\xi))=d.
\end{equation}
In particular, we see that $v_r^{\beta}(g)$ is independent of $r$.
 
\subsection{}\label{DiscPartial}
For the rest of this section, we keep the notation of \ref{Unit Balls}, \ref{ValuationDisqueUnité} and \ref{(Vr)}, and assume that $K'$ is $r$-admissible for the morphism $f: X\to D$ (see \ref{radmissible}). The discriminant ideal $\mathfrak{d}_f (r)$ is a non-zero principal ideal of $\O_{K'}\{T\}\subseteq V_r$. Therefore, it has well-defined valuations 
\begin{equation}
	\label{DiscPartial 1}
	\partial_f(r)=v_r(\mathfrak{d}_f(r)) ~\in \Q\times \Z, \quad \partial_f^{\alpha}(r)=v_r^{\alpha}(\mathfrak{d}_f(r))~\in \Q \quad {\rm and}\quad  \partial_f^{\beta}(r)=v_r^{\beta}(\mathfrak{d}_f(r))~\in \Z.
\end{equation}
From \eqref{DiscLutke 1} and \ref{(V) Explicit} $(ii)$, we have the following equality in $\lvert \overline{K}^{\times}\lvert$
\begin{equation}
	\label{DiscPartial 2}
\lvert\pi\lvert^{\partial_f^{\alpha}(r)}=\lvert \mathfrak{d}_{f_C} (r) \lvert_{\sup}.
\end{equation}
Hence, $\partial_f^{\alpha}$ is the (additive version of the) discriminant function defined in \cite[1.3]{Lutke}. The following is a rewriting of \cite[Lemma 2.6]{Lutke}.

\begin{prop} \label{DérivéeLutke}
Let $X$ be a smooth $K$-rigid space and let $f:X\to D$ be a finite flat morphism of degree $d$, which is étale over a nonempty open subset of $D$ containing $0$. Then, there exists a finite sequence of rational numbers $(r_i)_{i=1}^n$ such that $0=r_{n+1}< r_n <\cdots < r_1< r_0=+\infty$ and a decomposition of $\Delta_i=f^{-1}\left(A^{\circ}(r_{i-1}, r_i)\right)$ into a disjoint union of open annuli $\Delta_i=\coprod_j \Delta_{ij}$ such that the restriction $f\lvert \Delta_i$ is étale, the function $\partial_f^{\alpha}$ is affine on $\rbrack r_i, r_{i-1}\lbrack \cap \Q$ and its right slope at $t\in \lbrack r_i, r_{i-1}\lbrack \cap \Q$ is
\begin{equation}
	\label{DérivéeLutke 1}
	\frac{d}{dt}\partial_f^{\alpha}(t^+)=\sigma_i - d + \delta_f(i),
\end{equation}
where $\sigma_i$ is the total order of the derivative of $f\lvert \Delta_i$ {\rm \eqref{total order}} and $\delta_f(i)$ is the number of connected components of $\Delta_i$ $($i.e. the number of $\Delta_{ij}$'s$)$.
\end{prop}

\begin{proof}
Lemma \ref{DecompositionCouronnes}, applied to $f: X\to D=A(0, 1)$, gives a finite extension $K'$ of $K$, a sequence $0 <\lvert\pi\lvert^{r_1}<\cdots < \lvert\pi\lvert^{r_n}$ contained in $v_K(K')$, a decomposition $\Delta_i=\coprod_j \Delta_{ij}$ and étale morphisms defined by functions $f_{ij}(\xi_{ij})=\xi_{ij}^{d_{ij}}(1+h_{ij})$, which are the restrictions of $f_{K'}$ to the open annuli $\Delta_{ij}= A_{K'}(r_{i-1}/d_{ij}, r_i/d_{ij})$, $j=1, \ldots, \delta_f(r_{i-1}, r_i)$. For $t\in \rbrack r_i, r_{i-1}\lbrack \cap \Q$, we see from this decomposition that $X_{K'}^{[t]}=\coprod_j \Delta_{ij}^{(t)}$, where $\Delta_{ij}^{(t)}=A_{K'}(t/d_{ij}, t/d_{ij})$. It follows that $\O(X_{K'}^{[t]})=\prod_j \O(\Delta_{ij}^{(t)})$ and thus $\O^{\circ}(X_{K'}^{[t]})=\prod_j \O^{\circ}(\Delta_{ij}^{(t)})$ (notation of \ref{Couronne Epaisseur Nulle}). Therefore, we have
\begin{equation}
	\label{DérivéeLutke 2}
\mathfrak{d}_f[t]=\prod_j \mathfrak{d}_{f_{ij}\lvert \Delta_{ij}^{(t)}}[t].
\end{equation}
In the proof of \ref{FormuleLutke}, we saw that the ideal $\mathfrak{d}_{f_{ij}\lvert \Delta_{ij}^{(t)}}(t)$ is generated by $(d_{ij}+\nu_{ij})h_{\nu_{ij}}\pi^{\nu_{ij}t} T_{ij}^{\d_{ij}\sigma_{ij}}$, where $T_{ij}$ is a coordinate of $\Delta_{ij}^{(t)}\cong \Sp(K'\{T_{ij}\})$, $h_{\nu_{ij}}$ is the coefficient of $h_{ij}$ indexed by $\nu_{ij}$ and $\sigma_{ij}$ is the order of the derivative $f_{ij}'(\xi_{ij})$.
Combining this with \eqref{DérivéeLutke 2} and \eqref{Couronne Epaisseur Nulle 4} yields
\begin{equation}
	\label{DérivéeLutke 3}
\lvert \mathfrak{d}_f(t)\lvert_{\rm sup}=\lvert \mathfrak{d}_f[t]\lvert_{\rm sup}=\lvert \pi\lvert^{\nu_i t} \prod_j \lvert (d_{ij}+\nu_{ij})h_{\nu_{ij}}\vert,
\end{equation}
where $\nu_i=\nu_{i1}+\cdots+\nu_{i\delta_f(i)}=\sigma_i - d+\delta_f(i)$ \eqref{total order 1}. Then, with \eqref{DiscPartial 2}, we see that $\partial_f^{\alpha}(t)=\nu_i t + c_i$, where $c_i$ is the constant $v_K\left(\prod_j (d_{ij}+\nu_{ij})h_{\nu_{ij}}\right)$. This finishes the proof.
\end{proof}

\subsection{} \label{RelevementsPoints geometriques}
Let $r\geq 0$ be a rational number and resume the notation and assumption on $f:X\to D$ and $K'$ from \ref{DiscPartial}. Recall, from \ref{ValuationDisqueUnité} (and \ref{ALFormelEtale}), that we have a geometric point $\overline{x}\to \mathfrak{D}_{K'}$ and a height $1$ prime ideal $\p$ of $\O_{\mathfrak{D}_{K'}, \overline{x}}$ above $\m_K$. Through the renormalization $\mathfrak{D}_{K'}^{(r)}\xrightarrow{\sim} \mathfrak{D}_{K'}, ~ \xi\mapsto \xi/\pi^r$, they induce a geometric point $\overline{x}\to \mathfrak{D}_{K'}^{(r)}$ and a height $1$ prime ideal of $A_{\overline{x}}=\O_{\mathfrak{D}_{K'}^{(r)}, \overline{x}}$ which we again denote by $\p$. We let $S_{f, K'}^{(r)}$ be the set of couples $\tau= (\overline{x}_{\tau}, \p_{\tau})$, where $\overline{x}_{\tau}=\overline{x}\to\mathfrak{X}_{K'}^{(r)}$ is a geometric point (of the special fiber) of the normalized integral model $\mathfrak{X}_{K'}^{(r)}$ of $X_{K'}^{(r)}$, above $\overline{x}\to \mathfrak{D}_{K'}^{(r)}$,  and $\p_{\tau}$ is a height $1$ prime ideal of $B_{\overline{x}_{\tau}}=\O_{\mathfrak{X}_{K'}^{(r)}, \overline{x}_{\tau}}$ above $\p$.
Note that, as the morphism $\widehat{f}^{(r)}: \mathfrak{X}_{K'}^{(r)}\to \mathfrak{D}_{K'}^{(r)}$ is finite and flat (\ref{ModelsFormels} $(1)$), so is the induced morphism on special fibers which is then surjective; hence, $S_{f, K'}^{(r)}$ is not empty.
If $K''/K'$ is a finite extension, then the special fiber of $\mathfrak{X}_{K'}^{(r)}$ is canonically isomorphic to the special fiber of the base change to $K''$ (\ref{Proprietes formelles objets} (iii)) and the height $1$ prime ideals of $B_{\overline{x}_{\tau}}$
are the minimal primes of $B_{\overline{x}_{\tau}}/\pi B_{\overline{x}_{\tau}}\xrightarrow{\sim} \O_{\mathfrak{X}_{s'}, \overline{x}_{\tau}}$; it follows that we have a canonical bijection $S_{f, K''}^{(r)}\xrightarrow{\sim} S_{f, K'}^{(r)}$. Thus $S_{f, K'}^{(r)}$ is independent of the chosen $r$-admissible extension $K'$ and we subsequently drop $K'$ from the notation.
For $\tau= (\overline{x}_{\tau}, \p_{\tau})$ in $S_f^{(r)}$, we get a $\Z^2$-valuation ring $V_r(\tau)=V_{B_{\overline{x}_{\tau}}}(\p_{\tau})$ (\ref{(V) Independant y}). The henselization $V^h_r$ (resp. $V_r^h(\tau)$) of $V_r$ (resp. $V_r(\tau)$) is a henselian $\Z^2$-valuation ring (\ref{(V) Independant y}) whose field of fractions is denoted $\K_r^h$ (resp. $\K^h_{r, \tau}$).
By \ref{Produit Locaux Affine} and \ref{etaleseparable}, $\widehat{f}^{(r)}$ induces a morphism $A_{\overline{x}}\to B_{\overline{x}_{\tau}}$ in $\widehat{\mathcal{C}}_{K'}$ (\ref{Category CK formelle}). By functoriality (\ref{(V) Fonctoriel}) and \ref{V^h MonogenicLibre}, the latter morphism gives rise to a monogenic integral extension of henselian $\Z^2$-valuation rings $V_r^h(\tau)/V^h_r$, with $V_r^h(\tau)$ a free $V^h_r$-module of finite rank. We denote again by $v_r : \K_r^{h\times}\to \Q\times\Z$, $v_r^{\alpha}: \K_r^{h\times}\to \Q$ and $v_r^{\beta}: \K_r^{h\times}\to \Z$ the maps respectively induced by $v_r, v_r^{\alpha}$ and $v_r^{\beta}$ (\ref{(Vr)}) through the canonical isomorphism of value groups $\Gamma_{V_r}\xrightarrow{\sim} \Gamma_{V_r^h}$ (\ref{ExtensionValHenselisé}). We note that the residue field of $V_r^h$ at its height $1$ prime ideal has degree of imperfection $p$ and thus $V_r^h$ satisfy the assumptions of \ref{RemPdim}.

\subsection{}
	\label{dfs} 
With the notation of \ref{RelevementsPoints geometriques} above, and as in \ref{ds (BA)}, we define the integer 
\begin{equation}
	\label{dfs 1}
d_{f, s}(r, K')=\sum_{\tau \in S^{(r)}} v_r^{\beta}(c(V_r^h(\tau)/V^h_r)),
\end{equation} 
where $c(V_r^h(\tau)/V^h_r)$ is a generator of the discriminant ideal of $V_r^h(\tau)/V^h_r$.

\begin{prop}
	\label{DiscEgalite}
For $r\in \Q_{\geq 0}$, we have the following equalities of integers
\begin{equation}
	\label{DiscEgalite 1}
\sum_{j=1}^N d_{\eta, \overline{x}'_j}^{(r)}= d_{f, s} (r, K') + 2\sum_{j=1}^N \delta_{\overline{x}'_j}^{(r)},
\end{equation}
where $\overline{x}'_1, \ldots, \overline{x}'_N$ are the geometric points of $\mathfrak{X}_{K'}^{(r)}$ above $\overline{x}$,  $\delta_{\overline{x}'_j}^{(r)}=\delta (\O_{\mathfrak{X}^{(r)}, \overline{x}'_j})$ is defined as in {\rm \ref{(V) KatoGeneral}} and  $d_{\eta, \overline{x}'_j}^{(r)}=d_{\eta}(\O_{\mathfrak{X}^{(r)}, \overline{x}'_j}/\O_{\mathfrak{D}^{(r)}, \overline{x}})$ is defined as in {\rm \eqref{d eta (BA)}}.
\end{prop}

\begin{proof}
We denote $\mathcal{A}=\O^{\circ}(D_{K'}^{(r)})$, $\mathcal{B}=\O^{\circ}(X_{K'}^{(r)})$, $A_{\overline{x}}= \O_{\mathfrak{D}^{(r)}, \overline{x}}$ and $B_j=\O_{\mathfrak{X}^{(r)}, \overline{x}'_j}$. By \ref{ModelsFormels} $(2)$, the $\mathcal{S}'$-adic morphism $\widehat{f}^{(r)}: \mathfrak{X}_{K'}^{(r)}\to \mathfrak{D}_{K'}^{(r)}$ is finite. We can thus apply \ref{Produit Locaux Affine} to $\widehat{f}^{(r)}$, $\overline{x}$ of $\mathfrak{D}_{K'}^{(r)}$ and  $\overline{x}'_1, \ldots, \overline{x}'_N$ of $\mathfrak{X}_{K'}^{(r)}$, and obtain
\begin{equation}
	\label{DiscEgalite 2}
A_{\overline{x}}\otimes_{\mathcal{A}}\mathcal{B}= B_1\times \cdots\times B_N.
\end{equation}
Moreover, by Proposition \ref{ModelsFormels}, for each $j$, $(\mathfrak{X}_{K'}^{(r)}/\mathcal{S}', \overline{x}'_j)$ satisfies property $(P)$ from \ref{(V) formel} in an open neighborhood of $\overline{x}'_j$. Hence, $\widehat{f}^{(r)}$ induces a morphism $A_{\overline{x}} \to B_j$ in $\widehat{\mathcal{C}}_{K'}$. Since $A_{\overline{x}}$ and $A_{\overline{x}}/(\pi)$ are regular, its follows from \ref{Formule LocalGlobal de Kato} that the discriminant $\mathfrak{d}_{B_j/A_{\overline{x}}}$ is well-defined. Then, from \eqref{DiscEgalite 2}, viewing $\mathfrak{d}_f(r)$ in $A_{\overline{x}}\supset \mathcal{A}$ (\ref{(Vr)}), we get
\begin{equation}
	\label{DiscEgalite 3}
\mathfrak{d}_f(r)=\mathfrak{d}_{B_1/A_{\overline{x}}}\cdots \mathfrak{d}_{B_N/A_{\overline{x}}}.
\end{equation}
On the one hand, as $d_{\eta, \overline{x}'_j}^{(r)}=d_{\eta}(B_j/A_{\overline{x}})=v_r^{\beta}(\mathfrak{d}_{B_j/A})$ by \ref{d eta ord p} and \ref{Formule LocalGlobal de Kato}, it then follows from \eqref{DiscEgalite 3} that $\partial_f^{\beta}(r)=\sum_j d_{\eta, \overline{x}'_j}^{(r)}$.
On the other hand, \eqref{DiscEgalite 3} and \ref{Formule LocalGlobal de Kato} also imply that
\begin{equation}
	\label{DiscEgalite 4}
\partial_f^{\beta}(r)=\sum_{j=1}^N \left(d_s(B_j/A_{\overline{x}}) + 2\delta(B_j)\right)=d_{f, s}(r, K') + 2\sum_{j=1}^N \delta(B_j),
\end{equation}
where the last equality directly uses the definition of $d_s(B_j/A_{\overline{x}})$ given in \eqref{ds (BA) 1}. This establishes \eqref{DiscEgalite 1}.
\end{proof}

\begin{rem}
	\label{Independance de dfs K}
The integer $d_{f, s}(r, K')$ is independent of the choice of the $r$-admissible extension $K'$ of $K$ and thus is simply denoted by $d_{f, s}(r)$. Indeed, in the proof of \ref{DiscEgalite}, we have shown that
\begin{equation}
	\label{Independance de dfs K 1}
\partial_f^{\beta}(r)= d_{f, s} (r, K') + 2\sum_{j=1}^N \delta_{\overline{x}'_j}^{(r)},
\end{equation}
and the $\delta_{\overline{x}'_j}^{(r)}$'s as well as $\mathfrak{d}_f(r)$ (hence $\partial_f^{\beta}(r)$ too) have already been seen, in \eqref{(V) KatoGeneral 1} and \eqref{DiscrimRig 1} respectively, to be independent of the chosen $r$-admissible $K'$.
\end{rem}

\begin{prop}\label{Vanishing Cycles Lutke}
We keep the notation of {\rm \ref{DérivéeLutke}} and {\rm \ref{DiscEgalite}}. Moreover, we assume that $X$ is connected and has trivial canonical sheaf. For $i=1,\ldots, n$ and $t\in \rbrack r_i, r_{i-1}\lbrack \cap \Q$, we have the following equality
\begin{equation}
	\label{Vanishing Cycles Lutke 1}
	\sum_{j=1}^N \left(d_{\eta, \overline{x}'_j}^{(t)}- 2\delta_{\overline{x}'_j}^{(t)}+ \lvert P_{s, \overline{x}'_j}^{(t)}\lvert\right)= \sigma_i + \delta_f(i),
\end{equation}
where the finite set $P_{s, \overline{x}'_j}^{(t)}=P_s(\O_{\mathfrak{X}^{(t)}, \overline{x}'_j})$ is defined as in {\rm \ref{(V) KatoGeneral}}.
\end{prop}

\begin{proof}
The integer $i$ and the rational number $t\in \rbrack r_i, r_{i-1}\lbrack \cap \Q$  are \textit{fixed} throughout the proof.
We recall from \ref{Couronne Epaisseur Nulle} that $D^{[t]}$ is the annulus of radius $\lvert\pi\lvert^t$ with $0$-thickness, $X^{[t]}$ is its inverse image by $f$, and $\mathfrak{D}_{K'}^{[t]}$ and $\mathfrak{X}_{K'}^{[t]}$ are their respective normalized integral models over $K'$. From \ref{DecompositionCouronnes} and \ref{Couronne Epaisseur Nulle}, we see that
\begin{equation}
\label{Vanishing Cycles Lutke 2}
X_{K'}^{[t]}=\coprod_{j=1}^{\delta_f(i)} \Delta_{ij}^{(t)}, \quad {\rm with} \quad  \Delta_{ij}^{(t)}=A\left(\frac{t}{d_{ij}}, \frac{t}{d_{ij}}\right)=D_{K'}^{[t/d_{ij}]}.
\end{equation}
We get from this the decomposition $\mathfrak{X}_{K'}^{[t]}=\coprod_j \widehat{\Delta}_{ij}^{(t)}$, where the normalized integral model $\widehat{\Delta}_{ij}^{(t)}=\Spf(\O^{\circ}(\Delta_{ij}^{(t)}))$ of $\Delta_{ij}^{(t)}$ is the formal annulus of radius $\lvert\pi\lvert^{t/d_{ij}}$ with $0$-thickness, defined over $K'$, with $K'$ $t$-admissible for $f$ (\ref{radmissible}), and isomorphic to $\Spf(\O_{K'}\{T_j, T_j^{-1}\})$. To get a formal compactification of $\mathfrak{X}_{K'}^{(t)}$, for each $j$, we glue $\mathfrak{X}_{K'}^{(t)}$ and a formal closed disc $\mathfrak{D}_{ij}^{(t)}=\widehat{\mathbb{A}}_{K'}^1=\Spf(\O_{K'}\{S_j\})$ along the boundary $\widehat{\Delta}_{ij}^{(t)}=\Spf(\O_{K'}\{S_j, S_j^{-1}\})$ \eqref{Couronne Epaisseur Nulle 2}, with gluing map $T_j\mapsto S_j^{-1}$. The resulting formal relative curve
\begin{equation}
	\label{Vanishing Cycles Lutke 3}
\mathfrak{Y}_{K'}^{(t)}=(\mathfrak{X}_{K'}^{(t)}\cup (\coprod_j \mathfrak{D}_{ij}^{(t)}))\big/\sim_{\mathfrak{X}^{[t]}}~\to \mathcal{S}'=\Spf(\O_{K'})
\end{equation}
has smooth rigid fiber and contains $\mathfrak{X}_{K'}^{(t)}$ as a formal open subscheme. As $\mathfrak{X}_{K'}^{(t)}$ is normal (\ref{ModelsFormels}(3)), $\mathfrak{Y}_{K'}^{(t)}$ is also normal. Its special fiber $\mathfrak{Y}_{s'}^{(t)}$ is the gluing of $\mathfrak{X}_{s'}^{(t)}$ with the disjoint union of $\delta_f(i)$ copies of $\mathbb{A}_k^1$ along the overlap $\mathfrak{X}_{s'}^{[t]}$, which is a disjoint union of $\delta_f(i)$ copies of $\mathbb{G}_{m, k}$, with gluing map $T_j\to S_j^{-1}$ on each $\mathbb{G}_{m, k}$. It follows that $\mathfrak{Y}_{s'}^{(t)}$ is a proper, hence projective, $k$-curve.
Moreover, by construction, the singular locus of $\mathfrak{Y}_{s'}^{(t)}$ is contained in the set $\mathfrak{X}_{s'}^{(t)}-\mathfrak{X}_{s'}^{[t]}$. As the Cartesian square \eqref{Couronne Epaisseur Nulle 3} reduces to a similar Cartesian square on special fibers, the latter set lies over the origin $\mathfrak{D}_{s'}^{(t)}-\mathfrak{D}_{s'}^{[t]}=\{x\}$.
By Grothendieck's algebraization theorem, there exists a relative proper algebraic curve $Y_{S'}^{(t)}$ over $S'=\Spec(\O_{K'})$ whose formal completion along its special fiber is $\mathfrak{Y}_{K'}^{(t)}$ \cite[5.4.5]{EGA.III}. As the rigid fiber $\mathfrak{Y}_{\eta'}^{(t)}$ of $\mathfrak{Y}_{K'}^{(t)}$ is smooth, so is the generic fiber $Y_{\eta'}^{(t)}$ of $Y_{S'}^{(t)}$.
 It follows from our assumptions that the canonical sheaf of $X_{K'}^{(t)}$ is trivial; so there exists a global section $\omega \in \Gamma(X_{K'}^{(t)}, \Omega^1_{X/K})$ inducing a trivialization $\O_{X_{K'}^{(t)}}\xrightarrow[\times\omega]{\sim} \Omega^1_{X/K}\lvert X_{K'}^{(t)}$.
Thus we can write $df^{(t)}=f^{\dagger}\omega$, where $f^{\dagger}\in \Gamma(X_{K'}^{(t)}, \O_X)$.
As, for each $j$, both $\omega\lvert \Delta_{ij}^{(t)}$ and $dT_j$ trivialize $\Omega^1_{X/K}$ on $\Delta_{ij}^{(t)}$, we have $\omega\lvert \Delta_{ij}^{(t)}=u_j(T_j) dT_j$, for some $u_j(T_j)\in \Gamma(\Delta_{ij}^{(t)}, \O_{\Delta_{ij}^{(t)}})^{\times}$.
Hence, we deduce that $(f^{(t)})'(T_j)=u_j(T_j) f^{\dagger}\lvert \Delta_{ij}^{(t)}$.
We choose a point $y_j$ in the generic fiber $D_{ij}^{(t)}$ of $\mathfrak{D}_{ij}^{(t)}$ that is not in $\Delta_{ij}^{(t)}$. By the rigid Runge theorem \cite[3.5.2]{Raynaud-Abh}, we can then approximate $\widehat{f}^{(t)}: \mathfrak{X}_{K'}^{(t)}\to \mathfrak{D}_{K'}^{(t)}$ by the formal completion $\widehat{g}^{(t)}:\mathfrak{Y}_{K'}^{(t)}\to \widehat{\mathbb{P}}_{S'}^1$ of an algebraic morphism $g^{(t)}: Y_{S'}^{(t)}\to \mathbb{P}_{S'}^1$ satisfying $g^{(t)^{-1}}(\infty)\subset \{y_j, j=1, \ldots, \delta_f(i)\}$, such that the induced morphism $g^{(t)}_{\eta'}: \mathfrak{Y}_{\eta'}^{(t)}\to \mathbb{P}_{K'}^{1, {\rm rig}}$ on rigid fibers is meromorphic with poles at most at the $y_j$ and, on each $\Delta_{ij}^{(t)}$, we have
\begin{equation}
	\label{Vanishing Cycles Lutke 4}
\lvert g^{(t)}_{\eta'} - f^{(t)}\lvert_j ~< \lvert f^{\dagger}\lvert_j/\lvert u_j^{-1} (T_j)\lvert_{\sup},
\end{equation}
where $\lvert \cdot \lvert_j$ is defined as in \ref{Norme et Ordre}, namely the sup-norm of the restriction to $\Delta_{ij}^{(t)}$ ($i$ is fixed).
As for $f^{(t)}$, we have $dg_{\eta'}^{(t)}\lvert X_{K'}^{(t)}=g^{\dagger}\omega$, for some $g^{\dagger}\in \Gamma(X_{K'}^{(t)}, \O_X)$, and $(g_{\eta'}^{(t)})'(T_j)=u_j(T_j) g^{\dagger}\lvert \Delta_{ij}^{(t)}$.
Since $Y_{\eta'}^{(t)}$ is a proper smooth curve, hence projective, and $dg^{(t)}$ is a non-zero meromorphic section of the canonical sheaf $\Omega^1_{Y_{\eta'}^{(t)}/K'}$, we have
\begin{equation}
	\label{Vanishing Cycles Lutke 5}
	2g(Y_{\overline{\eta}'}^{(t)})-2\lvert \pi_0(Y_{\overline{\eta}'}^{(t)})\lvert=\deg({\rm div}(dg^{(t)}_{\eta'})),
\end{equation}
where $g(Y_{\overline{\eta}'}^{(t)})$ is the total genus of $Y_{\overline{\eta}'}^{(t)}$, i.e. the sum of the genera of its connected components. Let us compute the right-hand side of \eqref{Vanishing Cycles Lutke 5}.
Taking the derivative of a power series expansion of $g_{\eta'}^{(t)}-f^{(t)}$ on $\Delta_{ij}^{(t)}$ and using the strong triangle inequality gives
\begin{equation}
	\label{Vanishing Cycles Lutke 6}
\lvert (g_{\eta'}^{(t)})'(T_j) - (f^{(t)})'(T_j) \lvert_{\sup}\leq \lvert g^{(t)}_{\eta'} - f^{(t)}\lvert_j.
\end{equation}
Since $\lvert g^{\dagger} - f^{\dagger}\lvert_j\leq \lvert u_j^{-1}(T_j)\lvert_{\sup}\lvert (g_{\eta'}^{(t)} - f^{(t)})'(T_j) \lvert_{\sup}$ and $\lvert f^{\dagger}\lvert_j\leq \lvert u_j^{-1}(T_j)\lvert_{\sup}\lvert (f^{(t)})'(T_j) \lvert_{\sup}$,\eqref{Vanishing Cycles Lutke 4} and \eqref{Vanishing Cycles Lutke 6} yield both following inequalities
\begin{equation}
	\label{Vanishing Cycles Lutke 7}
\lvert (g_{\eta'}^{(t)})'(T_j) - (f^{(t)})'(T_j) \lvert_{\sup} < \lvert (f^{(t)})'(T_j)\lvert_{\sup}\quad {\rm and}\quad \lvert g^{\dagger} - f^{\dagger}\lvert_j < \lvert f^{\dagger}\lvert_j.
\end{equation}
Therefore, we also have $\lvert (g_{\eta'}^{(t)})'(T_j) \lvert_{\sup} =\lvert (f^{(t)})'(T_j)\lvert_{\sup}$ and $\lvert g^{\dagger} \lvert_j =\lvert f^{\dagger}\lvert_j$.
Hence, at each point of the normalization $\widetilde{\mathfrak{Y}}_{s'}^{(t)}$ of $\mathfrak{Y}_{s'}^{(t)}$, $g^{\dagger}$ and $f^{\dagger}$ have the same order as defined by \eqref{Norme et Ordre 1}, and so do $(g_{\eta'}^{(t)})'(T_j)$ and $(f^{(t)})'(T_j)$. It follows from \ref{BoschLutke}, that, for each $x'_j\in \mathfrak{X}_{s'}^{(t)}-\mathfrak{X}_{s'}^{[t]}$, we have $\deg({\rm div}(g^{\dagger})\lvert C_+(x'_j))=\deg({\rm div}(f^{\dagger})\lvert C_+(x'_j))$. Hence, as
\begin{equation}
	\label{Vanishing Cycles Lutke 8}
{\rm div}(dg^{(t)}_{\eta'})\lvert C_+(x'_j)={\rm div}(g^{\dagger})\lvert C_+(x'_j) + {\rm div}(\omega)\lvert C_+(x'_j),
\end{equation}
and similarly for $df^{(t)}$ and $f^{\dagger}$, we obtain
$\deg({\rm div}(dg^{(t)}_{\eta'})\lvert C_+(x'_j))=\deg({\rm div}(df^{(t)})\lvert C_+(x'_j))$.
Moreover, as $f^{(t)}$ is étale over $X_{K'}^{[t]}$, so is $g_{\eta'}^{(t)}$; hence, ${\rm div}(dg_{\eta'}^{(t)}\lvert X_{K'}^{(t)})$ is supported in the tube of $\mathfrak{X}_{s'}^{(t)}-\mathfrak{X}_{s'}^{[t]}$. Therefore, we have (see also \eqref{Formule LocalGlobal 1})
\begin{equation}
	\label{Vanishing Cycles Lutke 9}
\deg({\rm div}(dg^{(t)}_{\eta'}\lvert X_{K'}^{(t)}))=\sum_{j=1}^N \deg({\rm div}(df^{(t)})\lvert C_+(x'_j))=\sum_{j=1}^N d_{\eta, \overline{x}'_j}^{(t)}.
\end{equation}
We denote by $\Delta_{ij}^{-(t)}$ the annulus  $\Delta_{ij}^{(t)}$ seen as the boundary of the disc $D_{ij}^{(t)}$, with coordinate $S_j=T_j^{-1}$. Since $g_{\eta'}^{(t)}$ is étale over $\Delta_{ij}^{-(t)}$, ${\rm div}(dg_{\eta'}^{(t)})\lvert D_{ij}^{(t)}-\Delta_{ij}^{-(t)}$ is supported on $C_+(y_j)$. As $D_{ij}^{(t)}-\Delta_{ij}^{-(t)}=C_+(y_j)$, and $(g^{(t)})'(T_j)$ and $(f^{(t)})'(T_j)$ have the same order $\sigma_{ij}$ on the annulus $\Delta_{ij}^{(t)}$, Lemma \ref{BoschLutke} again yields
\begin{equation}
	\label{Vanishing Cycles Lutke 10}
\deg({\rm div}((dg^{(t)}_{\eta'}\lvert D_{ij}^{(t)}-\Delta_{ij}^{-(t)})))={\rm ord}_{y_j}((g_{\eta'}^{(t)})'(T_j^{-1}))=-2-\sigma_{ij}.
\end{equation}
Summing \eqref{Vanishing Cycles Lutke 10} over $j$ and adding \eqref{Vanishing Cycles Lutke 9}, we find at last that the total degree is
\begin{equation}
	\label{Vanishing Cycles Lutke 11}
\deg({\rm div}(dg^{(t)}_{\eta'}))=\sum_{j=1}^N d_{\eta, \overline{x}'_j}^{(t)} -\sigma_i -2\delta_f(i).
\end{equation}
Now, let $R\Psi$ be the nearby cycles functor associated to the proper structure morphism $Y_{S'}^{(t)}\to S'$ and let $\Lambda$ be a finite field of characteristic different from $p$. Denoting by $Z$ the closed subset $\mathfrak{X}_{s'}^{(t)}-\mathfrak{X}_{s'}^{[t]}$ of the special fiber $\mathfrak{Y}_{s'}^{(t)}\cong Y_{s'}^{(t)}$, $i : Z\to Y_{s'}^{(t)}$ the closed immersion and $j: U=Y_{s'}^{(t)} - Z \to Y_{s'}^{(t)}$ the inclusion of the complement, the long exact sequence of cohomology induced by the short exact sequence $0\to j_{!}(\Lambda_{\lvert U})\to \Lambda\to i_{*}(\Lambda_{\lvert Z})\to 0$ of sheaves on  $Y_{s'}^{(t)}$ gives the following equality of Euler-Poincaré characteristics
\begin{equation}
	\label{Vanishing Cycles Lutke 12}
\chi(Y_{s'}^{(t)}, \Lambda)=\chi_c (U, \Lambda_{\lvert U}) + \chi(Y_{s'}^{(t)}, i_{*}(\Lambda_{\lvert Z})),
\end{equation}
where $\chi_c (\cdot)$ is the Euler-Poincaré characteristic with compact support. As the residue field of the points in $Z$ is the algebraically closed field $k$, we get $\chi(Y_{s'}^{(t)}, i_{*}(\Lambda_{\lvert Z}))=\dim_{\Lambda} H_{\textrm{ét}}^0 (Z, \Lambda_{\lvert Z})=\lvert Z\lvert= N$. As $U$ is a disjoint union of $\delta_f(i)$ copies of $\mathbb{A}_k^1$, we see that
\begin{equation}
	\label{Vanishing Cycles Lutke 13}
\chi_c (U, \Lambda_{\lvert U}) =\delta_f(i)\cdot \chi_c (\mathbb{A}_k^1, \Lambda)=\delta_f(i)(\chi(\mathbb{P}_k^1, \Lambda)-1)=\delta_f(i).
\end{equation}
Since $Y_{S'}^{(t)}$ is normal, the strict localization$Y_{(\overline{x}')}^{(t)}$ at any geometric point $\overline{x}' \to Y_{s'}^{(t)}$ is also normal; hence, $Y_{(\overline{x}')}^{(t)}\times\eta'$ is reduced. Moreover, as $Y_{s'}^{(t)}\cong\mathfrak{Y}_{s'}^{(t)}$ is reduced, so is $Y_{(\overline{x}')}^{(t)}\times s'=Y_{s' (\overline{x}')}^{(t)}$. Therefore, applying \cite[18.9.8]{EGA.IV} to the flat local homomorphism $Y_{(\overline{x}')}^{(t)}\to S'$, we see that the Milnor tube $Y_{(\overline{x}')}^{(t)}\times \overline{\eta}'$ is connected. As $R^i\Psi(\Lambda_{\lvert Y_{\overline{\eta}'}^{(t)}})_{\overline{x}'}=H_{\textrm{ét}}^i(Y_{(\overline{x}')}^{(t)}\times \overline{\eta}', \Lambda)$ \cite[XIII, 2.1.4]{SGA7}, the sheaf $R^0\Psi(\Lambda_{\lvert Y_{\overline{\eta}'}^{(t)}})$ is thus isomorphic to $\Lambda_{\lvert Y_{s'}^{(t)}}$ and $R^i\Psi(\Lambda_{\lvert Y_{\overline{\eta}'}^{(t)}})=0$ for $i >1$ \cite[I, Théorème 4.2]{SGA7}. Moreover, $R^1\Psi(\Lambda_{\lvert Y_{\overline{\eta}'}^{(t)}})$ is concentrated in the singular locus of $Y_{s'}^{(t)}$ \cite[XIII, 2.1.5]{SGA7}, located in $Z$. It thus follows from \eqref{Vanishing Cycles Lutke 12} and \eqref{Vanishing Cycles Lutke 13} that
\begin{equation}
	\label{Vanishing Cycles Lutke 14}
N + \delta_f(i) -\sum_{j=1}^N \dim_{\Lambda} H_{\textrm{ét}}^1(Y_{(\overline{x}'_j)}^{(t)}\times \overline{\eta}', \Lambda)=\chi(Y_{s'}^{(t)}, R\Psi(\Lambda)).
\end{equation}
By the proper base change theorem, we also have the equality
\begin{equation}
	\label{Vanishing Cycles Lutke 15}
\chi(Y_{s'}^{(t)}, R\Psi (\Lambda))=\chi(Y_{\overline{\eta}'}^{(t)}, \Lambda)=2\lvert \pi_0(Y_{\overline{\eta}'}^{(t)})\lvert-2g(Y_{\overline{\eta}'}^{(t)}).
\end{equation}
It remains to link the cohomology group in \eqref{Vanishing Cycles Lutke 14} to $\delta_{\overline{x}'_j}^{(t)}$ and $P_{s, \overline{x}'_j}^{(t)}$ in the following way. As $\mathfrak{X}_{K'}^{(t)}$ is a formal open subscheme of $\mathfrak{Y}_{K'}^{(t)}$, we have $\O_{\mathfrak{X}^{(t)}, \overline{x}'_j}=\O_{\mathfrak{Y}^{(t)}, \overline{x}'_j}$. Then, \eqref{ALFE 2} gives that 
\begin{equation}
	\label{Vanishing Cycles Lutke 16}
\O_{\mathfrak{X}^{(t)}, \overline{x}'_j}/\m_{K'} \xrightarrow{\sim} \O_{\mathfrak{Y}_{s'}^{(t)}, \overline{x}'_j}=\O_{Y_{s'}^{(t)}, \overline{x}'_j}.
\end{equation}
Since, for $A\in {\rm Obj}(\mathcal{C}_{K'})$ (resp. ${\rm Obj}(\widehat{\mathcal{C}}_{K'}))$, $P_s(A)$ identifies with the set of minimal prime ideals of $A/\m_{K'}$, it then follows that
\begin{equation}
	\label{Vanishing Cycles Lutke 17}
\delta (\O_{\mathfrak{X}^{(t)}, \overline{x}'_j})=\delta(\O_{Y^{(t)}, \overline{x}'_j}) \quad {\rm and}\quad \lvert P_s(\O_{\mathfrak{X}^{(t)}, \overline{x}'_j})\lvert=\lvert P_s(\O_{Y^{(t)}, \overline{x}'_j})\lvert.
\end{equation}
As, locally around $\overline{x}'_j$, the couple $(Y^{(t)}/S', \overline{x}'_j)$ satisfy property $(P)$ in \ref{(V)}, \cite[Prop. 5.9]{K1} in conjunction with \eqref{Vanishing Cycles Lutke 17} implies that
\begin{equation}
	\label{Vanishing Cycles Lutke 18}
2\delta (\O_{\mathfrak{X}^{(t)}, \overline{x}'_j}) - \vert P_s(\O_{\mathfrak{X}^{(t)}, \overline{x}'_j})\lvert +1= \dim_{\Lambda} H_{\textrm{ét}}^1(Y_{(\overline{x}'_j)}^{(t)}\times \overline{\eta}', \Lambda).
\end{equation}
Finally, combining \eqref{Vanishing Cycles Lutke 5}, \eqref{Vanishing Cycles Lutke 11}, \eqref{Vanishing Cycles Lutke 14}, \eqref{Vanishing Cycles Lutke 15} and \eqref{Vanishing Cycles Lutke 18} yields \eqref{Vanishing Cycles Lutke 1}, which concludes the proof.
\end{proof}

\section{Group filtrations and Swan conductors.}
\label{Kato 1}

We recall here Kato's formalism for group filtrations and conductors \cite[Sections 1 and 2]{K1}. 
\subsection{} \label{Ordered QVS}
Let $\Gamma$ be a totally ordered $\Q$-vector space, with an order structure compatible with its $\Q$-vector space structure. It induces on the set $\Gamma\cup\{-\infty, \infty\}$ the structure of a totally ordered monoid with $-\infty$ and $\infty$ set as its minimum and maximum elements respectively.

\subsection{}\label{step functionintegral}
A function $g:\Gamma\to\Q$ is a \textit{step function} if there is a finite sequence $(a_i)_{0\leq i\leq n}$ of elements of $\Gamma$ such that $a_0 \leq a_1\leq\cdots \leq a_n$ and such that $g$ is constant on each open interval $]-\infty, a_0[, ]a_{i-1}, a_i[$ and $]a_n, \infty[$. For such a function $g$ that takes the value $c_i\in\Q$ on $]a_{i-1}, a_i[$, setting $a=a_0, b=a_n$, we can define the integral
\begin{equation}
	\label{step functionintegral 1}
\int_a^b g(t)dt=\sum_{i=1}^n c_i (a_i - a_{i-1}) \quad \in \Gamma.
\end{equation}
Given $a\leq b$ in $\Gamma$, the integral \eqref{step functionintegral 1} is independent of the choice of the sequence $(a_i)_{0\leq i\leq n}$, such that $ a_0=a, a_n=b$.

If $a> b$, we set $\int_a^b g(t)dt=-\int_b^a g(t)dt$. If the support of $g$ is bounded from above and $a\in \Gamma$, we set
$\int_a^{\infty} g(t)dt=\int_a^b g(t)dt$, where $b$ a big enough element of $\Gamma$.

\subsection{}
A function $h:\Gamma\to\Gamma$ is \textit{piecewise linear} if there is a finite sequence $(a_i)_{0\leq i\leq n}$ of elements of $\Gamma$ such that $a_0 \leq a_1\leq\cdots \leq a_n$ and, on each interval $I=]-\infty, a_0], ~ I_i=[a_{i-1}, a_i], ~ I_{n+1}=[a_n, \infty[$, we have $h(t)=b_i t + c_i$, for all $t\in I_i$, where $b_i\in \Q$ and $c_i \in\Gamma$.
 
\begin{lem}[{\cite[1.5]{K1}}]
	\label{bijquasilinear}
\begin{enumerate}
\item[$(1)$] If $h:\Gamma\to\Gamma$ is a bijective piecewise linear function, its inverse $h^{-1}$ is also piecewise linear.
\item[$(2)$] If $g:\Gamma\to\Q$ is a step function such that $g(t)> 0$ for all $t\in \Gamma$, then the function $h:\Gamma\to \Gamma$ defined by $h(t)=\int_a^t g(s)ds$, for a fixed element $a\in \Gamma$, is a bijective piecewise linear function.
\end{enumerate}
\end{lem}

\subsection{}\label{Filtrations}
Let $G$ be a finite group. An \textit{upper} (resp. \textit{lower}) \textit{filtration} on $G$ indexed by $\Gamma$ is a decreasing family of normal subgroups $(G^t)_{t\in\Gamma}$ (resp. $(G_t)_{t\in\Gamma}$) of $G$ indexed by $\Gamma$ such that $G=G^0$ (resp. $G=G_0$) and for each $\sigma\in G-\{1\}$, the set $\{t\in\Gamma\ ~|~\sigma\in G^t\}$ (resp. $\{t\in\Gamma\ ~|~\sigma\in G_t\}$) has a maximum element denoted $j_G(\sigma)$.

\subsection{}\label{Filtrations LowerToUpper}
For a lower filtration $(G_t)_{t\in\Gamma}$ on $G$, the \textit{associated upper filtration} $(G^t)_{t\in\Gamma}$ is obtained as follows. By Lemma \ref{bijquasilinear} $(2)$, the function $\varphi: \Gamma\to \Gamma$, defined by $\varphi(t)=\int_0^t \lvert G_s\lvert ds$, is a piecewise linear bijection. Let $\psi=\varphi^{-1}$ be its inverse and set $G^t=G_{\psi(t)}$ for any $t\in\Gamma$.

\begin{lem}[{\cite[2.3]{K1}}]
		\label{UpperToLower}
If $(G^t)_{t\in\Gamma}$ is an upper filtration on $G$ indexed by $\Gamma$, then there is a unique lower filtration $(G_t)_{t\in\Gamma}$ on $G$ indexed by $\Gamma$ such that $(G^t)_{t\in\Gamma}$ is the upper filtration associated to $(G_t)_{t\in\Gamma}$ in the sense of \ref{Filtrations LowerToUpper}. Explicitly, the function $\psi: \Gamma\to \Gamma$, defined by $\psi(t)=\int_0^t \lvert G^s\lvert^{-1}ds$, is a piecewise linear bijection. Denoting $\varphi=\psi^{-1}$, we have $G_t=G^{\varphi(t)}$.
\end{lem}

\subsection{} \label{Filtrations SubgQuot} For the rest of this section, let $(G_t)_{t\in\Gamma}$ be a lower filtration on $G$ and $(G^t)_{t\in\Gamma}$ the associated upper filtration (see \ref{Filtrations LowerToUpper}).

If $H$ is a subgroup of $G$, the \textit{induced lower filtration} on $H$ indexed by $\Gamma$ is defined by $H_t=G_t\cap H$, and the associated upper filtration formally given in \ref{Filtrations LowerToUpper} is called the \textit{induced upper filtration}. Any subgroup $H$ is implicitly assumed to be endowed with these filtrations.

If $H$ is a normal subgroup of $G$, the \textit{induced upper filtration} on $G/H$ indexed by $\Gamma$ is defined by $(G/H)^t=(G^t H)/H$, and the unique asociated lower filtrattion on $G/H$, given by Lemma \ref{UpperToLower}, is called the \textit{induced lower filtration} on $G/H$. Any quotient $G/H$ is implicitly assumed to be endowed with these filtrations.

\subsection{}\label{Swan Conductor}
Let $\Lambda$ be a field in which $\lvert G_t\lvert$ is invertible for any $t>0$, which is equivalent to $\lvert G^t\vert$ being invertible in $\Lambda$ for any $t>0$. For a $\Lambda[G]$-module of finite type $M$, we define its Swan conductor $\sw_G(M) \in \Gamma$ by
\begin{equation}
	\label{Swan Conductor 1}
\sw_G(M)=\int_0^{\infty} \dim_{\Lambda}(M/M^{G^t})dt=\int_0^{\infty} \lvert G_t\lvert \cdot\dim_{\Lambda}(M/M^{G_t})dt.
\end{equation}

\begin{lem}[{\cite[2.5]{K1}}] \label{SES Sum}
\begin{enumerate}
\item[$(1)$] For a short exact sequence of $\Lambda[G]$-modules of finite type
\begin{equation}
	\label{SES Sum 1}
0\to M'\to M\to M''\to 0,
\end{equation}
we have $\sw_G(M)=\sw_G(M')+ \sw_G(M'')$.
\item[$(2)$] If $\Lambda$ is of characteristic zero, then, in $\Lambda \otimes_{\Q}\Gamma$, we have
\begin{equation}
	\label{SES Sum 2}
	\sw_G(M)=\sum_{\sigma\in G, \sigma\neq 1}\left(\dim_{\Lambda} M-\chi_M(\sigma)\right)\otimes j_G(\sigma),
\end{equation}
where $\chi_M$ is the character of $M$.
\end{enumerate}
\end{lem}

\begin{lem}[{\cite[2.7]{K1}}] \label{SwanInduced}
Let $H$ be a subgroup of $G$.
\begin{enumerate}
\item[$(1)$] Denote by $\Lambda[G/H]$ the regular $\Lambda$-valued representation $\Ind_H^G 1_H$ of $G$. Then,
\begin{equation}
	\label{SwanInduced 1}
\sw_G(\Lambda[G/H])=\sum_{\sigma\in G-H}j_G(\sigma).
\end{equation}
\item[$(2)$] If $M$ is a $\Lambda[H]$-module of finite type, then
\begin{equation}
	\label{SwanInduced 2}
\sw_G\left(\Lambda[G]\otimes_{\Lambda[H]} M\right)=[G:H]\sw_H(M)+(\dim_{\Lambda}M) \sw_G(\Lambda[G/H]).
\end{equation}
\item[$(3)$] If the subgroup $H$ is normal and $M$ is a $\Lambda[G/H]$-module of finite type, then
\begin{equation}
	\label{SwanInduced 3}
\sw_G(M)=\sw_{G/H}(M).
\end{equation}
\end{enumerate}
\end{lem}

\subsection{} \label{SwanDim}
 By Lemma \ref{SES Sum} $(1)$, the Swan conductor extends to the Grothendieck group $R_{\Lambda}(G)$ of $\Lambda[G]$-modules of finite type. If $\chi\in R_{\Lambda}(G)$ is of dimension zero, then \eqref{SwanInduced 2} gives 
\begin{equation}
	\label{SwanDim 1}
\sw_G(\Ind_H^G \chi)=[G:H] \sw_H(\chi).
\end{equation}

\begin{lem}[{\cite[2.9]{K1}}]
	\label{Composition VarphiPsi}
Let $H$ be a normal subgroup of $G$, and let $\varphi_G$ and $\psi_G$ (resp. $\varphi_H$ and $\psi_H$, resp. $\varphi_{G/H}$ and $\psi_{G/H}$) be the bijective functions $\Gamma\to \Gamma$ associated to the filtrations on $G$ $($resp. the induced filtrations on $H$, resp. the induced filtrations on $G/H)$. Then,
\begin{enumerate}
\item[$(1)$] $\varphi_G=\varphi_{G/H}\circ\varphi_H$ and $\psi_G=\psi_H\circ\psi_{G/H}$.
\item[$(2)$] If $t\in\Gamma$ and $s=\psi_{G/H}(t)$, then the induced upper filtration on $H$ is given by $H^s=G^t\cap H$.
\end{enumerate} 
\end{lem}

\section{Ramification of \texorpdfstring{$\mathbb{Z}^2$}{Lg}-valuation rings.}\label{Kato 2}

\subsection{} \label{Notations Ramif}
Let $V$ be a valuation ring with value group $\Gamma_V$, field of fractions $K$ and valuation map $v_K: K^{\times}\to \Gamma_V$. Let $L$ be a finite Galois extension of $K$ of group $G$ and $W$ the integral closure of $V$ in $L$. We assume that $W/V$ is a monogenic integral  extension of valuation rings (\ref{MonogenicIntegral}). We put $\Gamma=\Q\otimes_{\Z}\Gamma_V=\Q\otimes_{\Z}\Gamma_W$ (\textit{cf.} \ref{remrank}) and
let $v:L^{\times}\to \Gamma$ be the unique valuation of $L$ such that $v\lvert_{K^{\times}}=v_K$.
Let $\varepsilon$ be an element of $\Gamma$ such that, for any $\sigma\in G$ and $x\in W$, we have $v(\sigma(x)-x)\geq \lvert G\lvert^{-1}\varepsilon$ (e.g. $\varepsilon=0$). Then, set
\begin{equation}
	\label{iG}
i_G(\sigma)={\rm min}\{v(\sigma(x)-x)~|~ x\in W\} \in \Gamma_W\quad {\rm for}\quad \sigma\in G-\{1\},\quad i_G(1)=\infty.
\end{equation}
The minimum in \eqref{iG} exists and is equal to $v(\sigma(a)-a)$ for  any element $a$ of $W$ such that $W=V[a]$. Indeed, this follows readily from an induction argument using the  \textit{almost derivation formula}
\begin{equation}
	\label{presquedrivation}
\sigma(xy)-xy=(\sigma(x)-x)y + \sigma(x)(\sigma(y)-y).
\end{equation}
For $\sigma\in G$, set also
\begin{equation}
	\label{jG}
	 j_{G, \varepsilon}(\sigma)=i_G(\sigma) - \frac{\varepsilon}{\lvert G\lvert} \in \Gamma.
\end{equation}
Notice that $i_G(\sigma^{-1})=i_G(\sigma)$ and $j_{G, \varepsilon}(\sigma^{-1})=j_{G, \varepsilon}(\sigma)$. We now define the \textit{lower ramification filtration} on $G$ indexed by $\Gamma$ by setting 
\begin{equation}
	\label{DefFilt}
	G_{t, \varepsilon}=\{\sigma\in G~|~j_{G, \varepsilon}(\sigma)\geq t\}.
\end{equation}

\begin{lem}[{\cite[Lemma (3.2)]{K1}}]
	\label{Tate}
Let $V, K, W, L$ and $G$ be as above. Let $H$ be a subgroup of $G$, $K'=L^H$ the corresponding sub-extension of $L/K$ and $V'$ the integral closure of $V$ in $K'$. Then, for $\tau\in G/H-\{1\}$, the minimum element $i_{G/H}(\tau)$ of the subset $\{v(\tau(y)-y) ~|~ y\in V'\}$ of $\Gamma$ exists and equals $\sum_{\sigma\mapsto\tau}i_G(\sigma)$, where $\sigma$ runs over the representatives of $\tau$ in $G$. 
\end{lem}

\begin{cor}[{\cite[Corollary (3.3)]{K1}}]
We keep the assumptions of {\rm \ref{Tate}} and assume moreover that the subgroup $H$ is normal. Then, the induced lower filtration on $G/H$ {\rm (cf. \ref{Filtrations SubgQuot})} coincides with the lower ramification filtration on $G/H$ defined by $j_{G/H, \varepsilon}$.
\end{cor}

\subsection{} \label{hypotheses} For the rest of this section, we let $V$ be a \textit{henselian} $\Z^2$-valuation ring with field of fractions $K$, $L$ a finite Galois extension of $K$ of group $G$ and $W$ the integral closure of $V$ in $L$. We denote by $(0)\subsetneq \p\subsetneq\m$ and $(0)\subsetneq \p'\subsetneq\m'$ the prime ideals of $V$ and $W$ respectively. We assume that the residue field $\kappa(\m)$ is \textit{perfect}, that $\kappa(\m)=\kappa(\m')$, that $[L : K]=[\kappa(\p') : \kappa(\p)]$ and, if $\kappa(\p)$ has characteristic $p>0$, that $[\kappa(\p) : \kappa(\p)^p]=p$. Then, by \ref{Monogene}, \ref{RemCompSep} and \ref{RemPdim}, $W/V$ is a monogenic integral extension of $\Z^2$-valuation rings. Moreover,
as $[L:K]=[\Gamma_W:\Gamma_V]\cdot[\kappa(\m'):\kappa(\m)]$ and $[\kappa(\p'):\kappa(\p)]=e_{\p'/\p} f_{\p'/\p}$, where $e_{\p'/\p}$ and $f_{\p'/\p}=[\kappa(\m'):\kappa(\m)]$ are respectively the ramification index and residue degree of the extension of discrete valuation rings $V/\p\subset W/\p'$, we have (\ref{Monogene})
\begin{equation}
	\label{hypotheses 1}
[\Gamma_W:\Gamma_V]=\lvert G\lvert=e_{\p'/\p}.
\end{equation}
Recall from \ref{AlphaEtBeta} that $\varepsilon_K$ (resp. $\varepsilon_L$) corresponds to a generator of the unique non-trivial isolated subgroup of $\Gamma_V$ (resp. $\Gamma_W$), hence a uniformizer of $V/\p$ (resp. $W/\p'$). From \ref{Monogene} (iii'), we see that the ramification index of $W_{\p'}/V_{\p}$ is one and thus $W/V$ has (weak) ramification only on the second factor in \eqref{AlphaEtBeta 1}, which corresponds to the extension $V/\p\subset W/\p'$.
It follows that
 \begin{equation}
 	\label{hypotheses 2}
 \varepsilon_K=e_{\p'/\p} \varepsilon_L=\lvert G\lvert \varepsilon_L.
 \end{equation}
Therefore, we can put $\varepsilon=\varepsilon_K$ in \ref{Notations Ramif} and thus get a lower ramification filtration on the group $G$ \eqref{DefFilt}. For the remainder of the section, we write simply $i_G$ and $j_G$ for $i_{G, \varepsilon}$ and $j_{G, \varepsilon}$.

\subsection{}\label{UnitIsom}
 Let $U_L$ be the group of units in $W$
\begin{equation}
	\label{UnitIsom 1}
U_L=\{x\in L^{\times} ~ | ~ v(x)=0\}.
\end{equation}
We define a decreasing filtration $(U_t)_{t\in \Gamma_{W \geq 0}}$ on $U_L$ by $U_0=U_L$, and for $t>0$, $U_t=1+\m'_t$, where $\m'_t$ is the ideal of elements $x$ in $W$ such that $v(x)\geq t$. Since $\m'=\m'_{\varepsilon_L}$, reduction modulo $\m'$ yields a canonical group isomorphism 
\begin{equation}
	\label{UnitIsom 2}
U_0/U_{\varepsilon_L} \xrightarrow{\sim} \kappa(\m')^{\times}. 
\end{equation}
For $t>0$, the surjection from $U_t$ to $ \m'_t  $ given by $1+x\mapsto x$ also induces a group isomorphism
\begin{equation}
	\label{Unitisom 3}
	U_t/U_{t+\varepsilon_L} \xrightarrow{\sim} \m'_t/\m'_{t+\varepsilon_L}.
\end{equation}
The additive group $\m'_t/\m'_{t+\varepsilon_L}$ is non canonically isomorphic to the additive group $\kappa(\m')$ as follows. Choose $b\in \m'_t - \m'_{t+\varepsilon_L}$, and define a surjective map from $\m'_t$ to $W$ by $x\mapsto x/b$. Since $v(x/b)=v(x)-t$, this map is well defined and induces an isomorphism
\begin{equation}
	\label{UnitIsom 4}
\m'_t / \m'_{t+\varepsilon_L}\xrightarrow{\sim} \kappa(\m').
\end{equation}
\subsection{} \label{Theta morphism}
 Let $t\geq 0$ be in $\Gamma_W$. For $\sigma\in G_t$ and $e$ in $W$ such that $v(e)=\varepsilon_L$, \eqref{DefFilt} and \eqref{hypotheses 2} give
\begin{equation}
	\label{Theta morphism 1}
v(\sigma(e)/e - 1)=v(\sigma(e)-e)-\varepsilon_L\geq t.
\end{equation}
So $\sigma(e)/e \in U_t$, and we set $\theta_t(\sigma)=\sigma(e)/e  \mod U_{t+\varepsilon_L}$. The element $\theta_t(\sigma)$ is independent of the choice of $e$ in $W$ satisfying $v(e)=\varepsilon_L$. Indeed, any other such $e'$ is of the form $ue$, where $u$ is a unit in $W$, and since $\sigma\in G_t$, we have
\begin{equation}
	\label{Theta morphism 2}
v(\sigma(u)/u -1)=v(\sigma(u)-u)\geq t+\varepsilon_L, \quad {\rm hence}\quad \frac{\sigma(e')}{e'}=\frac{\sigma(e)}{e}\mod U_{t+\varepsilon_L}.
\end{equation} 
If $\tau \in G_t$, then, 
\begin{equation}
	\label{Theta morphism 3}
\frac{\tau\sigma(e)}{e}=\frac{\tau(\sigma(e))}{\sigma(e)}\cdot \frac{\sigma(e)}{e}, \quad {\rm and}\quad v(\sigma(e))=\varepsilon_L;
\end{equation} 
hence, by the aforementioned independence, we have $\theta_t(\tau\sigma)=\theta_t(\tau)\theta_t(\sigma)$. Thus, we have defined a group homomorphism 
\begin{equation}
	 \label{Theta morphism 4}
	\theta_t : G_t \to  U_t/U_{t+\varepsilon_L}
\end{equation}
Now it is clear that the subgroup $G_{t+\varepsilon_L}$ of $G_t$ is exactly the kernel of $\theta_t$. In conclusion we get an injective group homomorphism 
\begin{equation}
	\label{Theta morphism 5}
	\overline{\theta}_t : G_t/G_{t+\varepsilon_L}\hookrightarrow U_t/U_{t+\varepsilon_L}.
\end{equation}
 We thus have proved the following.
\begin{prop}
	\label{Theta Isom}
For any $t\in \Gamma_W$, the map $\overline{\theta}_t$ induces an isomorphism between $G_t/G_{t+\varepsilon_L}$ and a subgroup of $U_t/U_{t+\varepsilon_L}$.
\end{prop}

\begin{cor}
	\label{cyclic}
 The group $G/G_{\varepsilon_L}$ is cyclic of order prime to the characteristic of $\kappa(\m')$.
\end{cor}

\begin{proof}
The group $U_0/U_{\varepsilon_L}$ is identified with $\kappa(\m')^{\times}$ \eqref{UnitIsom 2}. Setting $t=0$ in proposition \ref{Theta Isom}, we see that $G/G_{\varepsilon_L}$ is thus a finite subgroup of the group of roots of  unity in $\kappa(\m')$. Therefore, it is cyclic of order prime to the characteristic of $\kappa(\m')$ \cite[Chap. V, \S  11, n°1, Théorème 1]{Bourbaki2}.
\end{proof}

\begin{cor}
	\label{cyclic1}
If the characteristic of $\kappa(\m')$ is $0$, then $G_{\varepsilon_L}=\{1\}$, and $G$ is cyclic.
\end{cor}

\begin{proof}
The proof is the same as in \cite[Chap. IV, \S 2, Proposition 7, corollaire 2]{Serre1}. From \eqref{Unitisom 3}, \eqref{UnitIsom 4}, Proposition \ref{Theta Isom} and the fact that $\kappa(\m')$ has no non trivial finite subgroup in characteristic zero, we deduce that $G_t=G_{t+\varepsilon_L}$ for any $t>0$; since these subgroups are trivial for $t$ large enough, we see that $G_{\varepsilon_L}=\{1\}$, and thus $G$ is cyclic by Corollary \ref{cyclic}.
\end{proof}

\begin{cor}
	\label{cyclic2}
If the characteristic of $\kappa(\m')$ is $p>0$, then the quotients $G_t/G_{t+\varepsilon_L}$, $t>0$ in $\Gamma_W$, are directs products of cyclic groups of order $p$, and $G_{\varepsilon_L}$ is a $p$-group.
\end{cor}

\begin{proof}
The proof proceeds as in \cite[Chap. IV, \S 2, Proposition 7, corollaire 3]{Serre1}. We get the first half of the lemma from \eqref{Unitisom 3}, \eqref{UnitIsom 4}, Proposition \ref{Theta Isom} and the fact that, in characteristic $p$, every subgroup of $\kappa(\m')$ is an $\mathbb{F}_p$-vector space, hence a direct product of cyclic groups of order $p$; the second half ensues because the order of $G_{\varepsilon_L}$ is the product of the orders of the $G_t/G_{t+\varepsilon_L}$ for $t>0$.
\end{proof}

\begin{rems}
	\label{puissance div}
\begin{enumerate}
\item[$(i)$]	 For every $\sigma\in G - G_{\varepsilon_L}$, it is clear that $i_G(\sigma)=\varepsilon_L$. For every $\sigma\in G_{\varepsilon_L}$ and every $n\in \Z-\{0\}$ prime to the characteristic of $\kappa(\m')$, we have $i_G(\sigma^n)=i_G(\sigma)$. Indeed, in positive residue characteristic $p>0$, if $\sigma$ is in $G_t - G_{t+\varepsilon_L}$ for some $t\geq \varepsilon_L$, i.e. if $\sigma\neq 0$ in $G_t/G_{t+\varepsilon_L}$, so is $\sigma^n$ since this quotient is $p$-group and $p\nmid n$.
\item[$(ii)$] We deduce from $(i)$ that $i_G(\sigma)=i_G(\tau)$ if $\sigma$ and $\tau$ generate the same subgroup of $G$.
\end{enumerate}
\end{rems}

\subsection{} \label{Artin and Swan characters}
 We keep the notation and assumptions of \ref{hypotheses}. Recall that $\varepsilon_L$ is the minimum of the set of positive elements of $\Gamma_W$.
 Define $a_G$ and $\sw_G$ to be the following functions from $G$ to $\Gamma$. For $\sigma \neq 1$ in $G$, we set (\eqref{iG}, \eqref{jG}) 
 \begin{gather}
 	\label{Artin and Swan characters 1}
 a_G(\sigma)=-i_G(\sigma) \quad {\rm and} \quad \sw_G(\sigma)=-j_G(\sigma).
 \end{gather}
We  set also
\begin{gather}
	\label{Artin and Swan characters 2}
a_G(1)=\sum_{\sigma\in G-\{1\}} i_G(\sigma) \quad {\rm and}\quad \sw_G(1)=\sum_{\sigma \in G-\{1\}} j_G(\sigma).
\end{gather}
Clearly, $a_G$ and $\sw_G$ are class functions on $G$. They satisfy the relation 
\begin{equation}
\label{Artin and Swan characters 3}
 a_G=\sw_G+ \varepsilon_L \cdot u_G
\end{equation}
where $u_G$ is the augmentation character of $G$.

\subsection{}\label{Pairing} Let $\ell$ be  a prime number and let $\overline{\Q}_{\ell}$ be an algebraic closure of $\Q_{\ell}$. For $\Lambda_{\Q}$ a finite extension of $\Q_{\ell}$ in $\overline{\Q}_{\ell}$ with valuation ring $\Lambda$, we denote by $R_{\Lambda_{\Q}}(G)$ (resp. $P_{\Lambda}(G)$) the Grothendieck group of finitely generated $\Lambda_{\Q}[G]$-modules (resp. of finitely generated projective $\Lambda[G]$-modules). For $\chi \in R_{\Lambda_{\Q}}(G)$ and $\psi$ a class function on $G$ with values in $\Gamma$, we define the pairing
\begin{equation}
	\label{Pairing 1}
\langle\langle \psi,\chi\rangle\rangle= \sum_{\sigma \in G} \psi(\sigma^{-1})\otimes{\rm tr}_{\chi}(\sigma) \quad \in \Gamma\otimes_{\Q}\Lambda_{\Q}.
\end{equation}

\begin{rem}
	\label{Concordance}
If $M$ is a finitely generated $\Lambda_{\Q}[G]$-module of character $\chi_M$, then
\begin{equation}
	\label{Concordance 1}
	\sw_G(M)=\langle\langle \sw_G,\chi_M\rangle\rangle.
\end{equation}
Indeed, this follows from expanding the sum in the formula \eqref{SES Sum 2}.
\end{rem}

\begin{thm}[{\cite[Theorem 4.7]{K1}}]
\label{HasseArf}
Assume that $G$ is abelian. Then, for any $\Lambda_{\Q}[G]$-module of finite type $M$, we have 
\begin{equation}
	\label{HasseArf 1}
\sw_G(M) \in \Gamma_V \subsetneq \Gamma.
\end{equation}
\end{thm}

This is an analogue of the Hasse-Arf theorem. Kato's proof of this result relies on an interpretation of the ramification filtration $(G^t)$ in terms of a filtration of the Milnor $K$-group $K_2(K)$ via a higher reciprocity map \cite[Theorem 4.4]{K1}.

\subsection{}\label{ArtinSwan}
For the rest of this section, we let $\pi$ be a uniformizer of $V_{\p}$.
Recall from \ref{AlphaEtBeta} that $\pi$ induces group homomorphisms $\alpha, \beta:\Gamma_V \to \Z$ characterized respectively by $\alpha(v(\pi))=1$, $\alpha(\varepsilon_K)=0$, $\beta(v(\pi))=0$ and $\beta(\varepsilon_K)=1$. We also denote by $\alpha$ and $\beta$ their extensions to $\Gamma$. Notice that, because $\varepsilon_K=\lvert G\lvert\cdot\varepsilon_L$, $\beta$ sends $\Gamma_W$ in $\frac{1}{\lvert G\lvert} \cdot\Z$. The functions $a_G$ and $\sw_G$ induce class functions 
\begin{equation}
	\label{ArtinSwan 1}
a_G^{\beta}=\beta\circ a_G: G \to \Z \quad {\rm and} \quad \sw_G^{\beta}=\beta\circ \sw_G : G \to \Z.
\end{equation}
The relation \eqref{Artin and Swan characters 3} is carried into
\begin{equation}
	\label{ArtinSwan 2}
	 a_G^{\beta}=\sw_G^{\beta} + u_G.
\end{equation}
Composing with $\alpha$, we see that $\alpha\circ a_G=\alpha\circ \sw_G$ which also defines a class function $a_G^{\alpha}$ on $G$. Note that, while $a_G^{\alpha}$ is independent of the chosen uniformizer $\pi$, $a_G^{\beta}$ and $\sw_G^{\beta}$ do depend on $\pi$ a priori \eqref{AlphaEtBeta}.

\begin{rem}\label{Pairing SerreKato}
A straightforward calculation shows that
\begin{equation}
	\label{Pairing SerreKato 1}
\beta(\langle\langle \sw_G,\chi\rangle\rangle)=\lvert G\lvert\cdot \langle \sw_G^{\beta},\chi\rangle,
\end{equation}
for any $\chi \in R_{\Lambda_{\Q}}(G)$, where $\langle\cdot, \cdot \rangle$ denotes the usual pairing between $\Lambda_{\Q}$-valued class functions on $G$ \cite[2.2, Remarques] {Serre2}. The same is true with $a_G$ and $a_G^{\alpha}$.
\end{rem}

\begin{lem}[{\cite[Lemme 3.5]{K1}}]
	\label{ArtinSwanExplicit}
Let $H$ be a subgroup of $G$, $K'$ the corresponding sub-extension of $L/K$ and $V'$ the valuation ring of $K'$. Assume that the extension of valuation rings $V'/V$ is monogenic integral $($hence $V'$ is a free $V$-module of finite rank by {\rm \ref{Monogene}}$)$ and denote by $\mathfrak{d}_{V'/V}$ a generator of the discriminant ideal of the extension $V'/V$.
Then, with the notation of {\rm \ref{AlphaEtBeta}}, we have
\begin{equation}
	\label{ArtinSwanExplicit 1}
v(\mathfrak{d}_{V'/V})=[G:H] \sum_{\sigma\in G-H} i_G(\sigma),
\end{equation}
\begin{equation}
	\label{ArtinSwanExplicit 2}
\langle\langle a_G, \Lambda_{\Q}[G/H]\rangle\rangle=v(\mathfrak{d}_{V'/V}),
\end{equation}
\begin{equation}
	\label{ArtinSwanExplicit 3}
\langle\langle \sw_G, \Lambda_{\Q}[G/H]\rangle\rangle=v(\mathfrak{d}_{V'/V}) - ([G:H]-1)\varepsilon_K,
\end{equation}
\begin{equation}
	\label{ArtinSwanExplicit 4}
\langle a_G^{\alpha}, \Lambda_{\Q}[G/H]\rangle=\frac{1}{\lvert G\lvert} v^{\alpha}(\mathfrak{d}_{V'/V}),
\end{equation}
\begin{equation}
	\label{ArtinSwanExplicit 5}
\langle \sw_G^{\beta}, \Lambda_{\Q}[G/H]\rangle=\frac{1}{\lvert G\lvert} v^{\beta}(\mathfrak{d}_{V'/V}) -\frac{1}{\lvert H\lvert} + \frac{1}{\lvert G\lvert}.
\end{equation}
\end{lem}

\begin{proof}
Let $b$ in $V'$ be such that $V'=V[b]$. We see from \cite[Chap. III, \S 6, Prop. 11, Cor. 2]{Serre1} that the different $\D_{V'/V}$ is generated by 
\begin{equation}
	\label{ArtinSwanExplicit 6}
x=\prod_{\tau \in G/H-\{1\}}(\tau(b)-b)
\end{equation}
Hence, $\mathfrak{d}_{V'/V}$ is generated by $N_{K'/K}(x)$, where $N_{K'/K}: K^{'\times}\to K^{\times}$ is the norm map. Since $V'/V$ is monogenic integral, the extension $V'_{\q}/V_{\p}$ has ramification index $1$ and thus residue degree $\lvert G/H\lvert$. So,
\begin{equation}
	\label{ArtinSwanExplicit 7}
v(\mathfrak{d}_{V'/V})=v(N_{K'/K}(x))=\lvert G/H\lvert v(x)=\lvert G/H\lvert \sum_{\tau \in G/H-\{1\}} i_{G/H}(\tau).
\end{equation}
Therefore, \eqref{ArtinSwanExplicit 1} follows from \eqref{ArtinSwanExplicit 4} and Lemma \ref{Tate}. We deduce \eqref{ArtinSwanExplicit 3} from \eqref{SwanInduced 1}, \eqref{jG} and \eqref{ArtinSwanExplicit 1}. Since $a_G=\sw_G + \varepsilon_L u_G=\sw_G + (\varepsilon_K/\lvert G\lvert) u_G$ and $\langle\langle u_G, \Lambda_{\Q}[G/H]\rangle\rangle=\lvert G\lvert ([G:H]-1)$, \eqref{ArtinSwanExplicit 2} follows from \eqref{ArtinSwanExplicit 3}. Now, \eqref{ArtinSwanExplicit 4} follows from \ref{Pairing SerreKato} and \eqref{ArtinSwanExplicit 2}. As $\beta(\varepsilon_K)=\lvert G\lvert$ (\ref{ArtinSwan}), the identity \eqref{ArtinSwanExplicit 5} is deduced from \eqref{Pairing SerreKato 1} and \eqref{ArtinSwanExplicit 3}.
\end{proof}

\begin{lem}
	\label{ArtinSwanDifferente}
Let $H$ be a subgroup of $G$, $K'$ the corresponding sub-extension of $L/K$ and $V'$ the valuation ring of $K'$.

$(i)$ Denoting $r_H$ the character of the regular representation of $H$, we have the following relations between class functions
\begin{equation}
	\label{ArtinSwanDifferente 1}
a_G^{\alpha}\lvert H=\frac{1}{\lvert G\lvert} v^{\alpha}(\mathfrak{d}_{V'/V})\cdot r_H + a_H^{\alpha}, \quad a_G^{\beta} \lvert H=\frac{1}{\lvert G\lvert} v^{\beta}(\mathfrak{d}_{V'/V})\cdot r_H + a_H^{\beta}
\end{equation}
\begin{equation}
	\label{ArtinSwanDifferente 2}
\sw_G^{\beta} \lvert H=(\frac{1}{\lvert G\lvert}v^{\beta}(\mathfrak{d}_{V'/V})+1-[G : H])\cdot r_H + \sw_H^{\beta}.
\end{equation}
$(ii)$ Assume that $H$ is normal. We have the following relations between class functions
\begin{equation}
	\label{ArtinSwanDifferente 3}
a_{G/H}^{\alpha} = \Ind_G^{G/H} a_G^{\alpha}, \quad a_{G/H}^{\beta} = \Ind_G^{G/H} a_G^{\beta},\quad {\rm and}\quad \sw_{G/H}^{\beta}= \Ind_G^{G/H}\sw_G^{\beta}.
\end{equation}
\end{lem}

\begin{proof}
$(i)$ ~ Since $\sw_G^{\beta} \lvert H=a_G^{\beta} \lvert H - (r_G - 1_G) \lvert H$ \eqref{ArtinSwan 2} and $r_G|H=[K':K]\cdot r_H$, \eqref{ArtinSwanDifferente 2} follows from \eqref{ArtinSwanDifferente 1}. We have $a_G^{\alpha}\lvert H(\sigma)=a_H^{\alpha} (\sigma)$ and $a_G^{\beta}\lvert H (\sigma)=a_H^{\beta}(\sigma)$, if $\sigma\in H-\{1\}$, since, in that case, $i_H(\sigma)=i_G(\sigma)$. Moreover, from \eqref{ArtinSwanExplicit 1}, we see that 
\begin{equation}
	\label{ArtinSwanDifferente 4}
\frac{\lvert H\lvert}{\lvert G\lvert} v(\mathfrak{d}_{V'/V})= \sum_{\sigma\in G-\{1\}} i_G(\sigma) - \sum_{\sigma\in H-\{1\}} i_G(\sigma).
\end{equation}
Hence, $a_G (1)=\frac{1}{\lvert G\lvert} v(\mathfrak{d}_{V'/V})\cdot r_H(1) + a_H (1)$,
which gives \eqref{ArtinSwanDifferente 1} by applying $\alpha$ and $\beta$.

$(ii)$~ From Lemma \ref{Tate}, we get
\begin{equation}
	\label{ArtinSwanDifferente 5}
i_{G/H}(\tau)=\sum_{\sigma\to \tau} i_G(\sigma) \quad {\rm and} \quad j_{G/H}(\tau)=\sum_{\sigma\to \tau} j_G(\sigma).
\end{equation}
These two equalities imply $(ii)$. 
\end{proof}

\begin{prop}
	\label{Virtual Characters}
The class function $\lvert G\lvert \sw_G^{\beta}$ is an element of $P_{\Z_{\ell}}(G)$.
\end{prop}

\begin{proof}
By \cite[16.2, Théorème 37]{Serre2}, the claim is equivalent to asking that $(1)$ $\lvert G\lvert \sw_G^{\beta} \in P_{\Lambda_{\Q}}(G)$, for some finite extension $\Lambda_{\Q}$ of $\Q_{\ell}$ in $\overline{\Q}_{\ell}$, and that $(2)$ $\sw_G^{\beta}(\sigma)=0$ for every $\sigma$ of order divisible by $\ell$.
To prove $(1)$, it is enough to show that $\lvert G\lvert \sw_G^{\beta} \in R_{\overline{\Q}_{\ell}}(G)$; hence, by Brauer induction, it is enough to prove that $\lvert G\lvert \langle \sw_G^{\beta}\lvert H,\chi\rangle \in \Z$ for every elementary subgroup $H$ of $G$ and every homomorphism $\chi : H\to \overline{\Q}_{\ell}^{\times}$ \cite[11.1, Théorème 22, Corollaire]{Serre2}. By \eqref{ArtinSwanDifferente 2}, it is enough to prove that $\langle \sw_G^{\beta},\chi\rangle \in \Z$ for every homomorphism $\chi: G\to \overline{\Q}_{\ell}^{\times}$. (Notice that, in the formula \eqref{ArtinSwanDifferente 2}, as $\mathfrak{d}_{V'/V}\in V$, $v^{\beta}(\mathfrak{d}_{V'/V})$ lies in $\Z$ (\ref{ArtinSwan})). Applying \eqref{ArtinSwanDifferente 3} to $H=\ker(\chi)$, we can assume that $\chi$ is injective, hence that $G$ is abelian. Now, by the Hasse-Arf theorem \ref{HasseArf}, $\sw_G (M_{\chi})$ is in $\Gamma_V$, where $M_{\chi}$ is the $\overline{\Q}_{\ell}[G]$-module defined by $\chi$; hence, $\beta(\sw_G (M_{\chi})) \in \Z$ is in $\lvert G\lvert \cdot\Z$. Since $\sw_G (M_{\chi})=\langle \langle\sw_G,\chi_M\rangle\rangle$ (\ref{Concordance}), we thus have \eqref{Pairing SerreKato 1} 
\begin{equation}
	\label{Virtual Characters 1}
\lvert G\lvert \langle \sw_G^{\beta},\chi\rangle=\beta(\sw_G (M_{\chi})) \in \Z.
\end{equation}
For $(2)$, let $\sigma$ be an element of $G$ of order divisible by $\ell$. Since $\ell$ is different from the characteristic of $\kappa(\m)$, $\sigma$ is in $G-G_{\varepsilon_L}$ (\ref{cyclic1} and \ref{cyclic2}). By \eqref{DefFilt} and \eqref{hypotheses 2}, this means that $j_G(\sigma)< \varepsilon_L$, i.e $i_G(\sigma)<2\varepsilon_L$, hence $i_G(\sigma)=\varepsilon_L$ by minimality of $\varepsilon_L$. So $j_G(\sigma)=0$ and $\sw_G^{\beta}(\sigma)=0$. Thus, the proposition is proved.
\end{proof}

\begin{rem}
	\label{AnnulationConducteurs}
If $\sigma\in G-G_{\varepsilon_L}$, then $i_G(\sigma)=\varepsilon_L$, hence $a_G^{\alpha}(\sigma)=0$ and $\sw_G^{\alpha}(\sigma)=0$.
\end{rem}

\section{Variation of conductors.}\label{Var}
\subsection{}\label{NotationsVariation}
Let $K$ be a complete discrete valuation field, $\mathcal{O}_K$ its valuation ring, $\m_K$ its maximal ideal, $k$ its residue field, assumed to be algebraically closed of characteristic $p>0$, and $\pi$ a uniformizer of $\O_K$. Let $D=\Sp(K\{\xi\})$ be the rigid unit disc over $K$, $X$ a smooth and connected $K$-affinoid curve endowed with a right action of a finite group $G$, and let $f : X\to D$ be a finite flat morphism such that $X/G\cong D$. The ring $\O(X)^G$ is a $K$-affinoid closed sub-algebra of $\O(X)$ and $\O(X)$ is finite over $\O(X)^G$ \cite[6.3.3/3]{BGR}. As $D$ and $X$ are affinoid, assuming that $X/G\cong D$ is equivalent to requiring $\O(D)\cong\O(X)^G$. Assume that $f$ is étale and Galois of group $G$ over an admissible open subset of $D$ containing $0$.

Recall that, in \ref{DiscPartial}, we defined the functions $\partial_f^{\alpha}: \Q_{\geq 0} \to \Q$ and $\partial_f^{\beta}: \Q_{\geq 0} \to \Z$ that measure the valuation of the discriminant of $f$. Recall also that in \ref{RelevementsPoints geometriques}, to a rational number $r\geq 0$, denoting by $X^{(r)}$ the inverse image by $f$ of the subdisc $D^{(r)}$ of $D$ of radius $\lvert \pi\lvert^r$, we associated a henselian $\Z^2$-valuation ring $V^h_r=V_{r, K'}^h$, with degree of imperfection $p$ for its residue field at its height $1$ prime ideal, and a set $S_f^{(r)}$ of couples $\tau=(\overline{x}_{\tau}, \p_{\tau})$, where $K'$ is a finite extension of $K$ which is $r$-admissible for $f$ (see \ref{radmissible}), $\overline{x}_{\tau}$ is a geometric point of the special fiber of the normalized integral model $\mathfrak{X}_{K'}^{(r)}$ of $X_{K'}^{(r)}$ (defined over $\O_{K'}$), and $\p_{\tau}$ is a height one prime ideal of $\O_{\mathfrak{X}_{K'}^{(r)}, \overline{x}_{\tau}}$. To a couple $\tau \in S_f^{(r)}$, we associated a henselian $\Z^2$-valuation ring $V_r^h(\tau)$ which is a monogenic integral extension of $V^h_r$ (\ref{MonogenicIntegral}), \eqref{V^h MonogenicLibre}.

\subsection{}\label{ActionCouples}
The group $G$ acts also on $X^{(r)}$ and we have $X^{(r)}/G \cong D^{(r)}$. This induces actions of $G$ on $\mathfrak{X}_{K'}^{(r)}$ and $\mathfrak{X}_{s'}^{(r)}$. We thus obtain an action of $G$ on $S^{(r)}$ as follows. For $g\in G$ and $\tau=(\overline{x}_{\tau}, \p_{\tau})$, we put $g\cdot \tau=(g\circ\overline{x}_{\tau}, g_{\#}^{-1}(\p_{\tau}))$, where $g_{\#}$ is the induced  isomorphism
\begin{equation}
	\label{ActionCouples 1}
\O_{\mathfrak{X}_{K'}^{(r)}, g\circ \overline{x}_{\tau}}\xrightarrow{\sim} \O_{\mathfrak{X}_{K'}^{(r)}, \overline{x}_{\tau}}.
\end{equation}

\begin{lem} \label{Action Transitive}
The above action of $G$ on the set $S_f^{(r)}$ is transitive.
\end{lem}

\begin{proof}
The group $G$ acts transitively on the finite set of geometric connected components of $\mathfrak{X}_{K'}^{(r)}$. Hence, enlarging $K'$, we can assume that the connected components of $\mathfrak{X}_{K'}^{(r)}$ are geometrically connected and that $G$ acts transitively on the set formed by these connected components. Let $\mathfrak{X}_c^{(r)}$ be a connected component of $\mathfrak{X}_{K'}^{(r)}$ and let $G_c$ be its stabilizer. To prove the lemma, its enough to show that $G_c$ acts transitively on the subset of $S_f^{(r)}$ formed by couples $(\overline{x}_{\tau}, \p_{\tau})$ such that the image of $\overline{x}_{\tau}$ in $\mathfrak{X}_{s'}^{(r)}$ lies in $\mathfrak{X}_{c, s'}^{(r)}$. Therefore, we may assume that $\mathfrak{X}_{K'}^{(r)}$ is geometrically connected; hence $X_{K'}^{(r)}$ and $\mathfrak{X}_{s'}^{(r)}$ are also geometrically connected \cite[4.4]{A.S.1}.
As $r$ and $K'$ are fixed, we write $A, \mathcal{A}, B$ and $\mathcal{B}$ for $\O(D_{K'}^{(r)}), \O^{\circ}(D_{K'}^{(r)}), \O(X_{K'}^{(r)})$ and $\O^{\circ}(X_{K'}^{(r)})$ respectively.

First, we have the identity $A=B^{G}$ and a commutative diagram
\begin{equation}
	\label{Action Transitive 1}
\xymatrix{
\mathcal{B}\ar[r] & B \\
\mathcal{A}\ar[u]\ar[r] & A\ar[u],
}
\end{equation}
where the horizontal arrows are inclusions.
Since $\mathcal{A}$ is a normal integral domain, hence integrally closed, and $\mathcal{B}$ is finite over $\mathcal{A}$ (\ref{ModelsFormels}), we find that we have also $\mathcal{A}=\mathcal{B}^{G}$.
As $\Spec(\mathcal{A})$ and $\Spec(\mathcal{B})$ are both noetherian and connected, it follows from \cite[3.2.8]{Fu} that $G$ acts transitively on the set $\Spec(\mathcal{B})(\overline{x})$ of geometric points of $\Spec(\mathcal{B})$ above $\overline{x}$, where $\overline{x}$ (notation from \ref{ValuationDisqueUnité}) is regarded as a geometric point of $\Spec(\mathcal{A})$ through
\begin{equation}
	\label{Action Transitive 2}
\overline{x}\to \mathfrak{D}_{s'}^{(r)}=\Spec(\mathcal{A}/(\pi))\to \Spec(\mathcal{A}).
\end{equation}
(Note that, even though $\Spec(\mathcal{B})\to \Spec(\mathcal{A})$ may not be étale, we can still apply \cite[3.2.8]{Fu} since we only use the implication "(i) $\Rightarrow$ (ii)" therein.) We deduce that $G$ acts transitively too on the set $\mathfrak{X}_{s'}^{(r)}(\overline{x})=\mathfrak{X}_{K'}^{(r)}(\overline{x})$ of geometric points of $\mathfrak{X}_{K'}^{(r)}=\Spf(\mathcal{B})$ above $\overline{x}$.

Second, let $\overline{x}'$ be one such geometric point. Then, its stabilizer ${\rm St}_{\overline{x}'}$ in $G$ acts on the noetherian local ring $\O_{\overline{x}'}=\O_{\mathfrak{X}_{K'}^{(r)}, \overline{x}'}$. As $\mathcal{B}^{G}=\mathcal{A}$ and \eqref{Produit Locaux Affine 1}
\begin{equation}
	\label{Action Transitive 3}
\O_{\overline{x}}\otimes_{\mathcal{A}}\mathcal{B}\cong \prod_{\overline{x}'\in \Spf(\mathcal{B})(\overline{x})} \O_{\overline{x}'},
\end{equation} 
where $\O_{\overline{x}}$ is the noetherian local ring $\O_{\mathfrak{D}_{K'}^{(r)}, \overline{x}}$, we have $(\O_{\overline{x}'})^{{\rm St}_{\overline{x}'}}=\O_{\overline{x}}$. Then, by \cite[3.1.1 (ii)]{Fu}, ${\rm St}_{\overline{x}'}$ acts transitively on the set of prime ideals of $\O_{\overline{x}'}$ above the height $1$ prime ideal $\p$ of $\O_{\overline{x}}$ (notation of \ref{ValuationDisqueUnité}).

Now, let $\tau=(\overline{x}_{\tau}, \p_{\tau})$ and $\tau'=(\overline{x}_{\tau'}, \p_{\tau'})$ be two elements of $S_f^{(r)}$. Then, by the first part of the proof, there exists $g\in G$ such that $g\cdot \overline{x}_{\tau}=\overline{x}_{\tau'}$. Thus, we have a homomorphism $g_{\#}: \O_{\overline{x}_{\tau'}}\to \O_{\overline{x}_{\tau}}$ and $\p_1=g_{\#}^{-1}(\p_{\tau})$ is a height $1$ prime ideal of $\O_{\overline{x}_{\tau'}}$ above $\p$. Hence, by the second part above, there exists $g'\in {\rm St}_{\overline{x}_{\tau}}$ such that $g'\cdot \p_1= \p_{\tau'}$, i.e. $(g'_{\#})^{-1}(\p_1)=\p_{\tau'}$, where $g'_{\#}$ is the homomorphism $\O_{\overline{x}_{\tau'}}\to \O_{\overline{x}_{\tau'}}$ induced by $g'$. Then, the composition $g'g$ satisfies $g'g\cdot \overline{x}_{\tau}=\overline{x}_{\tau'}$ and $g'g\cdot \p_{\tau}=\p_{\tau'}$, i.e. $g'g\cdot \tau=\tau'$.
\end{proof}

\begin{lem}
	\label{Extensions Galoisiennes}
For every $\tau\in S_f^{(r)}$, the extension of fields of fractions $\K_{r, \tau}^h/\K_r^h$ induced by $V_r^h(\tau)/V^h_r$ is finite and Galois of group $G_{r, \tau}$ isomorphic to the stabilizer ${\rm St}_{r, \tau}$ of $\tau$ under the action of $G$ on $S_f^{(r)}$.
\end{lem}

\begin{proof}
Let $U$ be an open subset of $D$ containing $0$ such that $f^{-1}(U)\to U$ is étale and Galois of group $G$. Let $\{\overline{x}_1, \ldots, \overline{x}_N\}$ be the geometric points of $\mathfrak{X}_{K'}^{(r)}$ above $\overline{x}$ and denote by $\O_{\overline{x}}$ and $\O_{\overline{x}_j}$ the local rings $\O_{\mathfrak{D}_{K'}^{(r)}, \overline{x}}$ and $\O_{\mathfrak{X}_{K'}^{(r)}, \overline{x}_j}$ respectively. Denote also by $A, \mathcal{A}, B$ and $\mathcal{B}$ the rings $\O(D_{K'}^{(r)}), \O^{\circ}(D_{K'}^{(r)}), \O(X_{K'}^{(r)})$ and $\O^{\circ}(X_{K'}^{(r)})$ respectively. Then, by \eqref{Produit Locaux Affine 1}, we have
\begin{equation}
	\label{Extensions Galoisiennes 1}
\O_{\overline{x}}\otimes_{\mathcal{A}}\mathcal{B}\cong \prod_{j=1}^N \O_{\overline{x}_j}.
\end{equation}
Since $\K_r^h$ is the field of fractions of $(\O_{\overline{x}})_{\p}^{\rm sh}$, where $\p$ is a height $1$ prime ideal of $\O_{\overline{x}}$, it follows from \cite[17.17]{Endler} that, for each $j=1, \ldots, N$, we have
\begin{equation}
	\label{Extensions Galoisiennes 2}
\K_r^h\otimes _{\O_{\overline{x}}}\O_{\overline{x}_j}\xrightarrow{\sim} \prod_{\q}{\rm Frac}((\O_{\overline{x}_j})_{\q}^{\rm sh}),
\end{equation}
where $\q$ runs over the height $1$ prime ideals of $\O_{\overline{x}_j}$ above $\p$. Tensoring \eqref{Extensions Galoisiennes 1} with $\K_r^h$ and taking into account \eqref{(V) Independant y} yield
\begin{equation}
	\label{Extensions Galoisiennes 3}
\K_r^h\otimes_A B \xrightarrow{\sim} \prod_{\tau\in S^{(r)}} \K_{r, \tau}^h.
\end{equation}
As $(\K_r^h\otimes_A B)^G=\K_r^h$, we are thus reduced to showing that $\K_r^h\otimes_A B$ is finite étale over $\K_r^h$. As the morphism $f^{-1}(U^{(r)})\to U^{(r)}=U\cap D_{K'}^{(r)}$ induced by $f$ is finite étale and Galois of group $G$, it is enough to show that $\K_r^h$ defines a point in $U^{(r)}$, i.e. that $\K_r^h\otimes_A \O(U^{(r)})\neq 0$.
Now the admissible open immersion $U^{(r)}\hookrightarrow D_{K'}^{(r)}$ lifts to a formal morphism $\mathcal{U}^{(r)}\hookrightarrow \mathfrak{D}_{K'}^{(r)'}\to \mathfrak{D}_{K'}^{(r)}$, where $\mathfrak{D}_{K'}^{(r)'}\to \mathfrak{D}_{K'}^{(r)}$ is an admissible formal blow-up and $\mathcal{U}^{(r)}\hookrightarrow \mathfrak{D}_{K'}^{(r)'}$ is a formal open immersion. The origin $x$ of $D_{K'}^{(r)}$ defines a rig-point $x: \Spf(\O_{K'})\to \mathfrak{D}_{K'}^{(r)}$. As the pull-back along this morphism of the coherent open ideal on $\mathfrak{D}_{K'}^{(r)}$ defining $\mathfrak{D}_{K'}^{(r)'}$ is non-zero, hence invertible in $\O_{K'}$, this rig-point lifts to a rig-point $x':\Spf(\O_{K'})\to \mathfrak{D}_{K'}^{(r)'}$. As $x\in U^{(r)}$, this rig-point lies in $\mathcal{U}^{(r)}$ and also defines a geometric point $\overline{x}'\to \mathcal{U}$ above $\overline{x}$. Let $\mathcal{V}^{(r)}$ be a formal affine open subset of $\mathcal{U}^{(r)}$ containing $x'$. Then, $\O_{\mathcal{V}, \overline{x}'}=\O_{\mathfrak{D}_{K'}^{(r)'}, \overline{x}'}$ and we have a commutative square
\begin{equation}
	\label{Extensions Galoisiennes 4}
\xymatrix{
A \ar[r] \ar[d] & \O(U^{(r)}) \ar[r] & \O(\mathcal{V}^{(r)})\otimes_{\O_{K'}} K' \ar[d]\\
\K_r \ar[rr] & & \K'_r, 
}
\end{equation}
where $\K'_r$ is the field of fractions of $\O_{\mathcal{V}^{(r)}, \overline{x}'}$, the vertical arrows are induced by the inclusions $\mathcal{A}\subset \O_{\overline{x}}$ and $\O(\mathcal{V}^{(r)})\subset \O_{\mathcal{V}^{(r)}, \overline{x}'}$ and the bottom horizontal arrow is induced by $\O_{\overline{x}}\to \O_{\mathcal{V}^{(r)}, \overline{x}'}$. Then, $\K_r\otimes_A \O(U^{(r)})\neq 0$; hence $\K_r^h\otimes_A \O(U^{(r)})\neq 0$.
\end{proof}

\subsection{}\label{Characters} 
From \ref{Extensions Galoisiennes} and the independence of $S_f^{(r)}$ (and $G$) from the choice of a $K'$ which is $r$-admissible for $f$ (\ref{RelevementsPoints geometriques}), we see that the subgroups $G_{r, \tau} \subset G$ are independent of the choice of such a $K'$. Moreover, they are conjugate when $\tau$ varies over $S_f^{(r)}$. More precisely, if $\tau, \tau'\in S_f^{(r)}$ and $\tau'=g\cdot\tau$, for some $g\in G$, then $gG_{r, \tau}g^{-1}=G_{r, \tau'}$. 

For $\tau\in S_f^{(r)}$, as the extension $V_r^h(\tau)/ V_r^h$ satisfies the hypotheses in \ref{hypotheses}, we have a $\Q$-valued class function $a^{\alpha}_{G_{r, \tau}}$ and a $\Z$-valued character $\sw^{\beta}_{G_{r, \tau}}$ on $G_{r, \tau}$ (\ref{ArtinSwan}). We note here that, although they are defined using the fixed uniformizer $\pi$ of $\O_K$ \eqref{NotationsVariation}, these functions don't depend on $\pi$ (see \ref{ValuationDisqueUnité}).
These functions are also conjugate in the sense that $a^{\alpha}_{G_{r, \sigma^{-1}\tau}}(\sigma^{-1}g\sigma)=a^{\alpha}_{G_{r, \tau}}(g)$ and $\sw^{\beta}_{G_{r, \sigma^{-1}\tau}}(\sigma^{-1}g\sigma)=\sw^{\beta}_{G_{r, \tau}}(g)$, for $g\in G_{r, \tau}$.
Thus, from the definition of induction, we see that the following class functions on $G$
\begin{equation}
	\label{Characters 1}
a_{f, K'}^{\alpha}(r)=\Ind_{G_{r, \tau}}^G a_{G_{r, \tau}}^{\alpha} \quad {\rm and}\quad \sw_{f, K'}^{\beta}(r)=\Ind_{G_{r, \tau}}^G \sw^{\beta}_{G_{r, \tau}}
\end{equation}
are independent of $\tau\in S_f^{(r)}$ used to make the induction.

\begin{lem}
	\label{Independence of Base Change}
Let $L$ be a finite extension of $K'$. Then, we have
\begin{equation}
	\label{Independence of Base Change 1}
a_{f, L}^{\alpha}(r)=a_{f, K'}^{\alpha}(r) \quad {\rm and}\quad \sw_{f, L}^{\beta}(r)=\sw_{f, K'}^{\beta}(r).
\end{equation}
In particular, $a_{f, K'}^{\alpha}(r)$ and $\sw_{f, K'}^{\beta}(r)$ are independent of the choice of the extension $K'$ of $K$ which is $r$-admissible for $f$; we denote them by $a_f^{\alpha}(r)$ and $\sw_f^{\beta}(r)$ respectively.
\end{lem}

\begin{proof}
The extension $L/K$ is also $r$-admissible for $f$ (\ref{radmissible}). Let $ \mathfrak{D}_L^{(r)}$ and $\mathfrak{X}_L^{(r)}$ be the normalized integral models, defined over $\O_{L}$, of $D_K^{(r)}$ and $X_K^{(r)}$ respectively. We have a Cartesian diagram
\begin{equation}
	\label{Independence of Base Change 2}
\xymatrix{\ar @{} [dr] | {\Box} 
\mathfrak{X}_L^{(r)} \ar[r]^{g^{(r)}} \ar[d]_{\widehat{f}_L^{(r)}} & \mathfrak{X}_{K'}^{(r)} \ar[d]^{\widehat{f}_{K'}^{(r)}}\\
 \mathfrak{D}_L^{(r)} \ar[r] & \mathfrak{D}_{K'}^{(r)}.
}
\end{equation}
We denote by $\overline{x}'$ the unique geometric point of $\mathfrak{D}_L^{(r)}$ (the origin of the special fiber) above the geometric point $\overline{x}$ of $\mathfrak{D}_{K'}^{(r)}$ and $\p'$ the height $1$ prime ideal of $\O_{D_L^{(r)}, \overline{x}}$ above $\p$ (\ref{RelevementsPoints geometriques}). As $\mathfrak{X}_{K'}^{(t)}$ and $\mathfrak{X}_{L}^{(t)}$ have isomorphic special fibers, we have a canonical bijection $S_{f, L}^{(r)}\xrightarrow{\sim} S_{f, K'}=S_f^{(r)}$, (\ref{RelevementsPoints geometriques}). With this identification, let $\tau'=(\overline{x}_{\tau'}, \p_{\tau'})$ be an element of $S_f^{(r)}$ and put $\tau=(\overline{x}_{\tau}, \p_{\tau})=(g^{(r)}(\overline{x}_{\tau'}), g_{\#}^{(r)^{-1}}(\p_{\tau'}))$. Then, the canonical homomorphism ${\rm St}_{r, \tau'}\to {\rm St}_{r, \tau}$ is an isomorphism and thus induces an isomorphism $G_{r, \tau'}\xrightarrow{\sim} G_{r, \tau}$ (\ref{Extensions Galoisiennes}).
By functoriality (\ref{(V) Fonctoriel}), $\mathfrak{D}_L^{(r)}\to \mathfrak{D}_{K'}^{(r)}$ and $\mathfrak{X}_L^{(r)}\to \mathfrak{X}_{K'}^{(r)}$ induce extensions of $\Z^2$-valuation rings $(V_{r, K'}^h, v_{r, K'})\to (V_{r, L}^h, v_{r, L})$ and $(V_r^h(\tau), v_{r, \tau})\to (V_r^h(\tau'), v_{r, \tau'})$. We have a normalized valuation map $(\K_{r, K'}^h)^{\times}\to \Gamma_{V_{r, K'}^h}\xrightarrow{\sim}\Z\times\Z\hookrightarrow \Q\times\Z$, with the last inclusion given by $(a, b)\mapsto (a/e_{K'/K}, b)$, where $e_{K'/K}$ is the ramification index of $K'/K$, and induced projections $v_{r, K'}^{\alpha}: (\K_{r, K'}^h)^{\times} \to \Q$ and  $v_{r, K'}^{\beta}: (\K_{r, K'}^h)^{\times} \to \Z$ (\ref{(Vr)}). We have a similar normalized valuation map and projections for $L/K$ too. Then, the composition $(\K_{r, K'}^h)^{\times}\to (\K_{r, L}^h)^{\times}\xrightarrow{v_{r, L}^{\beta}}\Z$ coincides with $v_{r, K'}^{\beta}$. Indeed, as the special fibers of $\mathfrak{X}_{K'}^{(r)}$ and $\mathfrak{X}_L^{(r)}$ are canonically isomorphic, we have $A_{\overline{x}_{\tau}}/\p_{\tau}\xrightarrow{\sim} A_{\overline{x}_{\tau'}}/\p_{\tau'}$ and thus $V_r^h(\tau)/\p_{\tau}\xrightarrow{\sim} V_r^h(\tau')/\p_{\tau'}$, whence the claim. As
$a e_{L/K'}/e_{L/K}=a/e_{K'/K}$, it follows from the above normalizations that the composition 
$(\K_{r, K'}^h)^{\times}\to (\K_{r, L}^h)^{\times}\xrightarrow{v_{r, L}^{\alpha}}\Q$ coincides with $v_{r, K'}^{\alpha}$.
We can write $V_r^h(\tau)=V^h_{r, K}[a]$ for some $a \in V_r^h(\tau)$ whose reduction mod $\p_{\tau}$ is a uniformizer of the totally ramified (discretely valued) extension $\kappa(\p_{\tau})$ of $\kappa(\p)$ (\ref{MonogenicIntegral} and \ref{V^h MonogenicLibre}). Then, viewed as an element of $V_r^h(\tau')$, $a$ also satisfies $V_r^h(\tau')=V^h_{r, L}[a]$, as seen from the isomorphism $V_r^h(\tau)/\p_{\tau}\xrightarrow{\sim} V_r^h(\tau')/\p_{\tau'}$.
Then, for $\sigma'\in G_{r, \tau'}-\{1\}$, we have (\ref{Notations Ramif})
\begin{equation}
	\label{Independence of Base Change 3}
i_{G_{r, \tau}}(\varphi(\sigma'))=v_{r, \tau}(\varphi(\sigma')(a)-a) \quad {\rm and}\quad j_{G_{r, \tau}}(\varphi(\sigma'))=i_{G_{r, \tau}}(\varphi(\sigma'))-\frac{\varepsilon}{\lvert G_{r, \tau'}\lvert},
\end{equation}
where $v_{r, \tau}$ denotes the valuation map of $V_r^h(\tau)$ and $\varepsilon$ is the minimum positive element of $\Gamma_{V^h_{r, K}}$, which is also the minimum positive element of $\Gamma_{V^h_{r, L}}$. As $v_{r, \tau'}(\sigma'(a)-a)=v_{r, \tau}(\varphi(\sigma')(a)-a)$ and $\lvert G_{r, \tau}\lvert =\lvert G_{r, \tau'}\lvert$, we deduce that 
\begin{equation}
	\label{Independence of Base Change 4}
i_{G_{r, \tau}}(\varphi(\sigma'))=i_{G_{r, \tau'}}(\sigma') \quad {\rm and}\quad j_{G_{r, \tau}}(\varphi(\sigma'))= j_{G_{r, \tau'}}(\sigma'),
\end{equation}
which proves the lemma.
\end{proof}

\begin{rem}
	\label{ResidueGaloisGroupBaseChange}
We keep the notation of \ref{Independence of Base Change} and of its proof. We have also shown in the above proof that we have canonical isomorphisms
\begin{equation}
	\label{ResidueGaloisGroupBaseChange 1}
V_{r, K'}^h/\p\xrightarrow{\sim} V_{r, L}^h/\p' \quad {\rm and}\quad V_r^h(\tau)/\p_{\tau}\xrightarrow{\sim} V_r^h(\tau')/\p_{\tau'}.
\end{equation}
Therefore, the residue Galois group of the extension $(V_r^h(\tau))_{\p_{\tau}}/(V_r^h)_{\p}$ is independent of the choice of $K'$ which is $r$-admissible for $f$. Hence, the corresponding inertia subgroup $I_{r, \tau}$ of $G_{r, \tau}$ and wild inertia subgroup $P_{r, \tau}$ of $I_{r, \tau}$ are also independent of the choice of such a $K'$.
\end{rem}

\begin{lem}
	\label{Change of Cover}
Let $X'$ be a smooth $K$-rigid space and let $g: X'\to X$ be a morphism such that the composition $X'\xrightarrow{g} X \xrightarrow{f} D$, which we denote by $f'$, is a finite flat morphism which is étale and Galois of group $G'$ over an admissible open subset of $D$ containing $0$. Then, $G$ is a quotient of $G'$ and we have
\begin{equation}
	\label{Change of Cover 1}
a_f^{\alpha}(r)=\Ind_{G'}^G a_{f'}^{\alpha}(r) \quad {\rm and} \quad \sw_f^{\beta} (r)=\Ind_{G'}^G \sw_{f'}^{\beta} (r).
\end{equation}
\end{lem}

\begin{proof}
By functoriality (\ref{Functoriality of Models}), we have a map $g^{(r)}: S_{f'}^{(r)}\to S_f^{(r)}$.
Let $\tau'$ be an element of $S_g^{(r)}$ and $\tau\in S_f^{(r)}$ its image by $g^{(r)}$. Functoriality again (\ref{(V) Fonctoriel}) induces extensions of $\Z^2$-valuation rings $V_r^h\to V_r^h(\tau)\to V_r^h(\tau')$ and thus a commutative diagram, with inclusion horizontal arrows,
\begin{equation}
	\label{Change of Cover 2}
\xymatrix{
 G'_{r, \tau'} \ar[r] \ar @{>>}[d] & G' \ar @{>>}[d]\\
 G_{r, \tau} \ar[r] & G.
}
\end{equation}
The surjectivity of the left vertical homomorphism in \eqref{Change of Cover 2} stems from functoriality (\ref{Extensions Galoisiennes}).
Since $a_{f'}^{\alpha}(r)=\Ind_{G'_{r, \tau'}}^{G'} a^{\alpha}_{G_{r, \tau}}$, from \eqref{Change of Cover 2} and the transitivity of induction, we get
\begin{equation}
	\label{Change of Cover 3}
\Ind_{G'}^G a_{f'}^{\alpha}(r)=\Ind_{G_{r, \tau}}^G \left(\Ind_{G'_{r, \tau'}}^{G_{r, \tau}} a^{\alpha}_{G'_{r, \tau'}}\right)=a_f^{\alpha}(r),
\end{equation}
where the last equality stems from \eqref{ArtinSwanDifferente 3} and \eqref{Characters 1}. The same reasoning applies to $\sw_f^{\beta}(r)$ and $\sw_{f'}^{\beta}(r)$.
\end{proof}

\subsection{}\label{Normalized Conductors}
We normalize the functions $a_f^{\alpha}(r)$ and $\sw_f^{\beta}(r)$ in the following way. We put
\begin{equation}
	\label{Normalized Conductors 1}
\widetilde{a}_f^{\alpha}(r)=\frac{\lvert G\lvert}{\lvert S_f^{(r)}\lvert} a_f^{\alpha}(r) \quad {\rm and}\quad \widetilde{\sw}_f^{\beta}(r)=\frac{\lvert G\lvert}{\lvert S_f^{(r)}\lvert} \sw_f^{\beta}(r).
\end{equation}
As $\lvert S_f^{(r)}\lvert$ has already been seen to be independent of the choice of the $r$-admissible extension of $K$ defining the normalized integral models (\ref{RelevementsPoints geometriques}), $\widetilde{a}_f^{\alpha}(r)$ and $\widetilde{\sw}_f^{\beta}(r)$ are also independent of that choice (\ref{Independence of Base Change}).

\subsection{}\label{Revetement quotient}
Let $H$ be a subgroup of $G$. Then, the quotient $Y=X/H=\Sp(\O(X)^H)$ is a smooth $K$-rigid space and the morphism $f_H : X/H \to D$ induced by $f$ is finite, flat and étale over an open subset of $D$ containing $0$.
Moreover, for $t\in \Q_{\geq 0}$, the quotient morphism $g: X\to Y$ induces a map $g_S^{(t)}:  S_f^{(t)}\to S_{f_H}^{(t)}$.

\begin{lem}
	\label{Surjective}
The map $g_S^{(t)}: S_f^{(t)}\to S_{f_H}^{(t)}$ is surjective and its fibers are the orbits of $H$.
\end{lem}

\begin{proof}
As $t$ is fixed, we fix also a finite extension $K'$ of $K$ which is $t$-admissible for $f$ and $f_H$. We denote $\O(Y_{K'}^{(t)})$, $\O^{\circ}(Y_{K'}^{(t)})$, $\O(X_{K'}^{(t)})$ and $\O^{\circ}(X_{K'}^{(t)})$ by $A, \mathcal{A}, B$ and $\mathcal{B}$ respectively. Let $\overline{y}$ be a geometric point of $\Spf(\mathcal{A})$ and let $\overline{x}$ be a geometric point of $\Spf(\mathcal{B})$ above $\overline{y}$. We denote by $\O_{\overline{y}}$ and $\O_{\overline{x}}$ the respective formal étale local rings (\ref{ALFormelEtale 1}). Then, as in the proof of \ref{Action Transitive}, we have $\mathcal{A}=\mathcal{B}^H$. Combined with \eqref{Produit Locaux Affine 1}, this yields 
\begin{equation}
	\label{Surjective 1}
\O_{\overline{y}}\cong (\prod_{h\in H} \O_{h\cdot\overline{x}})^H.
\end{equation}
By \cite[3.1.1 (ii)]{Fu}, the canonical morphism $\coprod_{h\in H}\Spec(\O_{h\cdot\overline{x}})\to \Spec(\O_{\overline{y}})$ is then surjective and its fibers are the orbits of $H$. Hence, $g_S^{(t)}$ is surjective, and if $\tau=(\overline{x}, \q)$ and $\tau'=(h\cdot\overline{x}, \q')$, for some $h\in H$, are elements of $S_f^{(t)}$, they have the same image by $g_S^{(t)}$ if and only $h_{\#}^{-1}(\q')=\q$ (notation of \ref{ActionCouples}), i.e. if and only if $h\cdot\tau=\tau'$
\end{proof}

\begin{prop} \label{DiscVar}
We use the notation of {\rm \ref{DérivéeLutke}} and {\rm \ref{Revetement quotient}}. We assume that $X/H$ has \emph{trivial canonical sheaf}. Then,
for any $t\in \Q_{\geq 0}$, we have the identity
\begin{equation}
	\label{DiscVar 1}
\partial_{f_H}^{\alpha}(t)=\langle \widetilde{a}_f^{\alpha}(t), \Q[G/H]\rangle,
\end{equation}
where the representation $\Q[G/H]=\Ind_H^G 1_H$ of $G$ stands in abusively for its character.
The right derivative of $\partial_{f_H}^{\alpha}$ at $t$ is 
\begin{equation}
	\label{DiscVar 2}
\frac{d}{dt}\partial_{f_H}^{\alpha}(t+)=\langle \widetilde{\sw}_f^{\beta}(t), \Q[G/H]\rangle.
\end{equation}
\end{prop}

\begin{proof}
Let $\tau$ be an element of $S_f^{(t)}$.We first note that, as $G$ acts transitively on $S_f^{(t)}$, we have $\lvert G_{t, \tau}\lvert =\lvert G\lvert /\lvert S_f^{(t)}\lvert$. By Frobenius reciprocity, we have the following identities
\begin{equation}
	\label{DiscVar 3}
\langle a_f^{\alpha}(t), \Q[G/H]\rangle=\langle a_{G_{t, \tau}}^{\alpha}, \Q[G/H]\lvert G_{t, \tau} \rangle,
\end{equation}
\begin{equation}
	\label{DiscVar 4}
\langle \sw_f^{\beta}(t), \Q[G/H]\rangle=\langle \sw_{G_{t, \tau}}^{\beta}, \Q[G/H]\lvert G_{t, \tau} \rangle.
\end{equation}
Let $R$ be a set of representatives in $G$ of the double cosets $G_{t,\tau}\backslash G/H$. From \cite[\S 7.3, Prop. 22]{Serre2}, for $\tau \in S_f^{(t)}$, we have the identity
\begin{equation}
	\label{DiscVar 5}
\Q[G/H]\lvert G_{t, \tau}= \bigoplus_{\sigma \in R} \Ind_{H_{\sigma}}^{G_{t, \tau}} 1_{H_{\sigma, \tau}},
\end{equation}
where $H_{\sigma, \tau}=\sigma H\sigma^{-1}\cap G_{t, \tau}$. If $\sigma\in R$ and $g\sigma h$ is another representative of the double coset $G_{t, \tau}\sigma H$, then $H_{g\sigma h, \tau}=H_{\sigma, \tau}$. Hence, $H_{\sigma, \tau}$ depends only on the double coset $G_{t, \tau}\sigma H$ and the sum \eqref{DiscVar 5} is taken over $G_{t,\tau}\backslash G/H$. Therefore, we have
\begin{equation}
	\label{DiscVar 6}
\langle a_f^{\alpha}(t), \Q[G/H]\rangle=\sum_{\sigma\in G_{t,\tau}\backslash G/H} \langle a_{G_{t,\tau}}^{\alpha}, \Q[G_{t,\tau}/H_{\sigma, \tau}]\rangle,
\end{equation}
\begin{equation}
	\label{DiscVar 7}
\langle \sw_f^{\beta}(t), \Q[G/H]\rangle=\sum_{\sigma\in G_{t,\tau}\backslash G/H} \langle \sw_{G_{t,\tau}}^{\beta}, \Q[G_{t,\tau}/H_{\sigma, \tau}]\rangle.
\end{equation}
From \eqref{ArtinSwanExplicit 4} and \eqref{ArtinSwanExplicit 5}, we get
\begin{equation}
	\label{DiscVar 8}
\lvert G_{t, \tau}\lvert \langle a_{G_{t, \tau}}^{\alpha}, \Q[G_{t,\tau}/H_{\sigma, \tau}]\rangle =v_t^{\alpha}(\mathfrak{d}_{V_t^h(\sigma, \tau)/V_t^h}),
\end{equation}
\begin{equation}
	\label{DiscVar 9}
\lvert G_{t, \tau}\lvert \langle \sw_{G_{t, \tau}}^{\beta}, \Q[G_{t,\tau}/H_{\sigma, \tau}]\rangle =v_t^{\beta}(\mathfrak{d}_{V_t^h(\sigma, \tau)/V_t^h}) -\frac{\lvert G_{t, \tau}\lvert}{\lvert H_{\sigma, \tau}\lvert}+ 1,
\end{equation}
where $V_t^h(\sigma, \tau)=V_t^h(\tau)^{H_{\sigma, \tau}}$. The subgroup $\sigma^{-1} H_{\sigma, \tau}\sigma$ of $\sigma^{-1} G_{t, \tau}\sigma=G_{t, \sigma^{-1}\cdot\tau}$ (\ref{Characters}) is $H_{{\rm id}, \sigma^{-1}\cdot\tau}$. Then, $\sigma^{-1}$ yields an isomorphism of $\Z^2$-valuation rings $V_t^h(\sigma^{-1}\cdot\tau)\xrightarrow{\sigma} V_t^h(\tau)$, via \eqref{ActionCouples 1}, which induces an isomorphism
\begin{equation}
	\label{DiscVar 10}
V_t^h({\rm id}, \sigma^{-1}\cdot\tau)=(V_t^h(\sigma^{-1}\cdot\tau))^{\sigma^{-1}H_{\sigma, \tau}\sigma} \xrightarrow{\sim} V_t^h(\sigma, \tau).
\end{equation}
By \ref{Surjective}, the map $C: G_{t,\tau}\backslash G/H \to S_{f_H}^{(t)}, ~  G_{t, \tau}\sigma H\mapsto g_S^{(t)}(\sigma^{-1}\cdot \tau)$ is well-defined since, for $g\in G_{t, \tau}$ and $h\in H$, $(g\sigma h)^{-1}\cdot \tau=h^{-1}\cdot \tau$. Moreover, as $G$ acts transitively on $S_f^{(t)}$ (\ref{Action Transitive}) and $g_S^{(t)}$ is surjective, $C$ is also surjective. If $C( G_{t, \tau}\sigma H)=C( G_{t, \tau}\sigma' H)$, then there exists $h\in H$ such that $\sigma^{-1}\cdot \tau=h\sigma^{'-1}\cdot \tau$; so $\sigma h \sigma^{'-1} \in G_{t, \tau}$ and thus $\sigma'\in  G_{t, \tau}\sigma H$. Hence, $C$ is also injective, hence a bijection.
We also have $V_t^h({\rm id}, \sigma^{-1}\cdot\tau)=V_t^h(g_S^{(t)}(\sigma^{-1}\cdot\tau))$. It follows that, if $\sigma, \sigma' \in R$ represent double cosets such that $g_S^{(t)}(\sigma^{-1}\cdot\tau)=g_S^{(t)}(\sigma^{'-1}\cdot\tau)$, then $V_t^h({\rm id}, \sigma^{-1}\cdot\tau)=V_t^h({\rm id}, \sigma^{'-1}\cdot\tau)$.

Combining all this with \eqref{DiscVar 6}, \eqref{DiscVar 7}, \eqref{DiscVar 8} and \eqref{DiscVar 9} yields
\begin{equation}
	\label{DiscVar 11}
\langle \widetilde{a}_f^{\alpha}(t), \Q[G/H]\rangle =\sum_{\tau'\in S_{f_H}^{(t)}} v_t^{\alpha}(\mathfrak{d}_{V_t^h(\tau')/V_t^h})=v_t^{\alpha}\left(\prod_{\tau'\in S_{f_H}{(t)}} \mathfrak{d}_{V_t^h (\tau')/ V_t^h}\right)=\partial_{f_H}^{\alpha}(t),
\end{equation}
\begin{equation}
	\label{DiscVar 12}
\langle \widetilde{\sw}_f^{\beta}(t), \Q[G/H]\rangle=v_t^{\beta} \left(\prod_{\tau'\in S_{f_H}{(t)}}\mathfrak{d}_{V_t^h (\tau')/ V_t^h}\right) - \deg(f_H) + \lvert S_{f_H}^{(t)}\lvert.
\end{equation}
From the defining formula \eqref{dfs 1} (for $f_H$), we thus see that $\langle \widetilde{\sw}_f^{\beta}(t), \Q[G/H]\rangle=d_{f_H, s}(t)- \deg(f_H)+ \lvert S_{f_H}^{(t)}\lvert$. Now, by Proposition  \ref{DiscEgalite}, applied to $f_H$, we deduce that
\begin{equation}
	\label{DiscVar 13}
\langle \widetilde{\sw}_f^{\beta}(t), \Q[G/H]\rangle=\sum_{j=1}^N (d_{\eta, \overline{x}'_j}^{(t)} - 2\delta_{\overline{x}'_j}^{(t)}) -\deg(f_H) + \lvert S_{f_H}^{(t)}\lvert,
\end{equation}
where $\overline{x}'_1, \ldots, \overline{x}'_N$ are the geometric points of the special fiber of the normalized integral model of $Y_{K'}^{(t)}$ which are above the geometric point of special fiber $\mathfrak{D}_{s'}^{(t)}$ defined by the origin of $D_{s'}^{(t)}$ and the algebraically closed residue field $k$ (\ref{RelevementsPoints geometriques}) (see \ref{(V) KatoGeneral}, \eqref{d eta (BA)} and \ref{DiscEgalite} for the definition of the integers $\delta_{\overline{x}'_j}^{(t)}$ and $d_{\eta, \overline{x}'_j}^{(t)}$).
As $\lvert S_{f_H}^{(t)}\lvert $ is also the sum over $j=1, \ldots, N$ of the integers $\lvert P_s(\O_{\mathfrak{Y}^{(t)}, \overline{x}'_j})\lvert $, the combination of \eqref{Vanishing Cycles Lutke 1}, \eqref{DérivéeLutke 1}, both applied to $f_H$, and \eqref{DiscVar 13}, implies that, for $t\in ]r_i, r_{i-1}[$,
\begin{equation}
	\label{DiscVar 14}
\langle \widetilde{\sw}_f^{\beta}(t), \Q[G/H]\rangle=\sigma_i - \deg(f_H) + \delta_{f_H}(i)=\frac{d}{dt}\partial_{f_H}^{\alpha}(t+).
\end{equation}
\end{proof}

\begin{cor}
	\label{VariationReg}
We keep the assumption of {\rm \ref{DiscVar}}.
The function
\begin{equation}
	\label{VariationReg 1}
\Q_{\geq 0}\to \Q, ~ t\mapsto \langle \widetilde{a}_f^{\alpha}(t), \Q[G/H]\rangle
\end{equation}
is continuous and piecewise linear, with finitely many slopes which are all integers. The quantity $\langle \widetilde{\sw}_f^{\beta}(t), \Q[G/H]\rangle$, its derivative at $t$, is constant on each $]r_i, r_{i-1}[\cap \Q$ $($notation of {\rm \ref{DérivéeLutke}}$)$.
\end{cor}
\begin{proof}
This follows from \ref{DiscVar} and \ref{DérivéeLutke}.
\end{proof}

\begin{rem}
	\label{EtaleCanDivTrivial}
We note that, when the morphism $f: X\to D$ is étale and Galois of group $G$ and $H$ is a subgroup of $G$, then $X/H\to D$ is also étale and both $X$ and $X/H$ have trivial canonical sheaves.
\end{rem}

\subsection{}\label{Lambda-coefficients}
For the rest of this section, let $\ell$ be a prime number different from $p$ and $\overline{\Q}_{\ell}$ an algebraic closure of $\Q_{\ell}$.

\begin{thm}
	\label{Variation}
We assume that $f: X\to D$ is étale.
Let $\chi \in R_{\overline{\Q}_{\ell}}(G)$. The map
\begin{equation}
	\label{Variation 1}
\widetilde{a}_f^{\alpha}(\chi, \cdot): \Q_{\geq 0} \to \Q,\quad t\mapsto	\langle \widetilde{a}_f^{\alpha}(t), \chi\rangle
\end{equation}
is continuous and piecewise linear, with finitely many slopes which are all integers. Its right derivative at $t\in \Q_{\geq 0}$ is
\begin{equation}
	\label{Variation 2}
\frac{d}{dt} \widetilde{a}_f^{\alpha}(\chi, t+)=\langle \widetilde{\sw}_f^{\beta}(t), \chi\rangle.
\end{equation}
\end{thm}

\begin{proof}
By Artin's theorem \cite[\S 9.2, Corollaire]{Serre2}, we may assume that $\chi=\Ind_H^G \rho$, where $H$ is a cyclic subgroup of $G$ and $\rho: H\to \overline{\Q}_{\ell}^{\times}$ a character. By Frobenius reciprocity, we are reduced to proving the following statement.
\begin{itemize}
\item[] {\hfil $S_H(\rho)$ ~ The function $t\mapsto \langle \widetilde{a}_f^{\alpha}(t)_{\vert H}, \rho \rangle$ is continuous and piecewise linear, and its right derivative is $t\mapsto \langle \widetilde{\sw}_f^{\beta} (t)\lvert H, \rho \rangle$.}
\end{itemize}
We proceed by induction on the order $m$ of $\rho$, i.e. the smallest integer $i\geq 1$ such that $\rho^{i}=1_H$. For $m=1$, $\rho$ is trivial and $\chi$ is the character of the regular representation $\overline{\Q}_{\ell}[G/H]=\Ind_H^G 1_H$ of $G$. By \ref{VariationReg}, the function $t\mapsto \langle \widetilde{a}_f^{\alpha}(t), \overline{\Q}_{\ell}[G/H]\rangle$ is continuous and piecewise linear and its right derivative is $t\mapsto \langle \widetilde{\sw}_f^{\beta}(t), \overline{\Q}_{\ell}[G/H]\rangle$. Hence, the statement $S_H(1_H)$ holds.
Now, we assume that $m>1$ and that $S_H(\rho)$ is true if $\rho$ is of order $< m$. As $H$ is cyclic, it has a unique subgroup $I$ of index $m$. As
\begin{equation}
	\label{Variation 3}
\langle \widetilde{a}_f^{\alpha}(t)_{\vert H}, \overline{\Q}_{\ell}[H/I] \rangle=\langle \widetilde{a}_f^{\alpha}(t)_{\vert I}, 1_I\rangle \quad {\rm and}\quad \langle \widetilde{\sw}_f^{\beta} (t)_{\lvert H},  \overline{\Q}_{\ell}[H/I]\rangle=\langle \widetilde{\sw}_f^{\beta} (t)_{\lvert I}, 1_I\rangle,
\end{equation}
we see that $S_H( \overline{\Q}_{\ell}[H/I])=S_I(1_I)$. Hence,
the foregoing argument, where $H$ is replaced by $I$, implies that the statement $S_H(\overline{\Q}_{\ell}[H/I])$ also holds.
Now, the representation $\overline{\Q}_{\ell}[H/I]$ lies inside the regular representation $\overline{\Q}_{\ell}[H]$. In fact, it identifies with the direct sum of all characters of $H$ trivial on $I$, i.e. of order dividing $m$. Indeed, the restriction $\overline{\Q}_{\ell}[H]\lvert I$ is trivial; if $\widetilde{\chi}$ is a $\overline{\Q}_{\ell}$-valued irreducible character of $H$, then $\widetilde{\chi}$ is a direct summand of $\overline{\Q}_{\ell}[H/I]$ as it is readily seen that
\begin{equation}
	\label{Variation 4}
\langle \overline{\Q}_{\ell}[H/I], \widetilde{\chi}\rangle=\frac{\lvert I\lvert}{\lvert H\lvert}\dim_{\overline{\Q}_{\ell}}(\widetilde{\chi}) > 0.
\end{equation}
Hence, with our assumption above, we deduce that $S_H(\rho')$ is true, where $\rho'=\rho_1\oplus\cdots \oplus \rho_k$ is the sum of characters of $H$ of order $m$ and trivial on $I$ (among them is $\rho$). Denote $S_m$ the set formed by these latter characters; they correspond to the generators of the group of characters of $H/I$ and their number $k$ is the cardinal of $(\Z/m\Z)^{\times}$. Consider the natural action of $\Z$ on $R_{\overline{\Q}_{\ell}}(H)$ given by the operators 
\begin{equation}
	\label{Variation 5}
\Psi^j: \psi \mapsto (\sigma\mapsto \psi(\sigma^j))
\end{equation}
\cite[\S 9.2, Exercice 3]{Serre2}. The subgroup $m\Z$ acts trivially on $S_m$ and any $j\in \Z$ prime to $m$ stabilizes $S_m$, hence an action of $(\Z/m\Z)^{\times}$ on $S_m$.
Moreover, for any $j\in \Z$ prime to $m$, we have 
\begin{equation}
	\label{Variation 6}
\Psi^j(\widetilde{a}_f^{\alpha}(t)\lvert H)=\widetilde{a}_f^{\alpha}(t)\lvert H \quad {\rm and} \quad \Psi^j (\widetilde{\sw}_f^{\beta} (t)\lvert H)=\widetilde{\sw}_f^{\beta} (t)\lvert H.
\end{equation}
Indeed this follows from the relations
\begin{equation}
	\label{Variation 7}
\Psi^j (a_{G_{t, \tau}}^{\alpha}\lvert H)= a_{G_{t, \tau}}^{\alpha}\lvert H\quad {\rm and} \quad \Psi^j(\widetilde{a}_f^{\alpha}(t)\lvert H)=(\Ind_{G_{t, \tau}}^G \Psi^j (a_{G_{t, \tau}}^{\alpha}))\lvert H,
\end{equation}
and similarly for $\sw_{G_{t, \tau}}^{\beta}$. The first equality in \eqref{Variation 7} is directly implied by Remark \ref{puissance div} $(ii)$, applied to $\sigma$ and $\tau=\sigma^j$ in the notation of that remark, and the second is a straightforward computation.
Now for every $j\in \Z$ prime to $m$, we have
\begin{equation}
	\label{Variation 8}
\langle \widetilde{a}_f^{\alpha}(t)\lvert H, \rho\rangle=\langle \Psi^j(\widetilde{a}_f^{\alpha}(t)\lvert H), \Psi^j(\rho)\rangle=\langle \widetilde{a}_f^{\alpha}(t)\lvert H, \Psi^j(\rho)\rangle,
\end{equation}
where the first equality stems from the definition of the pairing $\langle\cdot, \cdot \rangle$ as a sum over $H$, and the second is \eqref{Variation 6}. A similar equality holds for $\widetilde{\sw}_f^{\beta} (t)\lvert H$. Thus, since the action of $(\Z/m\Z)^{\times}$ on $S_m$ is transitive, taking the sum over $j \in (\Z/m\Z)^{\times}$ in \eqref{Variation 8}, we obtain 
\begin{equation}
	\label{Variation 9}
\langle \widetilde{a}_f^{\alpha}(t)\lvert H, \rho\rangle=\frac{\langle \widetilde{a}_f^{\alpha}(t)\lvert H, \rho' \rangle}{\lvert (\Z/m\Z)^{\times}\lvert},
\end{equation}
and similarly with $\widetilde{\sw}_f^{\beta} (t)\lvert H$, which shows that $S_H(\rho)$ holds and concludes the proof.
\end{proof}

\subsection{}\label{Homomorphism de Cartan}
Let $\Lambda_{\Q}$ be a finite extension of $\Q_{\ell}$ inside $\overline{\Q}_{\ell}$ and $\overline{\Lambda}$ its residue field. We use the notation of \ref{Pairing}.
By \cite[16.1, Théorème 33]{Serre2}, the Cartan homomorphism $d_G: R_{\Lambda_{\Q}}(G)\to R_{\overline{\Lambda}}(G)$ is surjective.
\begin{lem}\label{ConductorsFiniteCoeffs}
Let $\overline{\chi}$ be an element of $R_{\overline{\Lambda}}(G)$ and $\chi$ a pre-image of $\overline{\chi}$ in $R_{\Lambda_{\Q}}(G)$. Then, for $t\in \Q_{\geq 0}$,
\begin{equation}
	\label{ConductorsFiniteCoeffs 1}
\widetilde{a}_f^{\alpha}(\overline{\chi}, t)=\langle \widetilde{a}_f^{\alpha}(t), \chi\rangle \quad {\rm and}\quad \widetilde{\sw}_f^{\beta}(\overline{\chi}, t)=\langle \widetilde{\sw}_f^{\beta}(t), \chi\rangle
\end{equation}
are independent of the choice of the pre-image $\chi$.
\end{lem}

\begin{proof}
We use notation from \ref{NotationsVariation} and \ref{Characters}.
From the definition \eqref{Characters 1} and Frobenius reciprocity, for any $\tau\in S^{(t)}$, we have 
\begin{equation}
	\label{ConductorsFiniteCoeffs 2}
\langle \widetilde{a}_f^{\alpha}(t), \chi \rangle=\lvert G_{t, \tau}\lvert \langle a^{\alpha}_{G_{t, \tau}}, \chi \lvert G_{t, \tau}\rangle \quad {\rm and}\quad \langle \widetilde{\sw}_f^{\beta}(t), \chi\rangle=\lvert G_{t, \tau}\lvert\langle \sw^{\beta}_{G_{t, \tau}}, \chi\lvert G_{t, \tau}\rangle.
\end{equation}
By \ref{Monogene}, the extension $\K_{t, \tau}^h/\K_t^h$ (\ref{Characters}) factors through a subfield $\K'_t$ which is a Galois extension of $\K_t^h$ such that, if $V'_t$ is the integral closure of $V^h_t$ in $K'_t$ and $\p$ (resp. $\p'$, resp. $\p_{\tau}$) is the height $1$ prime ideal of $V^h_t$ (resp. $V'_t$, resp. $V^h_t(\tau)$), then $(V'_t)_{\p'}/(V_t^h)_{\p}$ is an unramified extension and $(V^h_t(\tau))_{\p_{\tau}}/(V'_t)_{\p'}$ has ramification index one with purely inseparable and monogenic residue extension. The Galois group $G'_{t, \tau}$ of $\K_{t, \tau}/\K'_t$ is then a normal subgroup of $G_{t, \tau}$ of cardinality a power of $p$. The induction, from $G'_{t, \tau}$ to $G_{t, \tau}$, of the character $r_{G'_{t, \tau}}$ of the regular representation of $G'_{t, \tau}$ is the character $r_{G_{t, \tau}}$ of the regular representation of $G_{t, \tau}$. Moreover, by \cite[\S 7.2, Rem. 3)]{Serre2}, for a ($\Lambda_{\Q}$-valued) central function $\varphi$ on $G_{t, \tau}$, we have $\Ind_{G'_{t, \tau}}^{G_{t, \tau}} (\varphi\lvert G'_{t, \tau})=[G_{t, \tau}:G'_{t, \tau}]\cdot \varphi$. Therefore, taking $\varphi=a^{\alpha}_{G_{t, \tau}}$ and $\varphi=\sw^{\beta}_{G_{t, \tau}}$ successively, we see from \ref{ArtinSwanDifferente} and Frobenius reciprocity that
\begin{equation}
	\label{ConductorsFiniteCoeffs 3}
\langle a^{\alpha}_{G'_{t, \tau}}, \chi \lvert G'_{t, \tau}\rangle=[G_{t, \tau}:G'_{t, \tau}] \langle a^{\alpha}_{G_{t, \tau}}, \chi \lvert G_{t, \tau}\rangle - \frac{1}{\lvert G_{t, \tau}\lvert} v^{\alpha}(\mathfrak{d}_{V'_t/V^h_t}) \chi(1),
\end{equation}
\begin{equation}
\label{ConductorsFiniteCoeffs 4}
\langle \sw^{\beta}_{G'_{t, \tau}}, \chi \lvert G'_{t, \tau}\rangle=[G_{t, \tau}:G'_{t, \tau}] \langle \sw^{\beta}_{G_{t, \tau}}, \chi \lvert G_{t, \tau}\rangle - (\frac{1}{\lvert G_{t, \tau}\lvert}v^{\beta}(\mathfrak{d}_{V'_t/V^h_t})+1-[G_{t, \tau}:G'_{t, \tau}]) \chi(1).
\end{equation}
(In fact, as $(V'_t)_{\p'}/(V_t^h)_{\p}$ is unramified, $v^{\alpha}(\mathfrak{d}_{V'_t/V^h_t})=0$.) Now, as the ${\rm char}(\overline{\Lambda})\neq p$, \cite[18.2, Thm 42, Corollaire 2]{Serre2} implies that, for any element $\chi'$ of the kernel of $d_G$, we have ${\rm tr}_{\chi'}\lvert {G'_{r, \tau}}=0$ and thus \eqref{Pairing 1}
\begin{equation}
	\label{ConductorsFiniteCoeffs 5}
\langle a^{\alpha}_{G'_{t, \tau}}, \chi'\rangle=0 \quad {\rm and}\quad \langle \sw^{\beta}_{G'_{t, \tau}}, \chi'\rangle=0.
\end{equation}
Therefore, $\langle a^{\alpha}_{G'_{t, \tau}}, \chi \lvert G'_{t, \tau}\rangle$ and $\langle \sw^{\beta}_{G'_{t, \tau}}, \chi\lvert G'_{t, \tau}\rangle$ depend only on $\overline{\chi}$.
As $\chi(1)=\dim_{\Lambda_{\Q}}(\chi)$ depends only on $\overline{\chi}$, it follows from \eqref{ConductorsFiniteCoeffs 2}, \eqref{ConductorsFiniteCoeffs 3}) and \eqref{ConductorsFiniteCoeffs 4} that $\widetilde{a}_f^{\alpha}(\overline{\chi}, t)$ and $\widetilde{\sw}_f^{\beta}(\overline{\chi}, t)$ depend only on $\overline{\chi}$.
\end{proof}

\begin{rem}
	\label{Cartan-Z2-valuation}
What we have in fact shown in the above proof is the following.
Let $W/V$ be a monogenic integral extension of henselian $\Z^2$-valuation rings whose residue characteristic is $p>0$ and whose induced extension of fields of fractions $\L/\K$ is Galois of group $\mathbb{G}$. Let $\K'$ be the maximal unramified sub-extension of $\L/\K$ with respect to the discrete valuation ring $V_{\p}$ at the height $1$ prime ideal $\p$ of $V$ and $\mathbb{G}'$ the Galois group of $\L/\K'$.
Then, for $\overline{\chi}\in R_{\overline{\Lambda}}(\mathbb{G})$ and $\chi$ a pre-image of $\overline{\chi}$ in $R_{\Lambda_{\Q}}(\mathbb{G})$, we have
\begin{equation}
	\label{Cartan-Z2-valuation 1}
\lvert \mathbb{G}\lvert\langle a_{\mathbb{G}}^{\alpha}, \chi\rangle =\lvert \mathbb{G}'\lvert \langle a_{\mathbb{G}'}^{\alpha}, \chi\lvert \mathbb{G}'\rangle,
\end{equation}
\begin{equation}
	\label{Cartan-Z2-valuation 2}
\lvert \mathbb{G}\lvert\langle \sw_{\mathbb{G}}^{\beta}, \chi\rangle =\lvert \mathbb{G}'\lvert \langle \sw_{\mathbb{G}'}^{\beta}, \chi\lvert \mathbb{G}'\rangle,
\end{equation}
which follows readily from \ref{ArtinSwanDifferente} and, by  \cite[18.2, Thm 42, Corollaire 2]{Serre2}, implies that the pairings 
\begin{equation}
	\label{Cartan-Z2-valuation 3}
\langle a_{\mathbb{G}}^{\alpha}, \chi\rangle \quad {\rm and}\quad \langle \sw_{\mathbb{G}}^{\beta}, \chi\rangle
\end{equation}
depend only on $\overline{\chi}$.
\end{rem}

\begin{cor}
	\label{VariationCoeffsFinis} We assume that $X$ has \emph{trivial canonical sheaf}
and $\overline{\chi}\in R_{\overline{\Lambda}}(G)$. Then, the function $\widetilde{a}_f^{\alpha}(\overline{\chi}, \cdot)$ {\rm \eqref{ConductorsFiniteCoeffs}} is continuous and piecewise linear, with finitely many slopes which are all integers. Its right slope at $t$ is
\begin{equation}
	\label{VariationCoeffsFinis 1}
\frac{d}{d t} \widetilde{a}_f^{\alpha}(\overline{\chi}, t+)=\widetilde{\sw}_f^{\beta}(\overline{\chi}, t).
\end{equation} 
\end{cor}

\section{Characteristic cycles.}\label{Characteristic cycles}
\subsection{}\label{NotationsCharacteristicCycles}
In this section, we recall the constructions of Kato's characteristic cycle ${\rm KCC}_{\zeta}(\chi)$ and Abbes-Saito's characteristic cycle ${\rm CC}_{\psi}(\chi)$ in \cite[\S 3 and \S 4]{Hu1}.
Let $K$ be a complete discrete valuation field, $\mathcal{O}_K$ its valuation ring, $\m_K$ its maximal ideal, $k$ its residue field of characteristic $p>0$, and $\pi$ a uniformizer of $\O_K$. Let also $\overline{K}$ be a separable closure of $K$, $\O_{\overline{K}}$ the integral closure of $\O_K$ in $\overline{K}$, $\overline{k}$ its residue field, $G_K$ the Galois group of $\overline{K}$ over $K$, and $v:\overline{K}^{\times}\to \Q$ the valuation of $\overline{K}$ normalized by $v(\pi)=1$.

\subsection{}\label{Notations differentielles de Kato}
For a field $F$ and one-dimensional $F$-vector spaces $V_1,\ldots, V_m$, we define the $F$-algebra
\begin{equation}
	\label{Notations differentielles de Kato 1}
F\langle V_1,\ldots, V_m\rangle=\bigoplus_{(i_1,\ldots, i_m)\in \Z^m} V_1^{\otimes i_1}\otimes\cdots V_m^{\otimes i_m},
\end{equation}
We denote additively the law of the group of units $(F\langle V_1,\ldots, V_m\rangle)^{\times}$ of this $F$-algebra. More precisely, an element $x$ of $(F\langle V_1,\ldots, V_m\rangle)^{\times}$, which is in some $V_1^{\otimes i_1}\otimes\cdots \otimes V_m^{\otimes i_m}$, is denoted $[x]$ and we set $[x^{-1}]=-[x]$ and $[x\cdot y]=[x]+[y]$. If for each $1\leq i\leq m$, $e_i$ is a non-zero element of $V_i$, we have an isomorphism
\begin{equation}
	\label{Notations differentielles de Kato 2}
F\langle V_1,\ldots, V_m\rangle \xrightarrow{\sim} F[X_1, \ldots, X_m, X_1^{-1}, \ldots, X_m^{-1}], \quad e_i\mapsto X_i.
\end{equation}
Whence we deduce an isomorphism
\begin{equation}
	\label{Notations differentielles de Kato 3}
(F\langle V_1,\ldots, V_m\rangle)^{\times}\xrightarrow{\sim} F^{\times}\oplus \Z^m.
\end{equation}

\subsection{} \label{Extension de type (II)}
Let $L$ be a finite separable extension of $K$ in $\overline{K}$ with residue field $E=\O_L/\m_L$. We assume that the ramification index of $L/K$ is one and that residue extension $E/k$ is purely inseparable and monogenic, i.e $L/K$ is of type (II) in the terminology of \cite{K2} (as opposed to type (I) extensions which are the totally ramified ones).
Then, $\O_L$ is monogenic over $\O_K$. Let $h\in \O_L$ be an element whose reduction $\overline{h}\in E$ generates the extension $E/k$ of degree $p^n$ and let $a\in \O_K$ be a lift of $\overline{a}=h^{p^n}\in k$.

\subsection{}\label{Groupes SKL et SLK}
Let $Q$ be the kernel of the canonical homomorphism $\Omega^1_k \to \Omega^1_E$. It is a one-dimensional $k$-vector space generated by ${\rm d}\overline{a}$ (see \cite[3.4 (i)]{Hu1}, where it is denoted by $V$). The $E$-vector space $\Omega^1_{E/k}$ of relative differentials is also one-dimensional and is generated by ${\rm d}\overline{h}$ \cite[3.4 (ii)]{Hu1}. We set
\begin{equation}
	\label{Groupes SKL et SLK 1}
S_{K, L}=(k\langle \m_K/\m_K^2, Q\rangle)^{\times} \quad {\rm and}\quad S_{L/K}=(E\langle \m_L/\m_L^2, \Omega^1_{E/k}\rangle)^{\times} \quad \eqref{Notations differentielles de Kato 1}.
\end{equation}
As $L/K$ has ramification index $1$, we get an injective homomorphism of $k$-algebras
\begin{equation}
	\label{Groupes SKL et SLK 2}
k\langle \m_K/\m_K^2\rangle \hookrightarrow E\langle \m_L/\m_L^2 \rangle.
\end{equation}
We also have a canonical $E$-linear isomorphism (\cite[(1.6.1)]{K2} or \cite[(3.4.1)]{Hu1}
\begin{equation}
	\label{Groupes SKL et SLK 3}
E\otimes_k Q\xrightarrow{\sim} (\Omega^1_{E/k})^{\otimes p^n}
\end{equation}
which maps $y\otimes {\rm d}\overline{a}$ to $y({\rm d}\overline{h})^{\otimes p^n}$. Then, \eqref{Groupes SKL et SLK 2} and \eqref{Groupes SKL et SLK 3} induce a canonical injective map
\begin{equation}
	\label{Groupes SKL et SLK 4}
S_{K, L}\hookrightarrow S_{L/K}
\end{equation}
which sends $[{\rm d}\overline{a}]$ to $p^n [{\rm d}\overline{h}]$.

\subsection{}\label{RacinePrimitiveEtEpsilon(zeta)}
We assume that $L/K$ is Galois of type (II) with Galois group $G$. Let $C$ be an algebraically closed field of characteristic zero and denote by $\widetilde{\Z}$ the integral closure of $\Z$ in $C$. Let also $\zeta \in \widetilde{\Z}$ be a primitive $p$-th root of unity. For an element $\chi$ of the Grothendieck group $R_C(G)$ of finitely generated $C[G]$-modules, we put $\langle\chi, 1\rangle= \lvert G\lvert^{-1}\sum_{\sigma\in G} {\rm tr}_{\chi}(\sigma)$. We also set
\begin{equation}
	\label{RacinePrimitiveEtEpsilon(zeta) 1}
\varepsilon(\zeta)=\sum_{r\in \mathbb{F}_p^{\times} \subseteq E^{\times}}[r]\otimes \zeta^r ~ \in S_{L/K}\otimes_{\Z}\widetilde{\Z}.
\end{equation}

\subsection{}\label{Swan conductor with differential values}
We keep the assumptions of \ref{RacinePrimitiveEtEpsilon(zeta)}. For $\sigma\in G-\{1\}$, we put
\begin{equation}
	\label{Swan conductor with differential values 1}
s_G(\sigma)=[{\rm d}\overline{h}]-[h-\sigma(h)] \quad \in S_{L/K},
\end{equation}
where $[h-\sigma(h)]$ abusively denotes the class of $h-\sigma(h)$ viewed in $(\m_L/\m_L^2)^{\otimes v(h-\sigma(h))}$. This definition is independent of the choice of the generator $h$ \cite[1.8]{K2}. We also put
\begin{equation}
	\label{Swan conductor with differential values 2}
s_G(1)=-\sum_{\sigma\in G-\{1\}} s_G(\sigma)\quad \in S_{L/K}.
\end{equation}

For $\chi\in R_C(G)$, Kato defines the \textit{Swan conductor with differential values} as
\begin{equation}
	\label{Swan conductor with differential values 3}
\sw_{\zeta}(\chi)=\sum_{\sigma}s_G(\sigma)\otimes {\rm tr}_{\chi}(\sigma) + (\dim_C(\chi)-\langle \chi, 1\rangle)\varepsilon(\zeta) ~ \in S_{L/K}\otimes_{\Z}\widetilde{\Z}.
\end{equation}
For any $r\in \mathbb{F}_p^{\times}$, we have $\sw_{\zeta^r}(\chi)=\sw_{\zeta}(\chi) +  (\dim_C(\chi)-\langle \chi, 1\rangle)[r]$.
For $H$ a subgroup of $G$ and $\chi\in R_C(H)$, we have (\cite[3.3]{K2} or \cite[3.15]{Hu1})
\begin{equation}
\label{Swan conductor with differential values 4}
\sw_{\zeta}(\Ind_H^G \chi)=[G:H]\sw_{\zeta}(\chi) + (\dim_C(\chi)-\langle\chi, 1\rangle)\mathfrak{D}(L^H/K),
\end{equation}
where $L^H$ is the subfield of $L$ fixed by $H$ and $\mathfrak{D}(L^H/K)\in S_{L/K}$ is Kato's \textit{different} of $L^H/K$ (see \textit{loc. cit.} for the definition; we will not need to explicitly use \eqref{Swan conductor with differential values 2} anyway). If $H$ is  normal subgroup of $G$, $\chi'\in R_C(G/H)$ and $\chi$ the image of $\chi'$ under the canonical map $R_C(G/H)\to R_C(G)$, then $\sw_{\zeta}(\chi)=\sw_{\zeta}(\chi')$ \cite[3.3]{K2}, \cite[3.14]{Hu1}.

Kato proved the following generalization of the Hasse-Arf theorem.
\begin{thm}
	\label{HasseArf differentiel}
We keep the notation of {\rm \ref{Groupes SKL et SLK}}, {\rm \ref{RacinePrimitiveEtEpsilon(zeta)}} and {\rm \ref{Swan conductor with differential values}}.  For any $\chi\in R_C(G)$, we have \eqref{Groupes SKL et SLK 4}
\begin{equation}
	\label{HasseArf differentiel 1}
\sw_{\zeta}(\chi) \in S_{K, L}\subseteq S_{L/K}. 
\end{equation}
\end{thm}

\subsection{}\label{KCC}
We keep the assumptions of \ref{RacinePrimitiveEtEpsilon(zeta)} and follow \cite[3.17]{Hu1}. For $\chi\in R_C(G)$, the conductor $\sw_{\zeta}(\chi) \in  S_{K, L}$ can thus be written as
\begin{equation}
	\label{KCC 1}
\sw_{\zeta}(\chi)=[\Delta'] + [\pi^c] - m {\rm d}\overline{a},
\end{equation} 
where $\Delta'$ is a nonzero element of $k$, $\pi$ is the uniformizer of $\O_K$ fixed in \ref{NotationsCharacteristicCycles}, $c$ in an integer and $m=\dim_C(\chi) - \langle\chi, 1\rangle$. Then, we define as Kato's \textit{characteristic cycle of} $\chi$ the differential \footnote{We here point out that there is a mistake in the definition of ${\rm KCC}_{\zeta}(\chi)$ given in \cite[(3.17.1)]{Hu1}, where $\Delta'$ is used instead of $\Delta^{\prime -1}$. This however does not affect the results in \textit{loc. cit.}, since it is the correct definition that was subsequently used.}
\begin{equation}
	\label{KCC 2}
{\rm KCC}_{\zeta}(\chi)=\Delta^{\prime -1} ({\rm d}\overline{a})^m ~ \in (\Omega^1_k)^{\otimes m}.
\end{equation}
We note that, as the decomposition \eqref{KCC 1} depends on the chosen uniformizer $\pi$ of $\O_K$, so does ${\rm KCC}_{\zeta}(\chi)$.

\subsection{}\label{Type (II) over Unramified}
If $L$ is an extension of type (II) not over $K$ but over a larger sub-field $K'$ such that $K'/K$ is unramified, then, we can still define $\sw_{\zeta}(\chi)$ by setting 
\begin{equation}
	\label{Type (II) over Unramified 1}
\sw_{\zeta}(\chi)=\sw_{\zeta}(\chi\lvert \Gal(L/K'))
\end{equation}
and we still have $\sw_{\zeta}(\chi) \in S_{K, L}$ \cite[3.15]{K2}, \cite[3.18]{Hu1}.

\subsection{}\label{Changement de CoeffSwan Diff}
Let $L$ be a Galois extension of $K$ of group $G$ which is of type (II) over a sub-extension unramified over $K$. Let $\overline{\Lambda}^{\rm alg}$ be an algebraically closed field of characteristic $\ell\notin \{0, p\}$. Let $C$ be an algebraic closure of the field of fractions of the ring of Witt vectors $W(\overline{\Lambda}^{\rm alg})$. Then, the Cartan homomorphism $R_C(G)\to R_{\overline{\Lambda}^{\rm alg}}(G)$ is surjective \cite[16.1, Théorème 33]{Serre2}. Let $\chi$ be an element of $R_{\overline{\Lambda}^{\rm alg}}(G)$ and $\widehat{\chi}\in R_C(G)$ be a pre-image of $\chi$. Let $\zeta\in \overline{\Lambda}^{\rm alg}$ be a primitive $p$-th root of unity and $\widehat{\zeta}$ be the unique (primitive) $p$-th root of unity in $C$ lifting $\zeta$. Then, through \ref{Type (II) over Unramified}, we define \cite[3.16]{K2}
\begin{equation}
	\label{Changement de CoeffSwan Diff 1}
\sw_{\zeta}(\chi)=\sw_{\widehat{\zeta}}(\widehat{\chi}).
\end{equation}
This definition is independent of $\widehat{\chi}$ by \cite[18.2, Théorème 42, Corollaire 2]{Serre2} and \eqref{Swan conductor with differential values 3}. It follows that we also have a well-defined characteristic cycle ${\rm KCC}_{\zeta}(\chi)$ \eqref{KCC 2}.

\subsection{}\label{The logarithmicRamifFilt}
Abbes and Saito defined a decreasing filtration $(G_{K, \log}^r)_{r\in \Q_{>0}}$ of $G_K$ by closed normal subgroups, called \textit{the logarithmic ramification filtration} \cite[3.12]{A.S.1}. We denote by $G_{K, \log}^0$ the inertia subgroup of $G_K$. For any $r\in \Q_{r\geq 0}$, we put 
\begin{equation}
	\label{The logarithmicRamifFilt 1}
G_{K, \log}^{r+}=\overline{\cup_{s>r} G_{K, \log}^r} \quad {\rm and}\quad {\rm Gr}_{\log}^r G_K=G_{K, \log}^r/G_{K, \log}^{r+}.
\end{equation}
By \cite[3.15]{A.S.1}, $P=G_{K, \log}^{0+}$ is the wild inertia subgroup of $G_K$, i.e. the $p$-Sylow subgroup of $G_{K, \log}^0$.

Now, let $L\subseteq \overline{K}$ be a finite separable extension of $K$. For a rational number $r\geq 0$, we say that the (logarithmic) ramification of $L/K$ is bounded by $r$ (resp. $r+$) if $G_{K, \log}^r$ (resp. $G_{K, \log}^{r+}$) acts trivially on $\Hom_K(L, \overline{K})$ through its action on $\overline{K}$. The \textit{logarithmic conductor} $c=c(L/K)$ of $L/K$ is defined to be the infimum of rational numbers $r>0$ such that the ramification of $L/K$ is bounded by $r$. Then, $c$ is a rational number and the ramification of $L/K$ is bounded by $c+$ \cite[9.5]{A.S.1}. However, if $c>0$, then the ramification $L/K$ is not bounded by $c$.

\begin{thm}[{\cite[Theorem 1]{A.S.2}}]
	\label{GradedQuotientAbelian}
For every rational number $r>0$, the group ${\rm Gr}_{\log}^r$ is abelian and lies in the center of the pro-$p$ group $P/G_{K, \log}^{r+}$.
\end{thm}

\begin{lem}[{\cite[1.1]{Katz}, \cite[6.4]{A.S.3}}]
	\label{SlopeDecomposition}
Let $M$ be a $\Z[\frac{1}{p}]$-module on which $P$ acts through a finite quotient, say by $\rho: P\to {\rm Aut}_{\Z}(M)$. Then,
\begin{enumerate}
\item[$(i)$] The module $M$ has a unique direct sum decomposition
\begin{equation}
	\label{SlopeDecomposition 1}
M=\bigoplus_{r\in \Q_{\geq 0}} M^{(r)}
\end{equation}
into $P$-stables submodules $M^{(r)}$ such that $M^{(0)}=M^P$ and for every $r>0$, 
\begin{equation}
	\label{SlopeDecomposition 2}
(M^{(r)})^{G_{K, \log}^r}=0 \quad {\rm and}\quad (M^{(r)})^{G_{K, \log}^{r+}}=M^{(r)}.
\end{equation}
\item[$(ii)$] If $r>0$, then $M^{(r)}=0$ for all but the finitely many values of $r$ for which $\rho(G_{K, \log}^{r+})\subsetneq \rho(G_{K, \log}^r)$.
\item[$(iii)$] For a fixed $r\geq 0$ and variable $M$, the functor $M\mapsto M^{(r)}$ is exact.
\item[$(iv)$] For $M, N$ as above, we have $\Hom_{P-{\rm mod}}(M^{(r)}, N^{(r')})=0$ if $r\neq r'$.
\end{enumerate}
\end{lem}

\begin{defi}\label{SlopeDecompositionDefinition}
The decomposition $M=\bigoplus_r M^{(r)}$\eqref{SlopeDecomposition 1} is called \textit{the slope decomposition} of $M$. The values $r\geq 0$ for which $M^{(r)}\neq 0$ are the \textit{slopes} of $M$. We say that $M$ is \textit{isoclinic} if it has only one slope.
\end{defi}

\subsection{}\label{Le caractere}
For the rest of this section, we fix a prime number $\ell$ different from $p$. From now on until \ref{ComparisonThmCharacteristicCycles} included, $\Lambda$ is a local $\Z_{\ell}$-algebra which is of finite type as a $\Z_{\ell}$-module and $\psi:\mathbb{F}_p\to \Lambda^{\times}$ is a nontrivial character.

\begin{lem}[{\cite[6.7]{A.S.3}}]
	\label{CentralCharacterDecomposition}
Let $M$ be a $\Lambda$-module on which $P$ acts $\Lambda$-linearly through a finite discrete quotient, which is isoclinic of slope $r>0$. So the action of $P$ on $M$ factors through the quotient group $P/G_{K, \log}^{r+}$.
\begin{enumerate}
\item[$(i)$] Let $X(r)$ be the set of isomorphism classes of characters $\chi: {\rm Gr}_{\log}^r G_K\to \Lambda_{\chi}^{\times}$ such that $\Lambda_{\chi}$ is a finite étale $\Lambda$-algebra, generated by the image of $\chi$ and with connected spectrum. Then, $M$ has a unique direct sum decomposition 
\begin{equation}
	\label{CentralCharacterDecomposition 1}
M=\bigoplus_{\chi\in X(r)} M_{\chi},
\end{equation}
where each $M_{\chi}$ is a $P$-stable $\Lambda$-submodule on which $\Lambda[G_{K, \log}^r]$ acts through $\Lambda_{\chi}$.
\item[$(ii)$] There are finitely many characters $\chi\in X(r)$ such that $M_{\chi}\neq 0$.
\item[$(iii)$] For a fixed $\chi$ and variable $M$, the functor $M\mapsto M_{\chi}$ is exact.
\item[$(iv)$] For $M, N$ as above, we have $\Hom_{\Lambda[P]}(M_{\chi}, N_{\chi})=0$ if $\chi\neq\chi'$.
\end{enumerate}
\end{lem}

\begin{defi}\label{CentralCharacterDefinition}
The decomposition $M=\bigoplus_{\chi} M_{\chi}$ \eqref{CentralCharacterDecomposition 1} is called the \textit{central character decomposition} of $M$. The characters $\chi:{\rm Gr}_{\log}^r G_K\to\Lambda_{\chi}^{\times}$ for which $M_{\chi}\neq 0$ are the \textit{central characters} of $M$.
\end{defi}

\begin{rem}\label{LambdachiLambda}
Let $P_0$ be a finite discrete quotient of $P/G_{K, \log}^{r+}$ through which $P$ acts on $M$ and let $G^r_0$ be the image of ${\rm Gr}_{\log}^r G_K$ in $P_0$. Then, $G^r_0$ is abelian by \ref{GradedQuotientAbelian} and thus $\Lambda[G^r_0]$ is a commutative ring. The connected components of $\Spec(\Lambda[G^r_0])$ correspond to the isomorphism classes of characters $\chi: G^r_0\to \Lambda_{\chi}$, where $\Lambda_{\chi}$ is a finite étale $\Lambda$-algebra generated by the image of $\chi$ and with connected spectrum. If $p^n G^r_0=0$ and $\Lambda$ contains a $p^n$-th primitive root of unity, then $\Lambda_{\chi}=\Lambda$ for every $\chi$ satisfying $M_{\chi}\neq 0$.
\end{rem}

\subsection{}\label{LogarithmicDifferentials}
For the rest of this section, recalling the notation fixed in \ref{NotationsCharacteristicCycles}, we assume that $k$ is of finite type over a perfect subfield $k_0$. We define the $k$-vector space $\Omega_k^1({\log})$ by
\begin{equation}
	\label{LogarithmicDifferentials 1}
\Omega_k^1({\log})=(\Omega^1_{k/k_0}\oplus (k\otimes_{\Z} K^{\times}))/({\rm d}\overline{b}-\overline{b}\otimes b; b\in \O_K^{\times}).
\end{equation}
For a rational number $r$, we also define $\m^r_{\overline{K}}$ and $\m^{r+}_{\overline{K}}$ by
\begin{equation}
	\label{LogarithmicDifferentials 2}
\m^r_{\overline{K}}=\{x\in \overline{K} ~ \lvert ~ v(x)\geq r\} \quad {\rm and}\quad \m^{r+}_{\overline{K}}=\{x\in \overline{K} ~ \lvert ~ v(x)> r\}.
\end{equation}

\begin{thm}[{\cite[1.24]{Saito1}, \cite[Theorem 2]{Saito2}}]
	\label{MorphismRefinedSwanConductor}
For every rational number $r>0$, we have 
\begin{enumerate}
\item[$(i)$] The abelian group ${\rm Gr}_{\log}^r G_K$ is killed by $p$.
\item[$(ii)$] There is an injective group homomorphism
\begin{equation}
	\label{MorphismRefinedSwanConductor 1}
{\rm rsw}: \Hom({\rm Gr}^r_{\log} G_K, \mathbb{F}_p)\to \Hom_{\overline{k}}(\m^r_{\overline{K}}/\m^{r+}_{\overline{K}}, \Omega^1_k(\log)\otimes_k\overline{k}).
\end{equation}
\end{enumerate}
\end{thm}

The morphism \eqref{MorphismRefinedSwanConductor 1} is called the \textit{Refined Swan conductor}.

\subsection{} \label{CC}
Let $L \subset \overline{K}$ be a finite Galois extension of $K$ of group $G$. Recall that $\Lambda$ is a local $\Z_{\ell}$-algebra which is of finite type as a $\Z_{\ell}$-module. Let $M$ be a free $\Lambda$-module of finite rank on which $G$ acts linearly. Following \eqref{SlopeDecomposition 1}, let
\begin{equation}
	\label{CC 1}
M=\bigoplus_{r\in \Q_{\geq 0}} M^{(r)}
\end{equation}
be the slope decomposition of $M$. We define the \textit{Abbes-Saito logarithmic Swan conductor} to be the following rational number
\begin{equation}
	\label{CC 2}
\sw_G^{\rm AS}(M)=\sum_{r\in \Q_{\geq 0}} r\cdot \dim_{\Lambda} M^{(r)},
\end{equation}
where $\dim_{\Lambda}(M^{(r)}$ abusively denotes the rank of the $\Lambda$-module $M^{(r)}$.
Clearly, $\sw_G^{\rm AS}(M)=0$ if and only if the wild inertia $P$ acts trivially on $M$ (\ref{SlopeDecomposition} (i)).

For each rational number $r>0$, following \eqref{CentralCharacterDecomposition 1}, let 
\begin{equation}
	\label{CC 3}
M^{(r)}=\bigoplus_{\chi\in X(r)} M^{(r)}_{\chi}
\end{equation}
be the central character decomposition of $M^{(r)}$. Then, each $M^{(r)}_{\chi}$ is a free $\Lambda$-module. As ${\rm Gr}_{\log}^r$ is killed by $p$, the existence of $\psi$ ensures that $\chi$ factors as
\begin{equation}
	\label{CC 4}
{\rm Gr}_{\log}^r G_K \xrightarrow{\overline{\chi}} \mathbb{F}_p \xrightarrow{\psi} \Lambda^{\times}.
\end{equation}
Following \cite[(4.12.1)]{Hu1}, we define the Abbes-Saito \textit{characteristic cycle} ${\rm CC}_{\psi}(M)$ of $M$ by
\begin{equation}
	\label{CC 5}
{\rm CC}_{\psi}(M)=\bigotimes_{r\in \Q{>0}}\bigotimes_{\chi\in X(r)} ({\rm rsw}(\overline{\chi})(\pi^r))^{\otimes(\dim_{\Lambda} M^{(r)}_{\chi})} ~ \in (\Omega^1_k(\log)\otimes_k\overline{k})^{\otimes \dim_{\Lambda} M/M^{(0)}}.
\end{equation}
Denoting by $b$ a common denominator of all the slopes of $M$, for any such slope $r$, the element $\pi^r$ is defined up to a choice of a $b$-th root of $1$. Another choice changes the right-hand side of \eqref{CC 5} by a factor $\xi^{b\sw_G^{\rm AS}(M)}$, where $\xi$ is a $b$-th root of unity, which disappears as $\sw_G^{\rm AS}(M)$ is an integer (\cite[4.4.3]{Xiao1}, \cite[4.5.14]{Xiao2} and \cite[4.3.1]{Saito4}). Thus, ${\rm CC}_{\psi}(M)$ is unambiguously defined. We note that, just like ${\rm KCC}_{\zeta}(\chi_M)$, where $\chi_M$ is the image of $M$ in $R_{\Lambda}(G)$ \eqref{KCC 2}, ${\rm CC}_{\psi}(M)$ depends on the chosen uniformizer $\pi$ of $\O_K$.

The main result of \cite{Hu1} is the following comparison theorem for characteristic cycles.

\begin{thm}[{\cite[Theorem 10.4]{Hu1}}]
	\label{ComparisonThmCharacteristicCycles} 
We use the notation of {\rm \ref{RacinePrimitiveEtEpsilon(zeta)}}, {\rm \ref{KCC}} and {\rm \ref{CC}}. We assume that $p$ is not a uniformizer of $K$, that $\Lambda\subset C$ and that the extension $L/K$ is of type {\rm (II)}. Let $M$ be a finite free $\Lambda$-module with a $\Lambda$-linear action of $G$. Then, for the same uniformizer $\pi$, we have
\begin{equation}
	\label{ComparisonThmCharacteristicCycles 1} 
{\rm KCC}_{\psi(1)}(\chi_M)={\rm CC}_{\psi}(M) \quad {\rm in}\quad (\Omega^1_k)^{\otimes m},
\end{equation}
where $\chi_M$ is the image in $R_C(G)$ $($of the base change to $C)$ of $M$ and the exponent $m$ is the integer $\dim_C \chi_M - \langle\chi_M, 1\rangle =\dim_{\Lambda}(M/M^{(0)})$.
\end{thm}

\subsection{}\label{KCC CCCoeffFinis}
\emph{For the rest of this section}, let $\overline{\Lambda}$ be a finite field of characteristic $\ell\neq p$ which contains a primitive $p$-th root of unity, $\Lambda=W(\overline{\Lambda})$ its ring of Witt vectors, $\Lambda_{\Q}$ the field of fractions of $\Lambda$, $C$ an algebraic closure of $\Lambda_{\Q}$ and $\psi: \mathbb{F}_p\to \overline{\Lambda}^{\times}$ be a non-trivial character.

Here, we follow \cite[10.7]{Hu1}.
Assume that $p$ is not a uniformizer of $K$. Let $L$ be a finite Galois extension of $K$ of group $G$ which of type (II) over a larger subfield $K'$ such that $K'/K$ is unramified. The group $\Gal(L/K')$ has cardinality a power of $p$. Let $M$ be a $\overline{\Lambda}$-vector space of finite dimension with a linear action of $G$. Then, using the same uniformizer $\pi$, the Abbes-Saito characteristic cycle ${\rm CC}_{\psi}(M)$ \eqref{CC 5} and Kato's characteristic cycle ${\rm KCC}_{\psi(1)}(\chi_M)$ (\ref{Type (II) over Unramified}, \ref{Changement de CoeffSwan Diff}), where $\chi_M$ is (any pre-image by the Cartan homomorphism $R_{\Lambda_{\Q}}(G)\to R_{\overline{\Lambda}}(G)$ of) the image of $M$ in $R_{\overline{\Lambda}}(G)$, are both well-defined  and lie in $(\Omega^1_k)^{\otimes m}$, where the exponent $m$ is $\dim_{\overline{\Lambda}}(M/M^{(0)})$.
Then, H. Hu shows that we still have
\begin{equation}
	\label{KCC CCCoeffFinis 1}
{\rm KCC}_{\psi(1)}(\chi_M)={\rm CC}_{\psi}(M) \quad {\rm and}\quad (\Omega^1_k)^{\otimes m}.
\end{equation}

\subsection{} \label{CC-KCC-Z2-Valuations}
Let $V$ be a henselian $\Z^2$-valuation ring and $W/V$ a monogenic integral extension of (henselian) $\Z^2$-valuation rings (\ref{MonogenicIntegral}) with residue characteristic $p>0$ and trivial residue extension. Let $\p$ (resp. $\q$) be the height $1$ prime ideal of $V$ (resp. $W$). Let $\K$ (resp. $\L$) be the field of fractions of $V$ (resp $W$) and assume that $\L/\K$ is finite and Galois of group $\mathbb{G}$. The valuation ring $V$ is equipped with a normalized $\Z^2$-valuation map $v: \K^{\times}\to \Z^2$, and first and second projection maps $v^{\alpha}, v^{\beta}: \K^{\times}\to \Z$ (\ref{AlphaEtBeta}). Let $\widehat{\K}_{\p}$ (resp. $\widehat{\L}_{\q}$) be the field of fractions of the $\p$-adic (resp. $\q$-adic) completion $\widehat{V_{\p}}$ of $V_{\p}$ (resp. $\widehat{W_{\q}}$ of $W_{\q}$). Then, $\widehat{V_{\p}}\subset \widehat{W_{\q}}$ is a monogenic integral extension of complete discrete valuation rings and $\widehat{\L}_{\q}/\widehat{\K}_{\q}$ is Galois of group denoted $\widehat{\mathbb{G}}$. Let $\pi_{\p}$ be a uniformizer of $V_{\p}$. Then, we have well-defined class functions on $G$, $a_{\widehat{\mathbb{G}}}^{\alpha}$ and $\sw_{\widehat{\mathbb{G}}}^{\beta}$ with values in $\Q$ and $\Z$ respectively \eqref{ArtinSwan}.

The quotient $V/\p$ is a discrete valuation ring, whose field of fraction is the residue field $\kappa(\p)$ of $V_{\p}$ (and $\widehat{V_{\p}}$) with valuation map $\ord_{\p}: \kappa(\p)^{\times}\to \Z$ which satisfies $v^{\beta}(x)=\ord_{\p}(x\pi_{\p}^{-v^{\alpha}(x)}~{\rm mod}~\p)$ for any $x\in V$. We denote again by $\ord_{\p}: \Omega^1_{\kappa(\p)}-\{0\}\to \Z$ the valuation defined by $\ord_{\p}(b {\rm d}a)=\ord_{\p}(b)$, if $a, b\in \kappa(\p)^{\times}$ and $\ord_{\p}(a)=1$. It can be canonically extended multiplicatively to $\ord_{\p}: (\Omega^1_{\kappa(\p)})^{\otimes m}-\{0\}\to \Z$, for any integer $m>0$.

Moreover, with respect to $W_{\q}/V_{\p}$, the extension $\L/\K$ is of type (II) over a subfield of $\L$ which is unramified over $\K$. Hence, so is the extension $\widehat{\L}_{\q}/\widehat{\K}_{\p}$. Hence, by \ref{Type (II) over Unramified}, we can use the notation of \ref{RacinePrimitiveEtEpsilon(zeta)} and put $K=\widehat{\K}_{\p}$, $L=\widehat{\L}_{\q}$, $G=\widehat{\mathbb{G}}$ and $\pi=\pi_{\p}$.

We use the notation of \ref{KCC CCCoeffFinis}. For a finite dimensional $\overline{\Lambda}$-representation $M$ of the quotient $G$ of $G_K$, we have well-defined characteristic cycles ${\rm KCC}_{\psi(1)}(\chi_M)$ and ${\rm CC}_{\psi}(M)$ \eqref{KCC CCCoeffFinis} attached to the same $\pi$. Moreover, if $p$ is not a uniformizer of $V_{\p}$, then the following identity holds \eqref{KCC CCCoeffFinis 1}
\begin{equation}
	\label{CC-KCC-Z2-Valuations 1}
{\rm KCC}_{\psi(1)}(\chi_M)={\rm CC}_{\psi}(M) \quad {\rm in}\quad (\Omega^1_{\kappa(\p)})^{\otimes (\dim_{\overline{\Lambda}}(M/M^{(0)}))}
\end{equation}
where $M^{(0)}$ is the tame part of $M$, i.e. the sub-module of $M$ fixed by wild inertia subgroup of $G$ (\ref{SlopeDecomposition} (i)).

Independent of the chosen pre-image of $\chi_M$ in $R_{\Lambda_{\Q}}(G)$, we also have well-defined pairings \eqref{Cartan-Z2-valuation 3}
\begin{equation}
	\label{CC-KCC-Z2-Valuations 2}
\langle a_G^{\alpha}, \chi_M\rangle \quad {\rm and}\quad
 \langle \sw_G^{\beta}, \chi_M\rangle.
\end{equation}

\begin{prop}
	\label{Comparaison Kato AbbesSaito}
We use the notation of \eqref{CC-KCC-Z2-Valuations}. Assume that $p$ is not a uniformizer of $V_{\p}$. Let $M$ be a finite dimensional $\overline{\Lambda}$-representation of $G$. Then, we have the identities
\begin{equation}
	\label{Comparaison Kato AbbesSaito 1}
\lvert G\lvert \langle a_G^{\alpha}, \chi_M\rangle = \sw_G^{\rm AS}(M),
\end{equation}
\begin{equation}
	\label{Comparaison Kato AbbesSaito 2}
\lvert G\lvert \langle \sw_G^{\beta}, \chi_M\rangle = -\ord_{\p}({\rm CC}_{\psi}(M)) - \dim_{\overline{\Lambda}}(M/M^{(0)}).
\end{equation}
\end{prop}

\begin{proof}
We denote by $\zeta$ the $p$-th root of unity $\psi(1)$. We will prove the identities \eqref{Comparaison Kato AbbesSaito 1} and \eqref{Comparaison Kato AbbesSaito 2} by comparing the Swan conductor $\sw_G(M) \in \Z^2$ (\textit{cf.} Theorem \ref{HasseArf}) with Kato's Swan conductor with differential values $\sw_{\zeta}(\chi_M)$ \eqref{Swan conductor with differential values 3}, where $\chi_M\in R_{\Lambda_{\Q}}(G)$ now denotes a characteristic zero lift of the image of $M$ in $R_{\overline{\Lambda}}(G)$. We may assume that $L/K$ is an extension of type (II). Indeed, this follows from \eqref{Cartan-Z2-valuation} and the stability of the logarithmic ramification filtration under tame base change \cite[3.15 (3)]{A.S.1}.
We use (some of) the notation of \ref{Groupes SKL et SLK}. We denote by $Q$ the kernel of the canonical map $\Omega^1_{\kappa(\p)}\to \Omega^1_{\kappa(\q)}$, which is a $1$-dimensional $\kappa(\p)$-vector space generated by ${\rm d}\overline{a}$. We recall also that $\Omega^1_{\kappa(\q)/\kappa(\p)}$ is also a one-dimensional $\kappa(\q)$-vector space generated by ${\rm d}\overline{h}$. As $V/\p \to W/\q$ is an extension of discrete valuation rings with induced extension of fields of fractions  $\kappa(\q)/\kappa(\p)$ and trivial residue extension \eqref{CC-KCC-Z2-Valuations}, we can assume that $W/\q=V/\p[\overline{h}]$, with $\overline{h}$ is a uniformizer of $W/\q$ \cite[III, \S 6, Prop. 12 \& Lemme 4]{Serre1}. Then, since $\overline{h}^{p^n}=\overline{a}$, with $p^n=\lvert G\lvert=[\kappa(\q):\kappa(\p)]$, $\overline{a}$ is a uniformizer of $V/\p$.
Recall that $C$ is an algebraic closure of $\Lambda_{\Q}$ \eqref{CC-KCC-Z2-Valuations} and let
\begin{equation}
	\label{Comparaison Kato AbbesSaito 3}
\varphi: S_{L/K}=(\kappa(\q)\langle \q/\q^2, \Omega^1_{\kappa(\q)/\kappa(\p)}\rangle)^{\times}\to \kappa(\q)^{\times}\oplus \Q^2 \hookrightarrow \kappa(\q)^{\times}\oplus C^2
\end{equation}
be the composition of the canonical inclusion $\kappa(\q)^{\times}\oplus \Q^2 \hookrightarrow \kappa(\q)^{\times}\oplus C^2$ with the non-canonical homomorphism which sends $\pi$ mod $\q^2$ and ${\rm d}\overline{h}$ to $(1, 1, 0)$ and $(1, 0, 1)/\lvert G\lvert$ respectively. As the injection $S_{K, L}\hookrightarrow S_{L/K}$ identifies ${\rm d}\overline{a}$ with $\lvert G\lvert {\rm d}\overline{h}$ \eqref{Groupes SKL et SLK 4}, the composition $S_{K, L} \hookrightarrow S_{L/K}\xrightarrow{\varphi} \kappa(\q)^{\times}\oplus C^2$ sends ${\rm d}\overline{a}$ to $(1, 0, 1)$, inducing an isomorphism $S_{K, L}\xrightarrow{\sim}
\kappa(\p)^{\times}\oplus \Z^2$. 
We consider also the map
\begin{equation}
	\label{Comparaison Kato AbbesSaito 4}
\varpi: \kappa(\q)^{\times}\oplus C^2\to C^2, \quad (x, y, z) \mapsto (y, \ord_{\p}(x) + z\ord_{\p}({\rm d}\overline{a})).
\end{equation}
We can write $\sw_{\zeta}(\chi_M)=[\Delta']+ [\pi^c] - m [{\rm d}\overline{a}]$ \eqref{KCC 1}, where $\Delta'\in \kappa(\p)^{\times}$, $c$ is an integer and $m=\dim_{\overline{\Lambda}}(M/M^{(0)})=\dim_{\overline{\Lambda}}\chi_M - \langle \chi_M, 1\rangle$. It follows that
\begin{equation}
	\label{Comparaison Kato AbbesSaito 5}
\varpi\circ\varphi(\sw_{\zeta}(\chi_M))=(c, \ord_{\p}(\Delta')-m\ord_{\p}({\rm d}\overline{a})).
\end{equation}
Therefore, from the definition of ${\rm KCC}_{\zeta}(\chi_M)$, we deduce that 
\begin{equation}
	\label{Comparaison Kato AbbesSaito 6}
\beta\circ\varpi\circ\varphi(\sw_{\zeta}(\chi_M))=-\ord_{\p}({\rm KCC}_{\zeta}(\chi_M)).
\end{equation}
We also have, just by definition,
\begin{equation}
	\label{Comparaison Kato AbbesSaito 7}
\sw_{\zeta}(\chi_M)=\sum_{\sigma\in G} s_G(\sigma) \otimes {\rm tr}_{\chi_M}(\sigma) + m\sum_{r\in \mathbb{F}_p^{\times}\subseteq \kappa(\p)^{\times}} [r]\otimes \zeta^r.
\end{equation}
Now, we know from that $\lvert G \lvert [{\rm d}\overline{h}]=[{\rm d}\overline{a}]$ in $S_{L/K}$; so $\varpi\circ\varphi ([{\rm d}\overline{h}])=\ord_{\p}({\rm d}\overline{a})\varepsilon/\lvert G \lvert$, where $\varepsilon=(0, 1)$ (\textit{cf.} \ref{hypotheses}). 
For $\sigma\in G -\{1\}$, we also clearly have $\varpi\circ\varphi([h-\sigma(h)])=i_G(\sigma)$ \eqref{iG}. It thus follows from \eqref{Artin and Swan characters 1} and \eqref{Swan conductor with differential values 1} that
\begin{equation}
	\label{Comparaison Kato AbbesSaito 8}
\varpi\circ\varphi(s_G(\sigma))=\frac{(\ord_{\p}({\rm d}\overline{a})-1)\varepsilon}{\lvert G\lvert} - j_{G, \varepsilon}(\sigma)=\frac{(\ord_{\p}({\rm d}\overline{a})-1)\varepsilon}{\lvert G\lvert} + \sw_G(\sigma),
\end{equation} 
and we deduce from \eqref{Artin and Swan characters 2}, \eqref{Swan conductor with differential values 2}) and \eqref{Comparaison Kato AbbesSaito 8} that
\begin{equation}
	\label{Comparaison Kato AbbesSaito 9}
\varpi\circ\varphi(s_G(1))=\frac{(\ord_{\p}({\rm d}\overline{a})-1)\varepsilon}{\lvert G\lvert} + \sw_G(1) - (\ord_{\p}({\rm d}\overline{a})-1)\varepsilon.
\end{equation}
We can also compute
\begin{equation}
	\label{Comparaison Kato AbbesSaito 10}
\varpi\circ\varphi\otimes {\rm Id}_{C}\left(\sum_{r\in \mathbb{F}_p^{\times}\subsetneq \kappa(\p)^{\times}} [r]\otimes \zeta^r\right)=(\sum_r \ord_{\p}(r), 0)=(0, 0),
\end{equation}
where $\varpi\circ\varphi\otimes {\rm Id}_{C}(x\otimes y)=y\varpi\circ\varphi(x)$ for $x\in S_{L/K}$ and $y\in C$.
Hence, combining \eqref{Comparaison Kato AbbesSaito 7}, \eqref{Comparaison Kato AbbesSaito 8}, \eqref{Comparaison Kato AbbesSaito 9} and \eqref{Comparaison Kato AbbesSaito 10} with \eqref{Pairing 1} and \eqref{Concordance 1}, we obtain
\begin{equation}
	\label{Comparaison Kato AbbesSaito 11}
\begin{split}
\varpi\circ\varphi(\sw_{\zeta}(\chi_M)) & =\sw_G (M) + (\ord_{\p}({\rm d}\overline{a})-1)\langle \chi, 1\rangle\varepsilon - (\ord_{\p}({\rm d}\overline{a})-1)\chi_M(1)\varepsilon\\
& =\sw_G (M) -  m(\ord_{\p}({\rm d}\overline{a})-1)\varepsilon.
\end{split}
\end{equation}
Now we use the fact that $\overline{a}$ is a uniformizer of the valuation ring $V/\p$ of $\kappa(\p)$ and thus $\ord_{\p}({\rm d}\overline{a})=0$ to get the simplified identity
\begin{equation}
	\label{Comparaison Kato AbbesSaito 12}
\varpi\circ\varphi(\sw_{\zeta}(\chi_M))=\sw_G (M) + m\varepsilon.
\end{equation}
Applying the projection $\beta$ to this identity, we deduce from \eqref{Pairing SerreKato 1}, \eqref{CC-KCC-Z2-Valuations 1} and \eqref{Comparaison Kato AbbesSaito 6} that indeed $\lvert G\lvert \sw_G^{\beta}(M)=-\ord_{\p}({\rm CC}_{\psi}(M)) - m$. Applying the projection $\alpha$ to \eqref{Comparaison Kato AbbesSaito 12} and using \eqref{Comparaison Kato AbbesSaito 5} (and that $a_G^{\alpha}=\sw_G^{\alpha}$ (\ref{ArtinSwan})) yields
\begin{equation}
	\label{Comparaison Kato AbbesSaito 13}
\lvert G\lvert a_G^{\alpha}(M)=\lvert G\lvert \sw_G^{\alpha}(M)=c.
\end{equation}
Finally, as $W_{\q}/V_{\p}$ is a monogenic extension, we see from \cite[6.7]{A.S.1} that its logarithmic ramification is bounded by a rational number $r\geq 0$ if and only if $\alpha(i_G(\sigma)) \geq r$. (Note that, in \textit{loc. cit.}, Proposition 6.7 holds for all monogenic separable extensions and thus is validly applied here). Therefore, the logarithmic filtration of $G$ (\ref{The logarithmicRamifFilt}), relative to $W_{\q}/V_{\p}$, coincide with the filtration defined by $\alpha\circ i_G$. Hence, we have
\begin{equation}
	\label{Comparaison Kato AbbesSaito 14}
\sw_G^{\rm AS}(M)=\alpha(\sw_G(M))=\lvert G \lvert a_G^{\alpha}(M)=c,
\end{equation}
which finishes the proof.
\end{proof}

\section{Proof of Theorem \ref{Theoreme principal}.}
\label{Proof}

\subsection{} \label{Notations Proof}
Let $K$ be a complete discrete valuation field, $\mathcal{O}_K$ its valuation ring, $\m_K$ its maximal ideal, $k$ its residue field, assumed to be algebraically closed of characteristic $p>0$, and $\pi$ a uniformizer of $\O_K$.
Let also $\overline{\Lambda}$ be a finite field of characteristic $\ell \neq p$ and fix a nontrivial character $\psi : \mathbb{F}_p\to \overline{\Lambda}^{\times}$.

\subsection{} \label{FaisceauLisseSurD}
Let $D$ be the rigid unit disc over $K$ and $\F$ a \textit{lisse} étale sheaf of $\overline{\Lambda}$-modules on $D$. Let $\overline{0}\to D$ be a geometric point above the origin $0$ of $D$. By \cite[2.10]{deJong}, the datum $\F$ is equivalent to the data of a finite Galois étale connected cover $f: X\to D$ and a finite dimensional continuous $\overline{\Lambda}$-representation $\rho_{\F}$ of $\pi_1^{\textrm{ét}}(D, \overline{0})$ which factors through the quotient $G={\rm Aut}(X/D)$ of $\pi_1^{\textrm{ét}}(D, \overline{0})$. Let $\chi_{\F}$ be the image of $\rho_{\F}$ in the Grothendieck group $R_{\overline{\Lambda}}(G)$. Let $t\in \Q_{\geq 0}$, $\p^{(t)}$ the generic point of the special fiber of the normalized integral model the sub-disc $D^{(t)}$ of $D$ and $\tau=(\overline{x}_{\tau}, \p_{\tau})$ an element of the set $S_f^{(t)}$ (notation of \ref{NotationsVariation}) associated to the normalized integral model of $f^{(t)}: X^{(t)}\to D^{(t)}$ defined over a finite extension $K'$ of $K$ which is $t$-admissible for $f$ (\ref{radmissible}). (Recall that $S_f^{(t)}$ is independent of the choice of such a $K'$ (\ref{RelevementsPoints geometriques}).)  The group $G$ acts transitively on $S_f^{(t)}$, and any element $\tau\in S_f^{(t)}$ defines an monogenic integral extension of henselian $\Z^2$-valuation rings $V_t^h(\tau)/V_t^h$ whose induced extension of fields of fractions $\K_{t, \tau}^h/\K_t^h$ is Galois of group $G_{t, \tau}$, the stabilizer of $\tau$ under the action of $G$ (\ref{Action Transitive}). We complete this extension, which puts as in the situation of \ref{CC-KCC-Z2-Valuations}.
As $\lvert G_{t, \tau}\lvert=\lvert G\lvert/\lvert S_f^{(t)}\lvert$ (\ref{Action Transitive}), we deduce from \ref{Comparaison Kato AbbesSaito}, \ref{Normalized Conductors} and Frobenius reciprocity that (notation of \ref{DisqueFaisceau})
\begin{equation}
	\label{FaisceauLisseSurD 1}
\sw_{G_{t, \tau}}^{\rm AS}(\rho_{\F}\lvert G_{t, \tau})=\langle \widetilde{a}_f^{\alpha}(t), \chi_{\F}\rangle,
\end{equation}
\begin{equation}
	\label{FaisceauLisseSurD 2}
-\ord_{\overline{\p}^{(t)}}({\rm CC}_{\psi}(\rho_{\F}\lvert G_{t, \tau})) - \dim_{\overline{\Lambda}}\left(\rho_{\F}\lvert G_{t, \tau}/(\rho_{\F}\lvert G_{t, \tau})^{(0)}\right)=\langle\widetilde{\sw}_f^{\beta}(t), \chi_{\F}\rangle.
\end{equation}
It follows that $\sw_{G_{t, \tau}}^{\rm AS}(\rho_{\F}\lvert G_{t, \tau})$ is independent of the choice of both the $t$-admissible extension $K'$ and $\tau\in S_f^{(t)}$. Since $G_{t, \tau}$ and its wild inertia subgroup $P_{t, \tau}$, with respect to the extension of discrete valuation rings induced by $V_t^h(\tau)/V_t^h$ at the height $1$ prime ideals, are independent of the choice of $K'$ (see \ref{Characters} and \ref{ResidueGaloisGroupBaseChange}), so are $\rho_{\F}\lvert G_{t, \tau}$ and its tame part $(\rho_{\F}\lvert G_{t, \tau})^{(0)}=(\rho_{\F}\lvert G_{t, \tau})^{P_{t, \tau}}$. As the $G_{t, \tau}$ (resp. $P_{t, \tau}$), for all $\tau\in S_f^{(t)}$, are conjugate, $\rho_{\F}\lvert G_{t, \tau}$ and $(\rho_{\F}\lvert G_{t, \tau})^{(0)}$ are also independent of the choice of $\tau\in S_f^{(t)}$.
Hence, by \eqref{FaisceauLisseSurD 2}, $-\ord_{\overline{\p}^{(t)}}({\rm CC}_{\psi}(\rho_{\F}\lvert G_{t, \tau}))$ is also independent of the choice of both $K'$ and $\tau$.
Finally, we remark also that $\ord_{\overline{\p}^{(t)}}({\rm CC}_{\psi}(\rho_{\F}\lvert G_{t, \tau}))$ is independent of both the chosen uniformizer $\pi$ (by \eqref{Characters} and \eqref{FaisceauLisseSurD 2}) and the nontrivial character $\psi$.

\begin{defi}
	\label{NormalizedSwanCC}
We keep the notation and assumptions of {\rm \ref{FaisceauLisseSurD}} above. We define the normalized logarithmic Swan conductor of $\F$ at $t$ by
\begin{equation}
	\label{NormalizedSwanCC 1}
\sw_{\rm AS}(\F, t)=\sw_{G_{t, \tau}}^{\rm AS}(\rho_{\F}\lvert G_{t, \tau}).
\end{equation}
We define the normalized order of the characteristic cycle of $\F$ at $t$ by
\begin{equation}
	\label{NormalizedSwanCC 2}
\varphi_s(\F, t)=-\ord_{\overline{\p}^{(t)}}({\rm CC}_{\psi}(\rho_{\F}\lvert G_{t, \tau})) - \dim_{\overline{\Lambda}}\left(\rho_{\F}\lvert G_{t, \tau}/(\rho_{\F}\lvert G_{t, \tau})^{(0)}\right).
\end{equation}
\end{defi}

\begin{thm}[{Theorem \ref{Theoreme principal}}]
	\label{Main theorem}
We keep the notation and assumptions of {\rm \ref{FaisceauLisseSurD}} and {\rm \ref{NormalizedSwanCC}}. Then, the function $\sw_{\rm AS}(\F, \cdot): \Q_{\geq 0}\to \Q$ is continuous and piecewise linear, with finitely many slopes which are all integers. Its right derivative is the function $\varphi_s(\F, \cdot): \Q_{\geq 0}\to \Q$ which is locally constant.
\end{thm}

\begin{proof}
This follows from \eqref{FaisceauLisseSurD 1}, \eqref{FaisceauLisseSurD 2} and Corollary \ref{VariationCoeffsFinis}.
\end{proof}


\begin{thebibliography}{99} 
\bibitem[Abb10]{EGR} {\sc A. Abbes}, \'Eléments de Géométrie Rigide, I Construction et étude géométrique des espaces rigides, Birkhäuser, \textit{Progress in Mathematics} {\bf 286} (2010).

\bibitem[AS02]{A.S.1} {\sc A. Abbes, T. Saito}, Ramification of locals fields with imperfect residue fields, Amer. J. Math. {\bf 124} (2002), 879-920.

\bibitem[AS03]{A.S.2} {\sc A. Abbes, T. Saito}, Ramification of locals fields with imperfect residue fields II, Doc. Math. Extra Volume: Kazuya Kato's Fiftieth Birthday (2003), 5-72.

\bibitem[AS09]{A.S.4} {\sc A. Abbes, T. Saito}, Analyse micro-locale l-adique en caracteristique p> 0: Le cas d'un trait, Publ. RIMS, Kyoto Univ. {\bf 45} (2009), 25-74.

\bibitem[AS11]{A.S.3} {\sc A. Abbes, T. Saito}, Ramification and cleanliness, Tohoku Math. Journal. Centennial issue, {\bf 63} (2011), no. 4, 775-853.

\bibitem[Bal10]{Balda} {\sc F. Baldassarri}, Continuity of the radius of convergence of differential equations on $p$-adic analytic curves, Invent. Math. {\bf 182} (2010), 513-584.

\bibitem[Bei14]{Beilinson} {\sc A. Beilinson}, Constructible sheaves are holonomic, Selecta Mathematica {\bf 22} (2016), no. 4, 1797-1819.

\bibitem[BGR84]{BGR} {\sc S. Bosch, U. Güntzer, R. Remmert}, Non-Archimedean analysis, Springer-Verlag, {\bf 261} (1984).

\bibitem[BL85]{BoschLutke} {\sc S. Bosch, W. Lütkebohmert}, Stable reduction and uniformization of abelian varieties I, Mathematische Annalen, {\bf 270} (1985), no. 3, 349--379.

\bibitem[BLR95]{BLR}{\sc S. Bosch, W. Lütkebohmert, M. Raynaud}, Formal and rigid geometry IV. The Reduced Fiber Theorem, Invent. Math, {\bf 119} (1995), 361--398.

\bibitem[Bou06]{Bourbaki1} {\sc N. Bourbaki}, Algèbre Commutative, Springer-Verlag, (2006).

\bibitem[Bou07]{Bourbaki2} {\sc N. Bourbaki}, Algèbre, Springer-Verlag, (2007).

\bibitem[Con99]{Conrad} {\sc B. Conrad}, Irreducible components of rigid spaces, Annales de l'institut Fourier {\bf 49}, (1999), no. 2, 473-541.

\bibitem[deJ95]{deJong} {\sc A. J. de Jong}, {\'E}tale fundamental groups of non-Archimedean analytic spaces, Compositio math. {\bf 97} (1995), no. 1-2, 89-118.

\bibitem[Dub84]{Dubson} {\sc A. Dubson}, Formule pour l'indice des complexes constructibles
et D-modules holonomes, C. R. Acad. Sci. {\bf 298}, Série A, (1984), no. 6, 113-114.

\bibitem[EGA I]{EGA.I} {\sc A. Grothendieck, J.A. Dieudonné}, \'Eléments de Géométrie Algébrique, I Le langage des schémas, Pub. Math. IH\'ES {\bf 4} (1960),

\bibitem[EGA III]{EGA.III} {\sc A. Grothendieck, J.A. Dieudonné}, \'Eléments de Géométrie Algébrique, III \'Etude cohomologique des faisceaux cohérents, Pub. Math. IH\'ES {\bf 11} (1961), {\bf 17} (1963).

\bibitem[EGA IV]{EGA.IV} {\sc A. Grothendieck, J.A. Dieudonné}, \'Eléments de Géométrie Algébrique, IV \'Etude locale des schémas et
des morphismes de schémas, Pub. Math. IH\'ES {\bf 20} (1964), {\bf 24} (1965), {\bf 28} (1966), {\bf 32} (1967).

\bibitem[SGA 2]{SGA2} {\sc A. Grothendieck, Michèle Raynaud}, \textit{Séminaire de Géométrie Algébrique du Bois-Marie}, Cohomologie locale des faisceaux cohérents et théorèmes de Lefschetz locaux et globaux (SGA 2), Société Mathématique de France {\bf 4} (2005).

\bibitem[SGA 4]{SGA4} {\sc M. Artin, A. Grothendieck, J.-L. Verdier}, \textit{Séminaire de Géométrie Algébrique du Bois-Marie}, Théorie des topos et cohomologie étale des schémas (SGA 4), Springer-Verlag, Lect. Notes in Math. {\bf 269} (1972), {\bf 270} (1972), {\bf 305} (1973).

\bibitem[SGA 7]{SGA7} {\sc P. Deligne, N. M. Katz}, \textit{Séminaire de Géométrie Algébrique du Bois-Marie}, Groupes de Monodromie en Géométrie Algébrique (SGA 7), Springer-Verlag, Lect. Notes in Math. {\bf 288} (1972), {\bf 340} (1973).

\bibitem[End72]{Endler} {\sc O. Endler}, Valuation theory, Springer, {\bf 5} (1972).

\bibitem[Epp73]{Epp} {\sc H. Epp}, Eliminating wild ramification, Invent. Math. {\bf 19} (1973), 235-249.

\bibitem[Fu11]{Fu}{\sc L. Fu}, \'Etale cohomology theory, World Scientific, Nankai Tracts in Mathematics, {\bf 169} (2011).

\bibitem[GO08]{GO} {\sc O. Gabber, F. Orgogozo}, Sur la $p$-dimension des corps, Invent. Math. {\bf 174} (2008), 47-80.

\bibitem[Hen00]{Henrio} {\sc Y. Henrio}, Disques et couronnes ultramétriques, In \textit{Coubes semi-stables et groupe fondamental en Géométrie algébrique}, Birkhäuser, Progress in Mathematics, {\bf 187} (2000), 21-32.

\bibitem[Hu15]{Hu1} {\sc H. Hu}, Ramification and nearby cycles for $\ell$-adic sheaves on relative curves, Tohoku Math. Journal {\bf 67}, (2015), no. 2, 153-194.

\bibitem[ILO14]{TdeGabber} {\sc L. Illusie, Y. Laszlo, F. Orgogozo}, Travaux de Gabber sur l’uniformisation locale et la cohomologie étale des schémas quasi-excellents \textit{Séminaire à l’\'Ecole polytechnique 2006-2008}, Astérisque {\bf 363-364}, (2014).

\bibitem[Kas85]{Kashiwara} {\sc M. Kashiwara}, Index theorem for constructible sheaves, In \textit{Systèmes différentiels et singularités}, Astérisque, {\bf 130}, (1985), no. 2, 193-209.

\bibitem[Kat87a]{K1}{\sc K. Kato}, Vanishing cycles, ramification of valuations and class field theory, Duke Math. J. {\bf 55}, (1987), no. 3, 629-659.

\bibitem[Kat87b]{K2}{\sc K. Kato}, Swan conductors with differential values, Adv. Stud. Pure Math. {\bf 12} (1987), 315-342.

\bibitem[Kat89]{K3} {\sc K. Kato}, Swan conductors for characters of degree one in the imperfect residue field case, Contemp. Math. {\bf 83} (1989), 101-131.

\bibitem[Ka88]{Katz} {\sc N. M. Katz}, Gauss Sums, Kloosterman Sums, and Monodromy Groups, Ann. of Math. Studies, Princeton university press {\bf 116} (1988).

\bibitem[Ked15]{Kedlaya} {\sc K. Kedlaya}, Local and global structure of connections on nonarchimedean curves, Compositio Math. {\bf 151} (2015), 1096-1156.

\bibitem[Lau81]{Laumon1} {\sc G. Laumon}, Semi-continuité du conducteur de Swan (d'après Deligne), In \textit{Caractéistique d'Euler-Poincaré}, Astérisque, {\bf 82-83} (1981), 173-219.

\bibitem[Lau83]{Laumon2} {\sc G. Laumon}, Caractéristique d'Euler-Poincaré des faisceaux constructibles sur une surface, In \textit{Analyse et topologie sur les espaces singuliers (II-III)}, Astérisque {\bf 101-102} (1983), 193-207.

\bibitem[Lüt93]{Lutke} {\sc W. Lütkebohmert}, Riemann’s existence problem for a p-adic field, Invent. Math. {\bf 111} (1993), 309-330.

\bibitem[PP15]{PP} {\sc J. Poineau, A. Pulita}, The convergence Newton polygon of a p-adic differential equation II : Continuity and finiteness on Berkovich curves, Acta Mathematica {\bf 214}, (2015), no. 2, 357-393.

\bibitem[Pul15]{Pulita} {\sc A. Pulita}, The convergence Newton polygon of a $p$-adic differential equation I : Affinoid domains of the Berkovich affine line, Acta Mathematica {\bf 214}, (2015), no. 2, 307-355.

\bibitem[Ram05]{Ramero} {\sc L. Ramero}, Local monodromy in non-archimedean analytic geometry, Pub. Math. IH\'ES {\bf 102} (2005), 167-280.

\bibitem[Ray70]{Raynaud-ALH} {\sc M. Raynaud}, Anneaux Locaux Henséliens, Lecture Notes in Mathematics, Springer-Verlag {\bf 169} (1970).

\bibitem[Ray94]{Raynaud-Abh} {\sc M. Raynaud}, Revêtements de la droite affine en caractéristique $p> 0$ et conjecture d'Abhyankar, Invent. Math. {\bf 116}, (1994), no. 1, 425-462.

\bibitem[Sai09]{Saito1} {\sc T. Saito}, Wild ramification and the characterictic cycle of an $\ell$-adic sheaf, J. Inst. Math. Jussieu {\bf 8} (2009) 769-829.

\bibitem[Sai12]{Saito2} {\sc T. Saito}, Ramification of locals fields with imperfect residue fields III, Math. Ann. {\bf 352}, (2012), no. 3, 567-580.

\bibitem[Sai17]{Saito3} {\sc T. Saito}, The characterictic cycle and the singular support of a constructible sheaf, Invent. Math. {\bf 207} (2017), 597-695.

\bibitem[Sai20]{Saito4} {\sc T. Saito}, Graded quotients of ramification groups of local fields with imperfect residue fields, arXiv preprint arXiv:2004.03768 (2020).

\bibitem[Ser68]{Serre1} {\sc J.-P. Serre}, Corps Locaux, Hermann, Paris (1968).

\bibitem[Ser97]{Serre3} {\sc J.-P. Serre}, Algèbre locale, multiplicités: cours au Collège de France, 1957-1958, Springer {\bf 11} (1997).

\bibitem[Ser98]{Serre2} {\sc J.-P. Serre}, Representations linéaires des groupes finis, Hermann, Paris (1998).

\bibitem[Sta19]{Stacks} {\sc The Stacks Project Authors}, Stacks Project, \url{https://stacks.math.columbia.edu} (2019).

\bibitem[Xia10]{Xiao1} {\sc L. Xiao}, On ramification filtrations and p-adic differential equations, I: equal characterisitc case, Alg. and Num. Theory 4 (2010), no. 8, 969-1027.

\bibitem[Xiao12]{Xiao2} {\sc L. Xiao}, On ramification filtrations and p-adic differential equations, II: mixed characteristic case, Compositio Math. 148 (2012), no. 2, 415-463.

\end{thebibliography}
\end{document}